\documentclass[a4paper, 11pt]{amsart}

\usepackage{lmodern}
\usepackage[english]{babel}
\usepackage[utf8]{inputenc}
\usepackage{amssymb}
\usepackage{amsfonts}
\usepackage{amsthm}
\usepackage{amsmath}

\usepackage{mathrsfs, soul, bm}

\usepackage{hyperref}
\usepackage{enumitem}
\usepackage{bm, float}
\usepackage{tikz-qtree}

\usepackage{graphicx}

\def\namedlabel#1#2{\begingroup
  #2 
  \def\@currentlabel{#2} 
  \phantomsection\label{#1}\endgroup
}

\usepackage[margin=1.1in]{geometry}
\setlist[itemize]{leftmargin=*}

\newcommand{\image}[2]{\includegraphics[width=#1\linewidth]{#2.png}}

\usepackage{subcaption}
\newcommand{\cotgraph}[2]{
}

\usepackage[normalem]{ulem}
\usepackage{dsfont}

\theoremstyle{plain}
\newtheorem{theorem}{Theorem}[section]
\newtheorem{corollary}[theorem]{Corollary}
\newtheorem{definition}[theorem]{Definition}
\newtheorem{lemma}[theorem]{Lemma}
\newtheorem{proposition}[theorem]{Proposition}

\theoremstyle{definition}

\theoremstyle{remark}
\newtheorem{remark}[theorem]{Remark}

\numberwithin{equation}{section}

\newcommand\numberthis{\stepcounter{equation}\tag{\theequation}}

\renewcommand{\setminus}{\smallsetminus}
\newcommand{\seq}{\subseteq}
\newcommand{\ssum}[1]{\sum_{\substack{#1}}}

\newcommand{\e}{{\rm e}}
\newcommand{\df}{\mathop{}\!\mathrm{d}}

\newcommand{\uu}{{u}}

\let\C\relax
\newcommand{\C}{{\mathbb C}}

\newcommand{\N}{{\mathbb N}}
\newcommand{\Q}{{\mathbb Q}}
\renewcommand{\P}{{\mathbb P}}
\newcommand{\R}{{\mathbb R}}

\newcommand{\Z}{{\mathbb Z}}
\newcommand{\1}{{\mathbf 1}}

\newcommand{\bb}{{\bm b}}

\newcommand{\CC}{{\mathcal C}}
\newcommand{\CR}{{\mathcal R}}
\newcommand{\CG}{{\mathcal G}}

\newcommand{\gA}{{\mathfrak A}}

\newcommand{\gS}{{\mathfrak S}}
\newcommand{\gT}{{\mathfrak T}}

\newcommand{\vth}{\vartheta}

\newcommand{\I}{{\mathcal I}}

\newcommand{\cS}{{\mathscr S}}

\newcommand{\reta}{{R}}
\newcommand{\pphi}{{\beta}}

\DeclareMathOperator{\den}{Den}
\DeclareMathOperator{\num}{Num}

\DeclareMathOperator{\sgn}{sgn}
\DeclareMathOperator{\meas}{meas}

\DeclareMathOperator{\SL}{SL}

\renewcommand{\tilde}{\widetilde}
\renewcommand{\bar}{\overline}
\renewcommand{\hat}{\widehat}

\newcommand{\ct}{{\tilde c}}
\newcommand{\ctt}{{\breve c}}

\newcommand{\extpos}[1]{#1^\triangleright}
\newcommand{\extneg}[1]{#1^\triangleleft}
\newcommand{\fep}{\extpos{f}}
\newcommand{\fem}{\extneg{f}}

\newcommand{\abs}[1]{{\left| {#1} \right|}}

\newcommand{\floor}[1]{{\left\lfloor {#1} \right\rfloor}}

\renewcommand{\mod}[1]{\ ({\rm mod\ }#1)}

\renewcommand\Re{\operatorname{Re}}
\renewcommand\Im{\operatorname{Im}}

\newcommand{\eps}{\varepsilon}

\numberwithin{equation}{section}

\makeatletter
\@namedef{subjclassname@2020}{\textup{2020} Mathematics Subject Classification}
\makeatother

\title{On quantum modular forms of non-zero weights}

\date{\today}

\author{S. Bettin}
\address{SB: Dipartimento di Matematica, Universit\`a di Genova, via Dodecaneso 35, 16146 Genova, Italy}
\email{bettin@dima.unige.it}

\author{S. Drappeau}
\address{SD: Aix Marseille Universit\'e, CNRS, I2M UMR 7373, 13453 Marseille, France}
\email{sary-aurelien.drappeau@univ-amu.fr}



\thanks{
  This work has benefitted from support from Aix-Marseille Université FIR Invités; from INdAM group GNAMPA; and from FWF-ANR project Arithrand: FWF: I 4945-N and ANR-20-CE91-0006.  The work of the first author is partially supported by PRIN 2017 ``Geometric, algebraic and analytic methods in arithmetic". 
  The authors thank DIMA (Univ.\ Genova) and I2M (Univ.\ Aix-Marseille), where this work was carried out, for their hospitality.
}

\begin{document}

\begin{abstract}
  We study functions $f$ on~$\Q$ which statisfy a ``quantum modularity'' relation of the shape
  $$ f(x+1)=f(x), \qquad f(x) - \abs{x}^{-k} f(-1/x) = h(x) $$
  where~$h:\R_{\neq 0} \to \C$ is a function satisfying various regularity conditions. We study the case~$\Re(k)\neq 0$.
  We prove the existence of a limiting function~$f^*$ which extends continuously~$f$ to~$\R$ in some sense. This means in particular that in the $\Re(k)\neq0$ case the quantum modular form itself has to have at least a certain level of regularity.
  
  We deduce that the values~$\{f(a/q), 1\leq a<q, (a, q)=1\}$, appropriately normalized, tend to equidistribute along the graph of~$f^*$, and we prove that under natural hypotheses the limiting measure is diffuse.

  We apply these results to obtain limiting distributions of values and continuity results for several arithmetic functions known to satisfy the above quantum modularity: higher weight modular symbols associated to holomorphic cusp forms; Eichler integral associated to Maass forms; a function of Kontsevich and Zagier related to the Dedekind~$\eta$-function; and generalized cotangent sums.
\end{abstract}

\maketitle

\section{Introduction}

In~\cite{Zagier2010}, Zagier introduced the concept of quantum modular forms (QMF). Given a Fuchsian cofinite subgroup $\Gamma$ of $\SL(2, \Z)$ whose set of cusps~$C(\Gamma)\subset \P^1(\Q)$ is non-empty, and~$k\in\C$, QMF are defined as functions 
$$f : C(\Gamma)\setminus S \to \C,$$
for some finite set $S$, satisfying a form of modularity, in the purposely vague sense that for any $\gamma=(\begin{smallmatrix}a & b\\ c&d\end{smallmatrix})\in\Gamma$ the \emph{period function}
\begin{align}\label{mcg1}
  h_\gamma(x):= f(x) - |cx+d|^{-k}f\bigg(\frac{ax +b}{cx+d}\bigg),\qquad x\in C(\Gamma)\setminus (S\cup \gamma^{-1}S)
\end{align}
has some regularity property.

Numerous examples of quantum modular forms are known, in various contexts, and we refer in particular to~\cite{Zagier2010,BringmannEtAl2015, NgoRhoades2017,KimLimEtAl2016, BruggemanEtAl2015,JaffardMartin2018}. More references are listed in the introduction of~\cite{BettinDrappeau}.

In this paper we focus on QMF for the full modular group $\Gamma = \SL(2, \Z)$, so that~$C(\Gamma) = \P^1(\Q)$, and which are periodic (i.e. $h_{U}=0$ for $U=(\begin{smallmatrix}1 & 1\\ 0&1\end{smallmatrix})$)\footnote{See~\cite[Foonote~2]{BettinDrappeau} for a possible approach to the non-periodic case.}.
By composition, it is sufficient to consider~\eqref{mcg1} for the second generator $\gamma=(\begin{smallmatrix}0 & -1\\ 1&0\end{smallmatrix})$ of~$\SL(2, \Z)$, so that in order to prove that~$f$ is a QMF, one only needs to verify that
\begin{align}\label{mcg}
  h(x):= f(x) -  |x|^{-k}f(-1/x),\qquad x\in\Q\setminus\{0\} 
\end{align}
has some regularity property.

The cases of $\Re(k)=0$ and $\Re(k)\neq0$ are different in nature. In the case~$k=0$, by iterating the relation~\eqref{mcg} and using periodicity, we can express $f$ as a Birkhoff sum of $h$ evaluated along orbits under the Gauss map, see~\cite[eq.~(3.1)]{BettinDrappeau}. Using this observation, in~\cite{BettinDrappeau} it was then showed that for a large class of functions $h$, the multi-sets $\{f(x)\mid x\in\Q\cap[0,1),\, \den(x)\leq Q\}$ become asymptotically distributed, as~$Q\to\infty$, according to a stable law, which is in fact a normal law if $h$ is of moderate growth at $0$.\footnote{The same method yields an analogous result for $\{f(x)\mid x\in\Q\cap[a,b),\, \textnormal{Den}(x)\leq Q\}$ for any $0\leq a<b\leq1$.} As a consequence one has that, in general, $f$ is nowhere continuous according to the real topology, nor can~$f$ be extended by continuity to any point outside of $\Q$. For the same reasons, we expect a similar phenomenon to occur whenever~$\Re(k)=0$.

The purpose of the present paper is to study the case when~$\Re(k)\neq 0$.

\subsection{Weights with negative real parts}

If $\Re(k)<0$ and $h$ is continuous on $[-1,1]\setminus \{0\}$ with finite right and left limits $h(0^\pm)$ at $0$, then we show that in fact $f$ can always be extended by continuity to a bounded function $\fem:\R\to\C$.
Moreover, we prove that $\fem$ is continuous on $\R\setminus \Q$ and is continuous on the whole real line if $h(0\pm)=f(0)$, a condition that could be morally interpreted as saying that~\eqref{mcg} holds also at ``$0^{\pm}$''.
More precisely, the following holds.

\begin{theorem}\label{ds}
  Let $\Re(k)<0$ and let $f:\Q\to\C$ be a $1$-periodic function satisfying~\eqref{mcg} for a function $h:\R\setminus \{0\}\to \C$ which is continuous on $[-1,1]\setminus \{0\}$ with finite left and right limits $h(0^\pm)$ at $0$. Then the function 
  \begin{align}\label{dfd}
    \fem(x):=
    \begin{cases}
      f(x) & \text{if }x\in\Q\\
      {\displaystyle \lim_{\Q\ni y\to x}f(y)} & \text{if }x\notin\Q
    \end{cases}
  \end{align}
  is  defined for all $x\in\R$ and is continuous on $\R\setminus\Q$. Moreover, for any rational $x=\frac aq$ in reduced form, one has 
  $\fem(y)\to f(x) + q^k(h(0^\pm)-f(0))$ as $y\to x^\pm$. In particular, $\fem$ is continuous on $\R$ if and only if $h(0^\pm)=f(0)$.
  Furthermore, in this case, if $h\in\mathcal C^{m}([-1,1],\C)$ with~$0\leq m < \abs{\Re(k)}/2$, then $\fem\in\mathcal C^{m}(\R,\C)$.
\end{theorem}

If $h$ is continuous at $0$ but $h(0)\neq f(0)$, we can modify $f$ by letting $\tilde f(x):=f(x)-\den(x)^k(h(0)-f(0))$. The map~$\tilde f$ is a QMF with~$h$ as its period function, and~$\tilde f(0)=h(0)$.

As we will see below, in some of the examples of QMF that we will consider, the map~$f$ actually admits an expression as an absolutely converging Fourier series. Theorem~\ref{ds} gives a proof of the continuity and differentiability of $f$ extended to~$\R$ which does not rely on an \emph{a priori} knowledge of a Fourier expansion for~$f$. From this point of view, our methodology bears similarity to the works~\cite{BalazardMartin2019,JaffardMartin2018}. These works are concerned with the Brjuno function, which is close to being a QMF of weight~$-1$, see~\cite[eq.~(4)]{BalazardMartin2019}.

Moreover, Theorem~\ref{ds} shows that in fact one cannot expect the period function $h$ to be continuous if $f$ doesn't have some continuity property to start with (thus the truly interesting cases arise when $h$ is differentiable at least $\lceil -\Re(k)/2\rceil $ times). 

The existence of the function~$\fem$ in~\eqref{dfd} can actually be proved under much weaker hypotheses, which we explain in what follows. This will be required later on, when we will study cotangent sums.

\begin{theorem}\label{thm:kneg-fdagger-ext}
  Suppose that~$\Re(k)<0$ and let~$\delta>0$. Suppose that~$f$ satisfies
  \begin{equation}
    f(x) = f(x+1) \qquad (x \in \Q\setminus[-1, 0]),\label{eq:f-weak-per}
  \end{equation}
  in other words that~$f$ is~$1$-periodic separately on~$(-\infty, 0)$ and~$(0, \infty)$, and that the equation~\eqref{mcg} holds for~$x\in[-1, 1]\setminus\{0\}$, for a function $h$ satisfying
  \begin{align}\label{coh}
    \sup_{|x-y|<\e^{-\eps^{\delta-1}},\atop |x|>2\eps} |h(x)-h(y)|=o(1)\quad \text{as }\eps\to0^+,\qquad h(x)=O(\e^{|x|^{\delta-1}})\ \forall x\in[-1,1]\setminus\{0\}.
  \end{align}
  Then there is a set~$X \subset\R\setminus\Q$ of full measure, such that the limit
  \begin{equation}\label{dfd-X}
    \fem(x) = \lim_{j\to\infty} f(x_j)
  \end{equation}
  exists for~$x\in X$, where~$(x_j) = ([a_0(x); a_1(x), \dotsc, a_j(x)])_j$ denotes the sequence of convergents of~$x$.
  This value coincides with the limit~\eqref{dfd} when~$h$ is bounded. In general, for any~$\eps>0$, there is a subset~$X_\eps\subset X$ invariant by~$x\mapsto x+1$ and with~$\meas(X_\eps \cap [0, 1]) \geq 1-\eps$, such that~$\fem|_{X_\eps}$ is continuous on~$X_\eps$ equipped with the restricted topology. In particular~$\fem$ is Lebesgue-measurable.
\end{theorem}

The set~$X$ in this statement does not depend on~$f$, $h$ or~$\delta$. It consists of numbers having mildly growing continued fraction coefficients, see Lemma~\ref{bll0} below.

The extension of the definition of $f(x)$ also allows us to show that $f$ has a limiting distribution when evaluated at reduced rationals $0\leq \frac aq<1$ with $q\to\infty$.

\begin{theorem}\label{dless0}
  Suppose that~$\Re(k)<0$ and let $f:\Q\to\C$ satisfy~\eqref{mcg} and~\eqref{eq:f-weak-per}, for a function $h:\R\setminus \{0\}\to \C$ satisfying~\eqref{coh}. Then the multiset
  $$ \Big\{ f(\tfrac aq) \colon 1\leq a \leq q, (a, q)=1\Big\} $$
  becomes distributed, as~$q\to \infty$, according to the push-forward~$ \fem_\ast(\df \nu)$ of the measure~$\nu$ given by the Lebesgue measure on~$[0, 1]$.

  If, moreover,~$h$ is real-analytic on~$(-1, 1)\setminus\{0\}$ and~$f$ is non-constant on~$\Q_{>0}$, then the measure~$\fem_\ast(\df\nu)$ is diffuse. 
\end{theorem}

In contrast with the case $k=0$ treated in~\cite{BaladiVallee2005, BettinDrappeau}, we remark here that we did not need to perform an additional average over~$q$ in order to obtain a limiting statement.

We recall that a measure is diffuse if it has no atoms. When~$\fem$ is real-valued, then~$\fem_\ast(\df \nu)$ is supported on~$\R$, and diffuseness is equivalent to the continuity of the associated cumulative distribution function. Under appropriate conditions, we are able to reach the stronger conclusion that for any non-zero linear form~$\phi:\C\to\R$, the measure~$(\phi\circ\fem)_\ast(\df\nu)$ on~$\R$ is diffuse\footnote{Such a form~$\phi$ is of course proportional to~$z\mapsto \Re(\e^{i\theta} z)$ for some~$\theta\in\R$. This applies in particular to~$\phi = \Re$ or~$\phi = \Im$.}. This is equivalent to the statement that the graph of~$\fem$ never remains on a given straight line for a positive amount of time: for any line~$D\subset\C$, $\meas((\fep)^{-1}(D)) = 0$. Also, the statement that $(\phi\circ\fem)_\ast(\df\nu)$ is diffuse means that its cumulative distribution function is continuous.

\subsection{Weights with positive real parts}

In the case $\Re(k)>0$, we find that iterating the reciprocity formula~\eqref{mcg} does not imply the continuity of $f$, even if $h$ is continuous. It is however still possible to extend naturally $f$ if one considers~$f(x)$, not as a function of $x=\frac aq$, but rather as a function of $\overline {x}:=\frac{\overline {a}_q}q$, where $\overline a_q\in (0,q]$ is the multiplicative inverse of $a\mod q$. 

\begin{theorem}\label{ds1}
  Let $\Re(k)>0$ and let $f$ be $1$-periodic and satisfy~\eqref{mcg} with $h(x):\R\setminus\{0\}\to\C$ satisfying $h(x)=O(|x|^{-\Re(k)})$ for $|x|\in(0,1]$.
  Then the function
  \begin{align}\label{fst}
    \fep(x):=
    \begin{cases}
      q^{-k}f(\overline x) & \text{if }x=\frac aq\in\Q\\
      {\displaystyle \lim_{\Q\ni y=\frac aq\to x}q^{-k}f(\overline y)} & \text{if }x\notin\Q
    \end{cases}
  \end{align}
  defines a continuous function of $x\in\R\setminus\Q$. Furthermore, if $h(x)=o(|x|^{-\Re(k)})$ as $x\to0$, then $\fem$ is continuous on $\R$.

  Finally, there exists a set~$X \subset\R$ of full measure, such that~$\fep$ is~$\alpha$-Hölder continuous at any point of $X$, for any~$\alpha<\frac12\Re(k)$. In particular, if~$\Re(k)>2$ then~$\fep$ has derivative zero almost everywhere.
\end{theorem}

Note that~$h$ is bounded on~$\R\setminus\{0\}$ if and only if it is bounded on~$[-1, 1]\setminus\{0\}$, because from the reciprocity relation~\eqref{mcg} we have~$h(x) = -\abs{x}^{-k} h(-1/x)$ for~$x\neq 0$.

Also in the case $\Re(k)>0$, the limit~\eqref{fst} makes sense under more general hypotheses, which we will require in some of our applications: it suffices that~$f$ satisfies~\eqref{eq:f-weak-per} instead of being periodic on~$\R$, and that $h$ be unbounded but satisfying $h(x)=O(\e^{-|x|^{\delta-1}})$ for $x\in[-1,1]\setminus\{0\}$ and some $\delta>0$. More precisely, the following holds.

\begin{theorem}\label{ds1-ext}
  Suppose that~$\Re(k)>0$ and let $f:\Q\to\C$ satisfy~\eqref{mcg} and~\eqref{eq:f-weak-per} for a function $h:\R\setminus \{0\}\to \C$ with~$h(x)\ll \e^{|x|^{-1+\delta}}$ for $x\in(-1,1)\setminus\{0\}$ and some $\delta>0$. Then the limit
  \begin{equation}\label{fst-ext}
    \fep(x) := \lim_{j\to\infty} q(x_{2j+1})^{-k} f(\overline{x_{2j+1}})
  \end{equation}
  exists on a subset~$X$ of~$\R$ of full measure, and coincides with the value~\eqref{fst} under the hypotheses of Theorem~\ref{ds1}.

\end{theorem}

\begin{remark}
  As in Theorem~\ref{thm:kneg-fdagger-ext}, we will show under the hypotheses stated above the existence of sets~$X_\eps\subset X$ with~$\meas(X_\eps)\geq 1-\eps$ such that~$\fep|_{X_\eps}$ is continuous with the restricted topology, and in particular~$\fep$ is also Lebesgue-measurable.
\end{remark}

Similarly to Theorem~\ref{thm:kneg-fdagger-ext}, we can then show that the values of $f$ at rationals converge to a limiting distributions.

\begin{theorem}\label{dless1}
  Under the notations and conditions of Theorem~\ref{ds1-ext}, the multisets
  $$ \Big\{q^{-k} f(\tfrac aq) \colon 1\leq a \leq q, (a, q)=1\Big\} $$
  become distributed, as~$q\to \infty$, according to~$ \fep_\ast(\df \nu)$.

  Moreover, if $h$ is not identically zero, and is either continuous on $[-1,1]$; or satisfies $h(x)\ll \e^{\abs{x}^{-1+\delta}}$ and $h(x) \sim c x^{-\lambda}$ as $x\to0^+$ for some~$\delta>0$,~$c\in\C\setminus\{0\}$ and~$\lambda\in\R_{>0}\setminus\{k\}$, then the measure $\fep_\ast(\df \nu)$ is diffuse.
\end{theorem}

\begin{remark}
  One can relax considerably the conditions required to ensure the diffuseness of the limiting measures. For example it is sufficient that $h$ is right continuous at $0$ with $h(0^+)\neq0$ or that $h(x)$ goes to $\infty$ as $x\to 0$ without staying too close to a multiple of $1-|x|^{-k}$.
  See Section~\ref{sec:cont-cumul-distr} for more details.

  Similarly as for~$\Re(k)<0$, under natural conditions which are however more involved, we show that the measure~$(\phi \circ \fep)_\ast(\df\nu)$ on~$\R$ is diffuse for any non-zero linear form~$\phi:\C\to\R$.
\end{remark}

\subsection{Applications}

The above theorems apply to many objects, some of which are described in the following corollaries.

\subsubsection{Eichler integrals of classical holomorphic forms}\label{sec:appl-perpol-holo}

The elementary-looking function
\begin{equation}
  A_{k,D}(x):=\ssum{Q(x)=ax^2+bx+c>0\\ a\in-\N,\, b,c\in\Z,\, b^2-4ac=D }Q(x)^k,\qquad x\in\R,\quad k\in2\N+3,\ \square\neq D\equiv 1\mod 4\label{eq:def-AkD}
\end{equation}
was introduced and studied by Zagier in~\cite{Zagier1999}\footnote{If $k=1,3$ $A_{k,D}(x)$ can still be defined but it turns out to be constant.}; see~\cite{Bengoechea2015} for more details on the convergence of the sum.
With a simple algebraic computation we can verify that $A_{k,D}(x)$ is a $1$-periodic QMF of weight $-2k$, satisfying~\eqref{mcg} with a period function $h = h_{k,D}$ which is a polynomial of degree $2k$ satisfying by~\cite[eq.~(25)]{Zagier1999}
$$ h_{k,D}(0)=A_{k,D}(0) = \ssum{0\leq b < \sqrt{D} \\ b^2 \equiv D \mod{4}} \sigma_k(\tfrac{D-b^2}4) $$
where~$\sigma_k(n) = \sum_{d\mid n}d^k$. Theorem~\ref{ds} applies to~$A_{k,D}$ and gives another proof that it extends to~$\R$ with $A_{k,D}\in\mathcal C^{k-1}(\R,\R)$. The fact that $A_{k,D}(x)$ is in $\mathcal C^{k-1}(\R,\R)$ was instead obtained by Zagier in~\cite{Zagier1999} by observing that $h_{k,D}$ belongs to the space of period polynomials corresponding to modular forms of weight $2k+2$. This implies in particular that $A_{k,D}(x)$ has to coincide with the Eichler integral of a weight $2k+2$ modular form and thus can be written as a Fourier series whose~$n$-th Fourier coefficient decays roughly
as $n^{-k-1/2}$, see~\cite[eq.~(53)]{Zagier1999}.

From Theorem~\ref{dless0} we deduce the following distributional result for $A_{k,D}$.

\begin{corollary}\label{corakd}
  Let $k\geq 5$ be odd and assume $D \equiv 0,1 \mod 4$ is not a square. Then
  $$ \Big\{ A_{k,D}(\tfrac aq) \colon 1\leq a \leq q, (a, q)=1\Big\} \subset \R $$
  becomes distributed, as~$q\to \infty$, according to the measure~$(\extneg{A_{k,D}})_\ast(\df \nu)$. This measure has a continuous cumulative distribution function.
\end{corollary}

More generally, let $g$ be in the space $ \textnormal{S}_{k}(1)$ of cusp forms of weight $k\geq12$ and level $1$. If $g(z):=\sum_{n\geq1}a_n\e(nz)$ for $\Im(z)>0$, where $\e(z):=\e^{2\pi i z}$, the Eichler integral of~$g$
\begin{equation}
  \tilde g(z):=\sum_{n\geq1}\frac{a_n}{n^{k-1}}\e(nz), \qquad (\Im(z)\geq 0),\label{eq:def-eichlerint}
\end{equation}
restricted to $\Q$, is a quantum modular form of weight $2-k$ with period function given by the period polynomials of $g$.

\begin{corollary}\label{cormof}
  For $k\geq12$, let $0\neq g\in \textnormal{S}_{k}(1)$. Then
  $$ \Big\{ \tilde g(\tfrac aq) \colon 1\leq a \leq q, (a, q)=1\Big\} \subset \C$$
  becomes distributed, as~$q\to \infty$, according to~$\extneg{\tilde g}_\ast(\df \nu)$. For any non-zero linear form~$\phi:\C\to\R$, the measure~$(\phi\circ \extneg{\tilde g})_\ast(\df\nu)$ is diffuse.
\end{corollary}

It follows from Theorem~\ref{ds} that the map~$\tilde g$ is~$k/2-1$ times differentiable, but in this special case it could be seen immediately from the definition~\eqref{eq:def-eichlerint} and square-root cancellation in averages of~$a_n$, see~\cite[theorem~5.3]{Iwaniec1997}. We comment on this after the proofs, see Remark~\ref{rmk:regularity-eichler}.

\subsubsection{Period functions of Maa\ss{} forms}

In~\cite{Lewis2001}, Lewis and Zagier extend in some sense the theory of period polynomial to period functions associated to Maa\ss{} forms. Their work then allows to construct quantum modular forms also from Maa\ss{} forms, as was described by Bruggeman~\cite{Bruggeman2007}. We briefly review their construction. Let~$u$ be a Maa\ss{} cusp form for~$\SL(2, \Z)$ of Laplace eigenvalue $s(1-s)$, with $s\in1/2+i\R_{>0}$, which we expand as~$u(x+iy):=\sum_{n\neq0}a_n\sqrt yK_{s-\frac12}(2\pi |n|y)\e(nx)$. 
We define 
$$ \tilde u(z) := \frac{\pi^s \Gamma(1-s)}2 \sum_{n\geq1} n^{s-1/2} a_n \e(nz), \qquad (\Im(z)>0). $$
We then include $\Q$ in the domain of $\tilde u$ by letting
\begin{equation}
  \tilde u(x) := \lim_{y\to 0^+} \tilde u(x+iy), \qquad (x\in\Q).\label{eq:def-utilde}
\end{equation}
It is shown in~\cite{Bruggeman2007, Lewis2001} that $\tilde u(x)$ thus defined is a quantum modular form of weight $2s$ whose associated period function~$h$ is~$\CC^\infty$ on $\R$ and real-analytic on~$\R\setminus\{0\}$. Using a suitable variant of Theorem~\ref{dless1} we will deduce the following result. 

\begin{corollary}\label{cormaf}
  Let $u$ be a non-trivial Maass form of spectral eigenvalue $s(1-s)$. Then
  $$ \Big\{ q^{-2s}\tilde u(\tfrac aq), 1\leq a \leq q, (a, q)=1\Big\} $$
  becomes distributed, as~$q\to \infty$, according to~$\extpos{\tilde u}_\ast(\df \nu)$. Moreover, for any non-zero linear form~$\phi:\C\to\R$, the measure~$(\phi\circ \extpos{\tilde u})_\ast(\df\nu)$ is diffuse.
\end{corollary}

We also show that~$\extpos{\tilde u}$ is almost everywhere locally~$(1/2-\eps)$-Hölder continuous, but again in this case, it can be seen directly from the Fourier expansion of the underlying form, see Remark~\ref{rmk:regularity-eichler}.

\subsubsection{A function of Kontsevich and Zagier}

Another example of quantum modular form, described in~\cite{Zagier2001,Zagier2010}, is given by the series
\begin{equation}
  \varphi(x):=\e(x/24)\sum_{m=0}^\infty (1-\e(x))\cdots (1-\e(m x)),\qquad (x\in\Q).\label{eq:def-kontsevich}
\end{equation}
This function was studied by Kontsevich and is related to the Stoimenow's numbers, which are used to bound the number of linearly independent Vassiliev invariants of a given degree. In~\cite{Zagier2001}, a proof is sketched that $\varphi$ is a quantum modular form of weight $\frac32$ in the generalized meaning that
\begin{align}\label{rvas}
  \varphi(x) - \e(\pm 1/8)|x|^{-3/2}\varphi(-1/x)=h(x),\ \pm x\in\Q_{>0}\qquad \varphi(x+1)=\e(1/24)\varphi(x),
\end{align}
for a period function $h(x)$ which is smooth on $\R$ (and in fact also continues analytically to~$\C\setminus i\R_+$)\footnote{Here $h$ is, up to a constant, naturally interpreted as the period function for the $\frac12$ weight modular form given by the Dedekind $\eta$ function, namely~$h(x) = c \int_0^\infty \eta(iy) (y+ix)^{-3/2}\df y$ for some constant~$c\in\R$.}. See~\cite{BringmannRolen2016} for a proof in the context of period functions for half-integral weight forms, and~\cite{Goswami2021} for a generalization to periodic theta functions. In this case, due to the automorphy factors in the reciprocity relation~\eqref{rvas}, the definition of the limiting function~$\extpos{\varphi}$ is as follows: for~$x \in \Q\cap (0, 1]$, $x = [0; b_1, \dots, b_r]$ with~$r$ odd, denote~$\sigma(x) = 3 + \sum_{1\leq j \leq r} (-1)^j b_j$. Then
\begin{equation}
  \extpos{\varphi}(x) := \begin{cases} \e(\frac{-1}{24}\sigma(x))\den(x)^{-3/2} \varphi(\bar x) & (x\in \Q\cap (0, 1]), \\ \lim_{\Q \ni y \to x} \extpos{\varphi}(y) & (x\not\in\Q). \end{cases}\label{eq:def-extpos-kontsevich}
\end{equation}
By suitable variants of Theorems~\ref{ds1} and~\ref{dless1}, we obtain the following result.

\begin{corollary}\label{corvas}
The map~$\extpos{\varphi}$ is continuous and in fact almost everywhere locally~$(3/4-\eps)$-Hölder continuous for any~$\eps>0$.

Moreover, the multiset
  $$ \Big\{ \e(\tfrac{-1}{24}\sigma(x)) q^{-3/2} \varphi(\tfrac aq) \colon 1\leq a \leq q, (a, q)=1\Big\} $$
  becomes distributed, as~$q\to \infty$, according to~$ \extpos{\varphi}_\ast(\df \nu)$.   For any~$\theta\in\R$, the cumulative distribution function of $(\Re\e^{i\theta}\extpos{\varphi})_\ast(\df \nu)$ is continuous.
\end{corollary}

\begin{remark}
  The function~$\sigma(x)$ can be expressed in terms of the Dedekind sum $s(x)$~\cite{Rademacher1972}: using~\cite[Theorem~1]{Hickerson1977}, we find
  $$ \sigma(x) = x + \bar x + 12 s(x), \qquad (x\in (0, 1)\cap\Q). $$  
\end{remark}

The function~$\extpos{\varphi} : [0, 1] \to \C$ is depicted in Figure~\ref{fig:kontsevich}. Note that~$\extpos{\varphi}(0) = 1$.

\begin{figure}[h]
  \centering
  \begin{tabular}{cc}
    \image{.45}{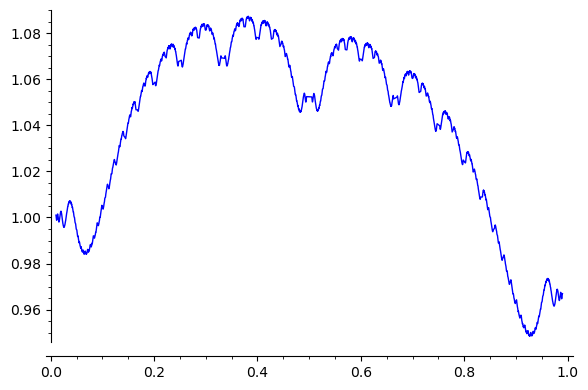} & \image{.45}{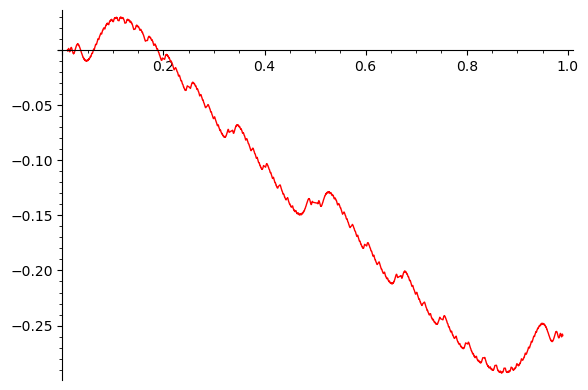}
  \end{tabular}
  \caption{Approximate plots of~$\Re \extpos{\varphi}$ (left) and $\Im \extpos{\varphi}$ (right) defined in~\eqref{eq:def-kontsevich}.}
  \label{fig:kontsevich}
\end{figure}

If we normalize $\varphi$ by letting $ \varphi^\dagger(x):=\den(x)^{-3/2}\varphi(x)$, then as an immediate consequence of Corollary~\ref{corvas} we obtain that $\overline{\varphi^\dagger(\Q)}$ is equal to the curve~$\CG$ given by
$$ \CG := \bigcup_{n=1}^{24} \e^{\pi i n/12} \extpos{\varphi}([0, 1]), $$
which passes through the~$24$-th roots of unity. In Figure~\ref{fig:graph-extposphi}, we have plotted the points~$\extpos{\varphi}(x)$ for~$\den(x)\leq 101$. The limiting curve is the graph of~$\extpos{\varphi}$. In Figure~\ref{fig:graph-phidagger}, we have plotted the points~$\varphi^\dagger(x)$ for~$\den(x)\leq 307$, colored depending on~$\sigma(x)$. The limiting curve is the curve~$\CG$.

\begin{figure}
  \centering
  \begin{subfigure}{.5\textwidth}
    \centering
    \image{.5}{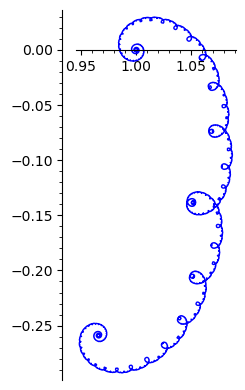}
    \caption{Approximate plot of~$\extpos{\varphi}([0, 1])$}
    \label{fig:graph-extposphi}
  \end{subfigure}%
  \begin{subfigure}{.5\textwidth}
    \centering{}
    \image{.8}{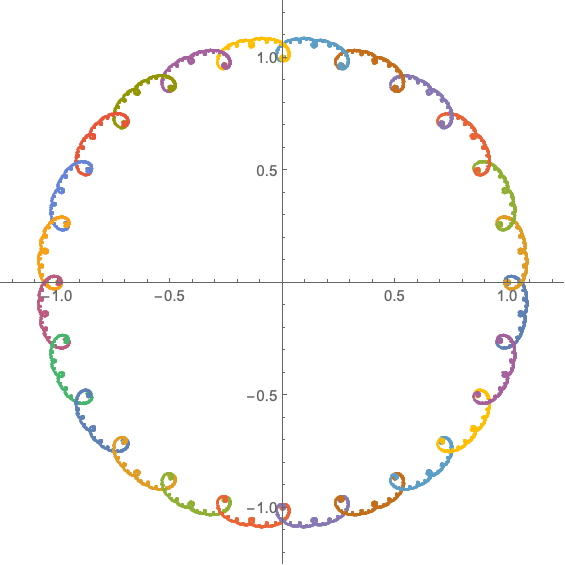}
    \caption{Approximate plot of~$\varphi^\dagger([0, 1])$}
    \label{fig:graph-phidagger}
  \end{subfigure}
  \caption{}
\end{figure}

\subsubsection{Generalized cotangent sums}

Finally we mention the generalized cotangent functions, studied in~\cite{Bettin2013a} and defined for~$b\in\Z$, $q\in\N_{>0}$ coprime by
$$
c_a\bigg(\frac bq\bigg):=q^a\sum_{m=1}^{q-1}\cot\bigg(\frac{\pi m b}{q}\bigg)\zeta\bigg(-a,\frac mq\bigg),
$$
where~$a\in \C$. Here $\zeta(s,x)$ denotes the Hurwitz zeta function, which is the analytic continuation of~$s\mapsto \sum_{n\geq 1} (n+x)^{-s}$. When~$a=-1$, the poles of $\zeta$ cancel out and~$c_a$ reduces in that case to the classical Dedekind sum~\cite{Rademacher1972}. In~\cite{Bettin2013a}, it was shown that the functions~$c_a$ are almost quantum modular forms of weight $1+a$ with period function $h_a$ satisfying, in a neighborhood of~$0$, the estimate
$$ h_a(x) = \frac{\zeta(1-a)}{\zeta(-a)}\frac {i}{\pi x} - \sgn(x)\cot(\frac{\pi a}2)\frac{i}{\abs{x}^{1+a}} + O(1). $$
The meaning of ``almost'' here will be made precise later on.
The function~$c_a$ is very closely related to Eichler integrals of certain real-analytic Eisenstein series, see~\cite[Section~2]{Bettin2013a} and~\cite{LewisZagier2019}.
We will deduce from our main results the following statement on the distribution of values of~$c_a$.

\begin{corollary}\label{cordedek}
  If~$\Re(a)<-1$, then the multiset
  $$ \Big\{ c_a(\tfrac bq) \colon 1\leq b \leq q, (b, q)=1\Big\} $$
  becomes distributed, as~$q\to \infty$, according to~$\df\lambda_a := (\extneg{c_a})_\ast(\df \nu)$. 
  If~$\Re(a)>-1$, then the multiset
  $$ \Big\{ q^{-1-a} c_a(\tfrac bq) \colon 1\leq b \leq q, (b, q)=1\Big\} $$
  becomes distributed, as~$q\to \infty$, according to a measure~$\df\lambda_a := (\extpos{c_a})_\ast(\df \nu)$. Moreover, 
  \begin{enumerate}[ref=(\arabic*)]
    \item When~$a\in (2\Z_{\geq0}+1)$, the measure~$\df\lambda_a$ is supported on~$\{0\}$.
    \item\label{it:dedek2} When~$a\in \R\setminus(2\Z_{\geq -1}+1)$, the measure~$\df\lambda_a$ is supported inside~$\R$ and is diffuse.
    \item\label{it:dedek3} When~$\Re(a)\neq -1$ and~$a\not\in\R$, then for any non-zero linear form~$\phi:\C\to\R$, the push-forward measure~$\phi_\ast(\df\lambda_a)$ on~$\R$ is diffuse.
  \end{enumerate}
  When~$\Re(a)>0$, the map~$\extpos{c_a}$ is continuous at irrationals. Moreover, for~$0<\Re(a)\leq 1$, the map~$x\mapsto \extpos{c_a}(x)$ is~$\alpha$-Hölder-continuous locally almost everywhere, for any~$\alpha<(\Re(a)+1)/2$. When~$\Re(a)>1$, the same is true for the map~$x\mapsto \extpos{c_a}(x) + x a \zeta(1-a)/\pi$, and in particular the map~$\extpos{c_a}$ is then differentiable almost everywhere with derivative~$-a\zeta(1-a)/\pi$.
\end{corollary}

In this statement we have kept the notations~$\extneg{c_a}, \extpos{c_a}$, however we warn that the actual definitions differ from those in Theorems~\ref{ds} and \ref{ds1}, mainly due to the fact $c_a$ is only ``almost'' a quantum modular form. The precise definitions are given in Section~\ref{sec:cotangent-sums} below in several cases depending on the value of~$a$.

In the case~\ref{it:dedek3} above, we obviously also have that~$\df\lambda_a$ itself has no atoms.

The empirical cumulative distribution functions (CDF) in Corollary~\ref{cordedek} for~$q=5000$ are plotted in Figure~\ref{fig:dedek} for some real values of~$a$. The fact that the corresponding measures are diffuse, stated in item~\ref{it:dedek2} of the previous corollary, translates into the continuity of the limiting functions.

\begin{figure}[h]
  \centering
  \image{.45}{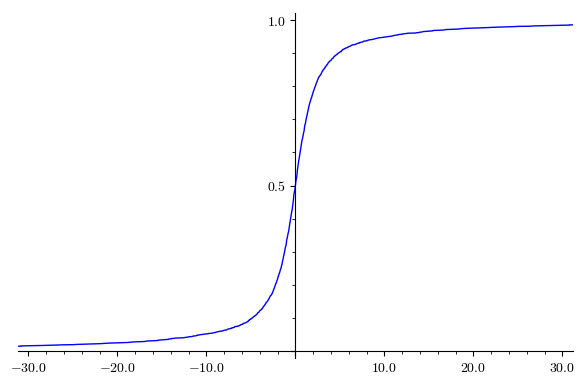}
  \image{.45}{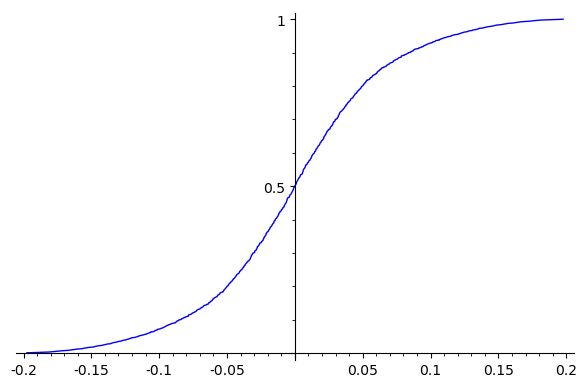}
  \caption{CDF of the empirical measures at~$q=5000$ in Corollary~\ref{cordedek}, for~$a = -2$ (left) and~$a = 1/2$ (right)}
  \label{fig:dedek}
\end{figure}

The case~$a=0$ was studied earlier in~\cite{Maier2016, Bettin2015}. The fact that the values of $c_0$ have a limiting distribution was proved in~\cite{Maier2016} with an elaborate argument. A simpler proof was then given in~\cite{Bettin2015}. The methods of~\cite{Maier2016, Bettin2015} use an explicit computation of the moments of both the original QMF and the limiting measure; thereafter, the limiting measure is identified through the method of moments, subject to a good bound on these moments. These ingredients are not available in the generality of Theorem~\ref{dless1}. The proof that we give of Theorem~\ref{dless1} instead identifies directly the limiting function, not passing through its moments, thus allowing for much more general period functions~$h$.

The continuity of the cumulative distribution function in the case~$a=0$ was also obtained in~\cite{Maier2016}. The method used there however relies crucially on the fact that $c_0(\frac{\overline b_q}{q})$ can be written explicitly as $\frac q{2\pi}\sum_{n=1}^\infty(\frac12-\{n\frac bq\})/{n}$, where $\{x\}$ is the fractional part of $x$, and there is no clear way to extend it to the generality of Theorems~\ref{dless0} and~\ref{dless1}.

The (weight zero) case of~$a=-1$ corresponds to Dedekind sums, and it is proved in~\cite{Vardi1993} (see~\cite[Section~9.4]{BettinDrappeau} for another argument) that the values~$c_{-1}(x)$ tend to distribute according to a Cauchy distribution when~$x$ is picked at random among rationals of denominators at most~$Q$, $Q\to\infty$. It might be interesting to know if the CDF of the Cauchy distribution can be obtained as a limit of the distribution of values of~$\extpos{c_a}, \extneg{c_a}$, appropriately normalized, as~$a\to -1$; in other words, if the limiting functions in Figure~\ref{fig:dedek} tend to the CDF of the Cauchy distribution as~$a\to -1$ after an appropriate normalization.

In Figure~\ref{fig:cotangent} below, we present the plots of the real part of~$\extneg{c_a}$ and~$\extpos{c_a}$ for various values of~$a$. The relevant period function~$h(x)=h_a(x)$ roughly satisfies~$h(x) \sim \kappa(a) x^{-1} + O(1)$ for some constant~$\kappa(a)$, see~\eqref{ash2} below. When~$\Re(a)<-1$, we witness a rise in regularity in $\extneg{c_a}$ as~$\Re(a)$ decreases, but without reaching full continuity, due to the fact that~$h_a$ has a pole at~$0$. When~$\Re(a)>0$, Corollary~\ref{cordedek} shows that the map~$\extpos{c_a}$ is continuous at irrationals, and when~$\Re(a)>1$, that it has derivative~$-a\zeta(1-a)/\pi$ almost everywhere. In Figure~\ref{fig:cotg-1.5}, we have~$a=1.5$, and~$-a\zeta(1-a)/\pi \approx 0.09926$.

\begin{figure}[h]
  \begin{subfigure}{.5\textwidth}
    \centering
    \image{.9}{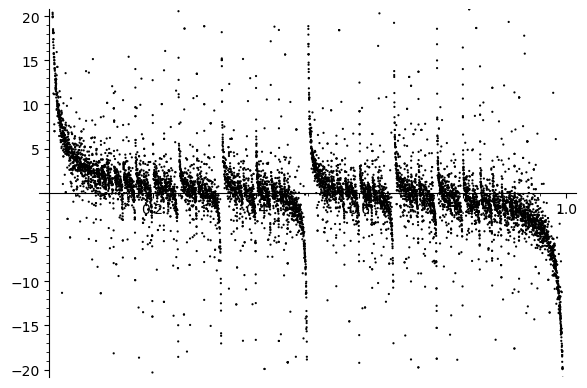}
    \caption{$(a, q, N) = (-0.7, 24001, 10000)$}
  \end{subfigure}%
  \begin{subfigure}{.5\textwidth}
    \centering
    \image{.9}{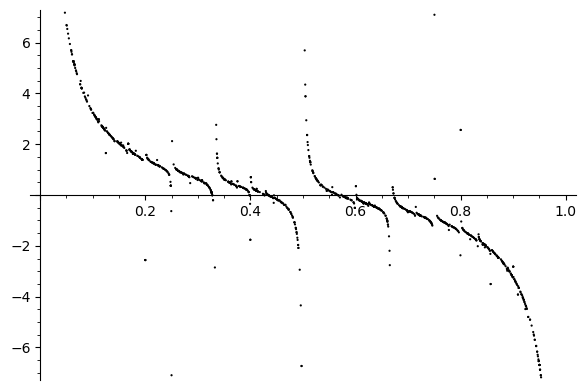}
    \caption{$(a, q, N) = (-3.2, 2001, 2000)$}
  \end{subfigure}
  \begin{subfigure}{.5\textwidth}
    \centering
    \image{.9}{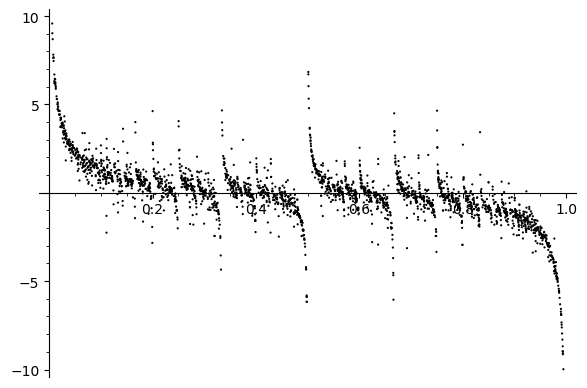}
    \caption{$(a, q, N) = (-0.5, 24001, 3000)$}
  \end{subfigure}%
  \begin{subfigure}{.5\textwidth}
    \centering
    \image{.9}{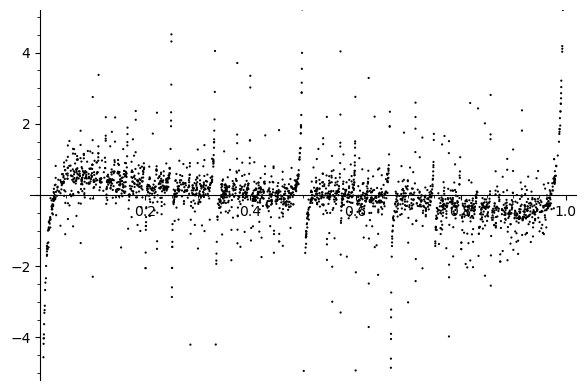}
    \caption{$(a, q, N) = (-0.5+0.51i, 24001, 3000)$}
  \end{subfigure}
  \begin{subfigure}{.5\textwidth}
    \centering
    \image{.9}{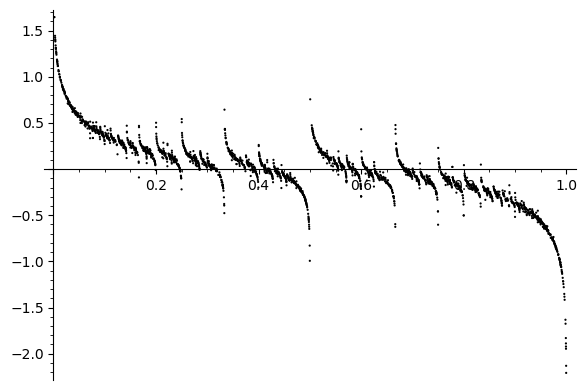}
    \caption{$(a, q, N) = (0, 24001, 3000)$}
  \end{subfigure}%
  \begin{subfigure}{.5\textwidth}
    \centering
    \image{.9}{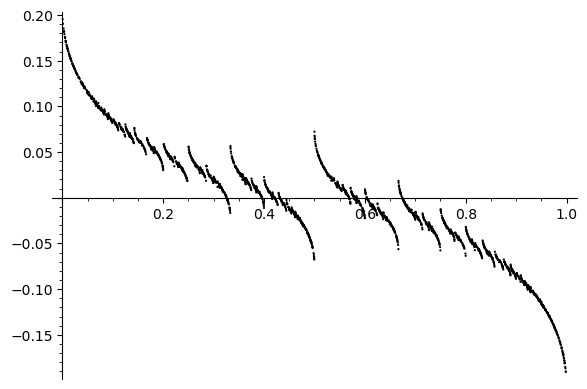}
    \caption{$(a, q, N) = (0.5, 24001, 3000)$}
  \end{subfigure}
  \begin{subfigure}{.5\textwidth}
    \centering
    \image{.9}{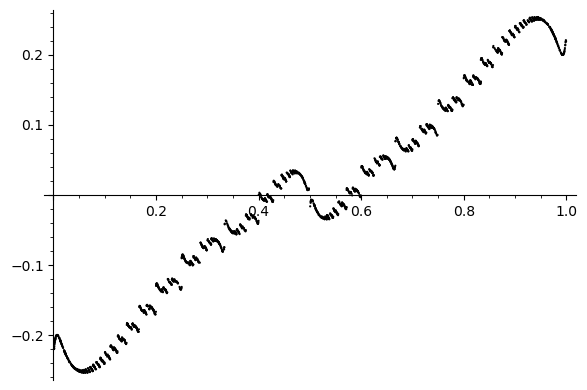}
    \caption{$(a, q, N) = (0.5+1.39i, 24001, 3000)$}
  \end{subfigure}%
  \begin{subfigure}{.5\textwidth}
    \centering
    \image{.9}{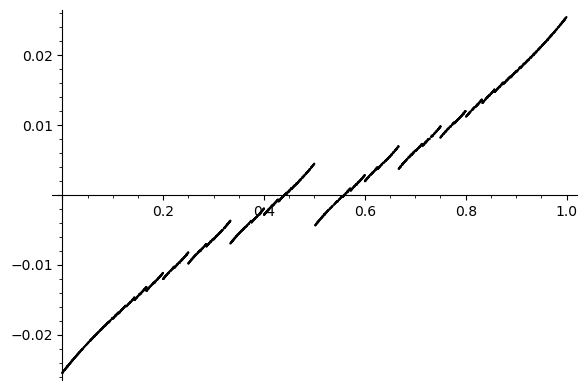}
    \caption{$(a, q, N) = (1.5, 24001, 3000)$}
    \label{fig:cotg-1.5}
  \end{subfigure}

  \caption{Sample points of~$c_a$ at $N$ points of denominator~$q$}
  \label{fig:cotangent}
\end{figure}

\section{Extending the domain of quantum modular forms}

In this section we define the main tools and notations, and then prove Theorems~\ref{ds}, \ref{thm:kneg-fdagger-ext}, \ref{ds1} and \ref{ds1-ext}.

\subsection{Notations}\label{sec:notations}

Let $T(x):=\{1/x\}$ be the Gauss map, $T^i$ its $i$-th iterate. For $x\in\Q$, we let $r:=r(x)$ be the minimum non-negative integer such that $T^{r}(x)=0$ and, for~$0\leq j < r$, we write $T^j(x)$ in simplest terms as
\begin{equation}
  T^j(x) = \frac{u_{j+1}(x)}{u_j(x)}.\label{eq:def-uj}
\end{equation}
In particular, we have
$$ u_0(x) = \den(x), \qquad u_r(x) = 1 ,$$
where~$\den(x)$ is the denominator of~$x$. Whenever~$u_{r-1}(x)>1$, we also define
$$ u_{r+1}(x) := 1, $$
which will be convenient to change the length of the continued fraction (CF) expansion.
Also, from the bound
\begin{equation}
  x T(x) \leq 1/2,\label{eq:size-T-consec}
\end{equation}
we deduce
\begin{equation}
  \frac{u_j(x)}{u_0(x)} \ll 2^{-j/2}, \qquad u_{r-j}(x) \gg 2^{j/2}.\label{eq:bounds-uj}
\end{equation}

We will also use the following other decomposition. Given~$y \in \Q\cap (0, 1)$ and writing uniquely $y$ as its CF expansion
$$ y = [0; c_1, \dotsc, c_s] $$
with~$s$ odd, we define
\begin{equation}
  v_0(y) := 1, \qquad v_j(y) := \den([0; c_1, \dotsc, c_j]) \qquad (1\leq j \leq s).\label{eq:def-vj}
\end{equation}
Equivalently,~$v_j(y) = \den([0; c_j, \dotsc, c_1])$.
Note that, with~$\bar{y} = [0; c_s, \dotsc, c_1]$, we have
$$ v_j(y) = u_{s-j}(\bar{y}), $$
and as such
\begin{equation}
  v_j(y) \gg 2^{j/2}\label{eq:growth-vj}
\end{equation}
with a uniform constant.

\subsection{Iteration of the reciprocity formula}

First, we notice that by the Euclid's algorithm, any QMF $f$ of level $1$ is determined by $h$ and its value at $0$. 
Indeed, by repeatedly applying~\eqref{mcg} and periodicity we obtain, for $x=\frac aq\in[0,1)$ in reduced form,
\begin{equation}\label{ffff}
  f(x)=\sum_{j=0}^{r-1}\bigg(\prod_{i=0}^{j-1}T^{i}(x)\bigg)^{-k} h((-1)^{j}T^j(x))+q^{k} f(0).
\end{equation}
Notice that we used that $\prod_{i=0}^{r-1}T^{i}(x)=1/q$.

The above quantities are naturally expressed in terms of continued fractions. Indeed, if $x\in(0,1)$, then the length of its CF expansion $x=[0;b_1,\dots,b_{r}]$ with minimal length (i.e. $b_{r}\neq1$ if~$r>1$) is $r$. Also, for $0\leq j\leq r$ we have $q \prod_{i=0}^{j-1}T^{i}(x)=(\prod_{i=j}^{r-1}T^{i}(x))^{-1}$ is the denominator $u_{j}(x)$ defined in~\eqref{eq:def-uj}.  Notice also that $u_0(x)=q$. Thus, abbreviating~$u_j = u_j(x)$, we can then rewrite~\eqref{ffff} as
\begin{equation}\label{ffff2}
  f(x)=\sum_{j=0}^{r-1}\Big(\frac{u_j}{u_0}\Big)^{-k} h\bigg((-1)^{j}\frac{u_{j+1}}{u_{j}}\bigg) + u_0^{k} f(0).
\end{equation}
This formula holds also if $r$ is formally replaced by~$r+1$, which corresponds to expressing~$x = [0;b_1,\dots,b_{r}]$ rather as~$[0; b_1, \dotsc, b_r - 1, 1]$, since the additional term in~\eqref{ffff2} is
\begin{equation}
  u_0^k h((-1)^r) = u_0^k(f(0)-f(0)) = 0,\label{eq:series-ffff-extended}
\end{equation}
by~\eqref{mcg} and periodicity. 

Finally, we deduce another variant of~\eqref{ffff2}. First, we notice that if $[0;b_1,\dots,b_r]$ is the CF expansion of $x\neq0$ with $r$ odd, then $\overline x=[0;b_r,\dots,b_1]$. In particular, after the change of variables $r\to r-j$,~\eqref{ffff2} can be rewritten as
\begin{equation}\label{ffff3}
  q^{-k}f(x)=\Psi(\overline x)
\end{equation}
where, for $y=[0;c_1,\dots,c_r]$ with $r$ odd, and the notation~\eqref{eq:def-vj},
\begin{equation}\label{ffff4}
  \Psi(y):=\sum_{j=1}^{r}v_{j}^{-k}h\bigg((-1)^{j-1}\frac{v_{j-1}}{v_{j}}\bigg)+f(0). 
\end{equation}

From the expansions~\eqref{ffff2} and \eqref{ffff4}, and the fact that both quantities~$u_j/u_0$ and~$v_j^{-1}$ decrease exponentially fast with~$j$, we see the difference of behaviour according to the sign of~$\Re(k)$, and the relevance of switching~$x$ to~$\bar x$ when~$\Re(k)>0$.

\subsection{Continuity almost everywhere and extension}

In this section, we state and prove a technical proposition which will be helpful when extending~$f$ almost everywhere in Theorems~\ref{thm:kneg-fdagger-ext} and \ref{ds1-ext}.

We first need the following Lemma. In the following we will use the notation $r(x)=\infty$ if $x\notin\Q$.

\begin{lemma}\label{bll0}
  For each~$B>0$, let
  $$ \gT(B) := \{x\in\R \colon  b_j(x) \leq \max(B, j(\log j)^2)\text{ for all } 1\leq j\leq r(x)\}. $$
  The set~$\gT(B)$ is invariant under~$x\mapsto x+1$. Moreover,~$\meas(\gT(B)\cap[0, 1]) = 1 + o(1)$ as~$B\to\infty$ and the set 
  $$ \gS = \bigcup_{B>0} \gT(B) $$
  is of full Lebesgue measure.
\end{lemma}
\begin{proof}
  By~\cite[Theorem~30]{Khintchine1963} one deduces that $\gS$ has full measure. Since $ \gT(B')\seq  \gT(B)$ if $0<B'\leq B$, one then deduces that $\meas(\gT(B)\cap[0, 1]) = 1 + o(1)$ as~$B\to\infty$.
\end{proof}

Given~$B>0$, an integer~$m\geq 1$ and a number~$x\in \gT(B)$, define
$$ V(B, m, x) := \big\{ x' \in \Q\cap \gT(B) \colon r(x') \geq m,\ \forall j\leq m, b_j(x') = b_j(x) \big\}, $$
the set of all rationals in~$\gT(B)$ whose first~$m$ coefficients coincide with those of~$x$.

\begin{definition}[Property~$\cS(\lambda)$]
  We say that~$f:\Q\to\C$ has the property~$\cS(\lambda)$ if the quantity
  $$ \Delta_\lambda(m) := \sup_{x\in\gT(m^\lambda)} \sup_{x', x'' \in V(m^\lambda, m, x)} \abs{f(x') - f(x'')} $$
  satisfies
  \begin{equation}
    \Delta_\lambda(m) \to 0\label{eq:precont-Delta}
  \end{equation}
  as~$m\to\infty$.
\end{definition}

\begin{proposition}\label{prop:fext-unbound-continuity}
  Let~$\lambda>1$, and assume that~$f$ has the property~$\cS(\lambda)$.
  Then for any~$x\in \gS \setminus\Q$, the limit
  \begin{equation}
    f^*(x) := \lim_{y\in \Q\cap\gT(B), y\to x} f(y)\label{eq:def-fext-unbound}
  \end{equation}
  exists for any $B>0$ such that $x\in \gT(B)\setminus\Q$. Moreover, let~$m\in\N$ and~$B>0$ be given with~$m \geq B^{1/\lambda}$, and define~$f^*(x) := f(x)$ for~$x\in\Q$. Then uniformly in~$x, x' \in \gT(B)$ satisfying~$r(x), r(x')\geq m$ and~$b_j(x) = b_j(x')$ for~$j\leq m$, we have
  \begin{equation}
    f^*(x) - f^*(x') = o(1), \qquad (m\to \infty),\label{eq:unicont-fext-unbound}
  \end{equation}
  where the rate of decay may depend on~$\lambda$, but not on~$B$.
\end{proposition}

\begin{remark}
  Let~$x\in \gT(B)\setminus\Q$. Since~$\gT(B') \subset \gT(B)$ for~$B \geq B' > 0$, from~\eqref{eq:def-fext-unbound} we also get~$f^*(x) = \lim_{y\in\Q\cap \gT(B'), y\to x} f(y)$, for any~$0<B'<B$ with~$x\in \gT(B')\setminus\Q$. In particular, the existence of the limit implies that its value~$f^*(x)$ is independent of~$B$. Notice that the rate of convergence, however, might depend on it.  
\end{remark}

\begin{proof}[Proof of Proposition~\ref{prop:fext-unbound-continuity}]
  Let~$x\in \gT(B)\setminus\Q$. We wish to show that the limit~\eqref{eq:def-fext-unbound} exists. By Cauchy's criterion, it suffices to show that
  \begin{align}\label{cdl}
    \lim_{\eps\to 0^+} \sup_{\substack{x', x'' \in \Q \cap\gT(B) \\ \abs{x'-x}, \abs{x''-x} \leq \eps}} \abs{f(x') - f(x'')} = 0. 
  \end{align}
  Let $\eps>0$. Since~$x\not\in\Q$, we may find $m\in\N$ such that for any~$y\in\Q\cap[x-\eps, x+\eps]$, then~$r(y)>m$ and~$b_j(y) = b_j(x)$ for all~$j\leq m$. Clearly $m\to\infty$ as $\eps\to0^+.$ In particular, we can assume~$m\geq B^{1/\lambda}$ so that~$\gT(B) \subset \gT(m^\lambda)$. Thus~$x\in\gT(m^\lambda)$ and~$\Q\cap[x-\eps, x+\eps]\cap \gT(B)\seq V(m^\lambda,m,x)$ and~\eqref{cdl} follows from~\eqref{eq:precont-Delta}.  

  It remains to show~\eqref{eq:unicont-fext-unbound}. Let~$x, x' \in \gT(B)$, $m \geq B^{1/\lambda}$, and for~$0\leq j \leq m$, define~$b_j := b_j(x) = b_j(x')$, consider the convergents~$(x_\ell)$, $(x'_\ell)$ of~$x$ and~$x'$ respectively, so that for instance
  $$ x_\ell = [b_0; b_1, \dotsc, b_m, b_{m+1}(x), \dotsc, b_{\ell}(x)], \qquad (m\leq \ell \leq r(x)). $$
  Also, consider the rational~$y := [b_0; b_1, \dotsb, b_m]$. We have~$y\in\gT(m^{\lambda})$ by hypothesis on~$m$. We obviously also have~$x_\ell\in V(m,m^{\lambda},y)$ whenever~$m\leq \ell \leq r(x)$, and similarly~$x'_\ell \in V(m,m^{\lambda},y)$ for~$m\leq \ell \leq r(x')$. We deduce that~$\abs{f(x_\ell) - f(x'_{\ell'})} \leq \Delta_{\lambda}(m)$ with the notation~\eqref{eq:precont-Delta}, and therefore
  $$ \abs{f^*(x) - f^*(x')} \leq \Delta_{\lambda}(m) $$
  by taking~$\ell = r(x)$ if~$x\in\Q$ or~$\ell\to\infty$ if~$x\not\in\Q$, and similarly for~$x'$. The right-hand side does not depend on~$B$ and tends to~$0$ by~\eqref{eq:precont-Delta}, which is what we claimed.
\end{proof}

\subsection{The case of $\Re(k)<0$.}

In this section we prove Theorems~\ref{ds} and \ref{thm:kneg-fdagger-ext}. We start by a lemma regarding the size and regularity of products of consecutive iterates of the Gauss map.

\begin{lemma}\label{lem:reg-bounds-w}
  Let~$j\in\N$, $\lambda\in\R$ and~$g:[0, 1] \to \C$. For~$x\in [0, 1)$, define
  \begin{equation}
    w(x) = w(x; j, \lambda, g) := \1(j\leq r(x)) \Big(\prod_{i=0}^{j-1} T^i(x)\Big)^\lambda g(T^j(x)),\label{eq:expr-w-prod}
  \end{equation}
  where we recall that~$r(x) = +\infty$ by convention if~$x\not\in\Q$ and $\1(j\leq r(x))$ indicates the indicator function of the condition $j\leq r(x)$.
  \begin{itemize}
    \item If~$\lambda>0$ and~$g$ is continuous, then~$w$ is continuous on~$[0, 1)\setminus \{x, r(x)\in\{j-1,j\} \}$. Moreover, letting~$\pm 1 = (-1)^j$, we have
    \begin{align}
      & w(x^\pm) = \den(x)^{-\lambda} g(0) && (r(x) \in \{j-1, j\}), \notag\\
      & w(x^\mp) = \den(x)^{-\lambda} g(1) && (r(x) = j), \label{1cs}\\
      & w(x^\mp) = 0 && (r(x)=j-1).\notag
    \end{align}
    \item If~$\lambda>2$ and~$g$ is of~$\CC^1$ class, then~$w$ is of~$\CC^1$ class on~$[0, 1)\setminus \{x, r(x)\in\{j-1,j\} \}$. When~$r(x)\not\in\{j-1, j\}$, we have
    \begin{equation}
      \begin{aligned}
        w'(x) = \1(j\leq r(x)) \Bigg\{& \lambda g(T^j(x)) \sum_{i=0}^{j-1} (-1)^i \Big(\prod_{0\leq \ell < i} T^\ell(x)\Big)^{\lambda-2} T^i(x)^{\lambda-1} \Big(\prod_{i<\ell<j} T^\ell(x)\Big)^\lambda \\
        {}& \quad + (-1)^j g'(T^j(x)) \Big(\prod_{0\leq \ell < j} T^\ell(x)\Big)^{\lambda-2}\Bigg\}.
      \end{aligned}\label{eq:expr-wprime-prod}
    \end{equation}
    Moreover, we have
    \begin{align}
      & w'(x^\pm) = \den(x)^{-\lambda+2}\Big(\lambda g(0)\sum_{i=0}^{j-1}\frac{(-1)^i}{u_i(x)u_{i+1}(x)} \pm g'(0)\Big) && (r(x) \in \{j-1, j\}), \notag\\
      & w'(x^\mp) = \den(x)^{-\lambda+2} \Big(\lambda g(1)\sum_{i=0}^{j-1}\frac{(-1)^i}{u_i(x)u_{i+1}(x)} \pm g'(1)\Big) && (r(x) = j), \label{fwp}\\
      & w'(x^\mp) = 0 && (r(x) = j-1).\notag
    \end{align}
    \item If~$\lambda>0$ and~$g\in\CC^0$, we have
    \begin{equation}
      \|w\|_\infty \leq 2^{\lambda(1 - j/2)} \|g\|_\infty,\label{eq:bound-w}
    \end{equation}
    while if~$\lambda>2$ and~$g\in\CC^1$, we have
    \begin{equation}
      \|w'\|_\infty \leq j 2^{(\lambda-2)(1-j/2)}(\|g\|_\infty + \|g'\|_\infty).\label{eq:bound-wprime}
    \end{equation}
  \end{itemize}
\end{lemma}
\begin{proof}
  \begin{itemize}
    \item \textit{Value of~$w'$.} For the convenience of our argument, we start by noting that if~$g\in \CC^1$ and~$\lambda>2$, then whenever~$r(x)>j$, the right-hand side of~\eqref{eq:expr-w-prod} defines a~$\CC^1$ function in a neighborhood of~$x$. When~$m\geq1$, its derivative can be computed using the expression
    \begin{align*}
      (T^i)'(x) = (-1)^i \Big(\prod_{0\leq \ell < i}T^\ell(x)\Big)^{-2}, 
    \end{align*}
    and this leads to the given formula for~$w'(x)$.

    \item \textit{Regularity.} We treat the continuity and differentiability claims simultaneously, proving the following more general assertion: given~$m\in\{0, 1\}$,~$j\geq 1$ and~$\lambda>2m$, if there is a countable set~$S$ for which~$g$ is~$\CC^m$ on~$[0, 1]\setminus S$, then~$x\mapsto w(j, \lambda, g)$ is $\CC^m$ on the set
    $$ [0, 1] \setminus \{x: r(x)\in\{j-1, j\} \text{ or } (r(x)>j \text{ and } T^j(x) \in S)\}. $$
    We proceed by induction. Assume first~$j=1$. We have~$r(x)\geq 1 \iff x\neq 0$, so that
    $$ w(x; 1, \lambda, g) = \1(x\neq 0) x^\lambda g(T(x)). $$
    This function is easily seen to be $\CC^m$ at~$x\in[0, 1]$ whenever~$x>0$, $T(x) \in (0, 1)$ and~$T(x) \not \in S$. These conditions amount to~$r(x) > 1$ and~$T(x) \not\in S$, which gives our claim in the case~$j=1$.

    Let~$j>1$, and suppose our claim is proven for~$j-1$. Then we note that
    $$ w(x; j, \lambda, g) = w\big(x; 1, \lambda, [x\mapsto w(x; j-1, \lambda, g)]\big). $$
    By induction, the function~$x\mapsto w(x; j-1, \lambda, g)$ is $\CC^m$ on~$[0, 1]\setminus S'$, where
    $$ S' = \{x, r(x) \in \{j-2, j-1\}\text{ or } (r(x)>j-1 \text{ and } T^{j-1}(x) \in S)\}. $$
    The set~$S'$ is again countable, since~$T^{j-1}$ has countably many inverse branches. From the case~$j=1$ of our claim, we deduce that~$x\mapsto w(x; j, \lambda, g)$ is $\CC^m$ on~$[0, 1]\setminus S''$, where
    \begin{align*}
      S'' ={}& \{x, r(x) \in \{0, 1\} \text{ or } (r(x)>1\text{ and } T(x) \in S'\} \\
      = {}& \{x, r(x) \in \{0, 1, j-1, j\} \text{ or } (r(x)>j\text{ and } T^j(x) \in S)\}
    \end{align*}
    by definition of~$S'$. There remains to check that~$w$ is~$\CC^m$ at~$x$ if~$r(x) = 0$ and~$j\geq 2$, or if~$r(x) = 1$ and~$j\geq 3$. But supposing~$r(x) = 1$ we have for~$\eps\to 0^+$
    $$ T(x-\eps) = O(\eps), \qquad T(x+\eps) = 1-O(\eps)\ \Rightarrow\ T^2(x+\eps) = O(\eps), $$
    and therefore, we have in any case~$\lim_{y\to x} T(y) T^2(y) = O(\abs{x-y})$.
    Thus if~$j\geq 3$ and~$\lambda>2m$, we then have as~$y\to x$,
    $$ w(y; j, \lambda, g) = O((T(y) T^2(y))^\lambda) = o(\abs{x-y}^m), $$
    and so~$w$ is continuous at~$x$, and differentiable there when~$m=1$, with derivative~$0$. Moreover, if~$m=1$, by the explicit expression~\eqref{eq:expr-wprime-prod} and the bounds above, we find
    $$ w'(y) = O(\abs{x-y}^{\lambda-2}) = o(1), $$
    and therefore~$w$ is~$\CC^m$ at~$x$.
    A similar argument holds when~$r(x) = 0$, and concludes the proof of our claim.

    \item \textit{Limits of~$w$.} Assume~$r(x) = j$. Note that for~$y\neq x$ in the neighborhood of~$x$, we have~$r(y)>r(x)$, and moreover
    $$ \lim_{y\to x^\pm} T^j(y) = 0, \qquad \lim_{y\to x^\mp} T^j(y) = 1. $$
    The expression given for~$w(x^\pm)$ and~$w(x^\mp)$ then follows upon taking the limit in the expression~\eqref{eq:expr-w-prod}. The case when~$r(x)=j-1$ is similar, using instead the limits
    $$ \lim_{y\to x^\pm} T^j(y) = \lim_{y\to x^\mp} T^{j-1}(y) = 0. $$

    \item \textit{Limits of~$w'$.} The expressions given for~$w'(x^\pm)$ when~$r(x) \in\{j-1, j\}$ follow from an argument identical to the computation of the limits~$w(x^\pm)$. In the claimed formulae, we have used the notation~\eqref{eq:def-uj} (for~$x$ rational).

    \item \textit{Bounds.} The bounds~\eqref{eq:bound-w}, \eqref{eq:bound-wprime} follow immediately from the explicit expressions above, along with the bound~\eqref{eq:size-T-consec}.
  \end{itemize}
\end{proof}

\subsubsection{Proof of Theorem~\ref{ds}. The domain of $\fem$}\label{dmf}

In this section we show that  $\fem(x)$ is defined for all $x\in\R$. By periodicity, we can assume $x\in (0,1)\setminus\Q$. We also recall that $h$ is bounded on $[-1,1]$ by hypothesis.
We let
\begin{align} \label{dwj}
  h_j(y) := h((-1)^j y), \qquad w_j(y) := w(y, j, -k, h_j),
\end{align}
for $y>0$, and extend $h_j$ at $y=0$ by continuity, so that in particular 
\begin{equation}\label{wr}
  w_{r(x')}(x')=\den(x') ^k h(0^\pm),\qquad (x'\in\Q,\ \pm = (-1)^{r(x')}).
\end{equation}
Thus, with this notation,~\eqref{ffff} reads
\begin{equation}
  f(x') = \sum_{j\geq 0} w_j(x') + \den(x') ^k(f(0)-h(0^\pm)).
   \label{eq:expr-f-series-rat}
\end{equation}
It suffices to show that $\sup |f(x')-f(x'')|\to0$ as $\eps\to0$, where the sup is over all $x',x''\in\Q$ with $|x-x'|,|x-x''|<\eps$. Let~$\eta>0$ be arbitrary. Since~$x\not\in\Q$, there exists~$m(x, \eps) \to \infty$ as~$\eps \to 0$ such that for all~$x^* \in \Q$, $\abs{x-x^*}\leq \eps$, we have~$r(x^*)>m(x, \eps)$ and~$b_j(x^*) = b_j(x)$ for~$j\leq m(x, \eps)$. By~\eqref{eq:bound-w} for~$\eps$ sufficiently small 
we have
$$ \abs{f(x') - f(x'')} \leq \eta + \sum_{j=1}^{m(x, \eps)-1} \abs{w_j(x') - w_j(x'')}, \qquad \forall x',x''\in (x-\eps,x+\eps)\cap \Q.$$
By Lemma~\ref{lem:reg-bounds-w}, the function~$y \to w_j(y)$ is continuous at $x$ and thus for $x',x''\in (x-\eps',x+\eps')\cap \Q$ with $\eps'\in(0,\eps)$ sufficiently small we have $\abs{f(x') - f(x'')} \leq2\eta$. Since~$\eta$ was arbitrary, this yields our claim.

In particular, approximating~$x = [0; b_1, b_2, \dotsc] \in(0, 1)\setminus\Q$ by the sequence~$x_n = [0; b_1, \dotsc, b_n]$, we get the normally converging series expression
\begin{equation}
  \fem(x) =\sum_{j\geq0} w_j(x)  
   \label{eq:expr-fdagger-irrat}
\end{equation}
for $x\notin\Q$. Notice that the same expression holds also for $x\in\Q$ if $f(0)=h(0^\pm)$, by~\eqref{eq:expr-f-series-rat} .

\subsubsection{Proof of Theorem~\ref{ds}. The continuity of $\fem$}\label{adad}

First, assume $x\in(0,1)\setminus\Q$. Then, using the expression~\eqref{eq:expr-fdagger-irrat} and an identical argument as above, the continuity of~$\fem$ at~$x$ follows immediately.

Assume, then, that~$x\in[0, 1]\cap\Q$. Let~$r = r(x)$ and~$q=\den(x)$.
For all~$y\neq x$ in a neighborhood of~$x$, we have~$r(y)\geq r + 1$, and therefore, with~$\pm = (-1)^{r}$,
\begin{align*}
  \fem(y) - f(x) = {}& q^k(h(0^\pm)-f(0)) + O(\den(y)^{\Re(k)}) + \sum_{j\geq 0}(w_j(y) - w_j(x))
\end{align*}
where the term~$O(\den(y)^k)$ is to be ignored if~$y\not\in\Q$. We let~$y\to x$.  
First note that~$\den(y)^k \to 0$. Also, for~$j\not\in\{r, r+1\}$, the function~$w_j$ is continuous by Lemma~\ref{lem:reg-bounds-w}, and so the~$j$-th summand in the sum tends to~$0$.
By dominated convergence, we deduce
\begin{align*}
  \fem(y) - f(x) & = {} o(1) + q^k(h(0^\pm) - f(0)) + w_r(y) - w_r(x) + w_{r+1}(y)\\
  & = {} o(1) + w_r(y)  + w_{r+1}(y)-  q^k f(0),
\end{align*}
by~\eqref{wr}.
Finally, by~\eqref{1cs} and since~$h_r(1) = h((-1)^r) = 0$ by~\eqref{eq:series-ffff-extended}, we have
\begin{equation}\label{wdis}
  \begin{split}
    w_r(y) = {}& o(1) + q^k \begin{cases} h_r(0), & (\sgn(y-x) = (-1)^r), \\ 0, & (\sgn(y-x) = (-1)^{r+1}), \end{cases} \\
    w_{r+1}(y) = {}& o(1) + q^k \begin{cases} 0, & (\sgn(y-x) = (-1)^r), \\ h_{r+1}(0), & (\sgn(y-x) = (-1)^{r+1}). \end{cases} 
  \end{split}  \end{equation}
Since~$h_r(0) = h(0^\pm)$,~$h_{r+1}(0) = h(0^\mp)$, then splitting in cases depending on the sign $\pm$ and on whether $y\to0$ from the right or the left, we find
$$ \fem(y) - f(x) \to \begin{cases} q^k(h(0^+) - f(0)), & (y\to x^+), \\ q^k(h(0^-) - f(0)), & (y \to x^-) \end{cases} $$
as claimed. We also notice for later use that if $h$ is continuous at zero then~\eqref{wdis} gives
\begin{equation*}
  \lim_{y\to x^+}  w_r(y)=\lim_{y\to x^-}  w_{r+1}(y),\quad\lim_{y\to x^-}  w_r(y)=\lim_{y\to x^+}  w_{r+1}(y).
\end{equation*}
\subsubsection{Proof of Theorem~\ref{ds}. The differentiability of $\fem$}

We argue by induction, starting from the case~$m=1$. Assume that~$h\in\CC^1([-1, 1])$ and~$\Re(k)<-2$. 

By Lemma~\ref{lem:reg-bounds-w} we have that $w_j$ is of~$\CC^1$ class on~$[0, 1)\setminus \{x, r(x)\in\{j-1,j\} \}$ (and thus on any neigborhood of irrational numbers). In particular, by~\eqref{eq:expr-fdagger-irrat},~\eqref{eq:bound-w} and~\eqref{eq:bound-wprime} it follows that $\fem$ is differentiable on $\R\setminus\Q$. The same argument also gives that $\sum_{j\neq r,r+1} w_j(x)$ is differentiable at rationals. In particular, it suffices to show that the right and left derivatives of $w_r(y)+w_{r+1}(y)$ coincide at rationals. To show this we start by observing that from the period relation~\eqref{mcg}, proceeding as in~\cite[eq.~(23)]{Zagier1999}, we have
$$ h(x+1) - h(x)  =-\abs{1+x}^{-k}h(1-1/(x+1)) ,\qquad x\in\Q\setminus\{0,-1\}. $$
In particular, taking the limit as~$x\to 0$ we obtain $h(1)=0$ and $h'( 1) =  k h(0)$. Thus, since by~\eqref{mcg} we have $h(-1/x)=-|x|^{-k}h(x)$, we find $h'(- 1)=-kh(0)$ and thus $h_j'(1)=kh(0)$ for all $j$. Thus, recalling that $h_j(1)=0$ (and $h_j'(0)=(-1)^j h(0)$) and observing that
$$ -k\sum_{i=0}^{r(x)} \frac{(-1)^i}{u_i(x)u_{i+1}(x)} = -k\sum_{i=0}^{r(x)-1} \frac{(-1)^i}{u_i(x)u_{i+1}(x)} - (-1)^r k,\qquad x\in\Q\cap(0,1), $$
we obtain from~\eqref{fwp} that
\begin{equation*}
  w'_r(x^+)+w'_{r+1}(x^+)=w'_r(x^-) +w'_{r+1}(x^-), \qquad x\in\Q\cap(0,1),
\end{equation*}
thus giving that~$\fem$ is differentiable at any~$x\in\Q\cap(0,1)$. Finally, by~\eqref{mcg}, we see that $\fem$ is differentiable at~$0$ also, with~$(\fem)'(0) = h'(0)$. We conclude that~$\fem$ is differentiable on~$[0, 1)$, hence on $\R$ by periodicity.
By~\eqref{mcg}, its derivative satisfies, for~$x\neq 0$,
$$ h_1(x) = (\fem)'(x) - \abs{x}^{-k-2}(\fem)'(-1/x), $$
where~$h_1(x) := h'(x) - k\sgn(x) \abs{x}^{-k-1}f(-1/x)$. The function~$h_1$, extended at~$0$ by~$h_1(0)=h'(0)=(\fem)'(0)$, is continuous on~$[-1, 1]$. We can then apply the continuity property proven in Section~\ref{adad} to~$(\fem)'$ and deduce that~$(\fem)'$ is continuous on~$[0, 1]$, and hence on~$\R$.

If~$k<-2m$ and~$h\in\CC^m$, then this argument can be iterated~$m$ times, with functions~$(h_n)_{n\leq m}$ defined inductively by
$$ h_0(x) = h(x), \qquad h_{n+1}(x) = h_n'(x) - (k+2n) \sgn(x) \abs{x}^{-k-2n-1} (\fem)^{(n)}(-1/x). $$

\subsubsection{Proof of Theorem~\ref{thm:kneg-fdagger-ext}. Extending $f$ when $h$ is unbounded.}\label{exsc}

We now justify Theorem~\ref{thm:kneg-fdagger-ext}. We assume that~$h$ satisfies~\eqref{coh}. First note that due to the fact that periodicity is assumed separately on~$\R_+$ and~$\R_-$, we have the following variant of~\eqref{ffff},
$$ f(x) = \sum_{j=0}^{r-1} \Big(\prod_{i=0}^{j-1} T^i(x)\Big)^{-k} h((-1)^j T^j(x)) + q^k f((-1)^r). $$

\begin{lemma}\label{lem:fem-cS}
  Let~$\delta>0$ be such that the hypotheses~\eqref{coh} hold. Then~$f$ has the property~$\cS(1+\delta)$.
\end{lemma}

\begin{proof}
  Let~$x = [b_0; b_1, b_2, \dotsc] \in \gT(m^{1+\delta})$. By periodicity we can assume $x\in[0,1)$, that is $b_0=0$.
  If~$U:[0,1)\to\R$ is the inverse branch of~$T^m$, given by~$U(y) = [0; b_1, \dotsc, b_m+ y]$, then any~$x'\in V(m^{1+\delta}, m, x)$ is in the range of~$U$. Since~$T$ is expanding (\cite[eq.~(2.4)]{BaladiVallee2005} with~$\rho< 2^{-1/2}$), we have~$\|(T^j\circ U)'\|_\infty \ll 2^{-(m-j)/2}$ uniformly for~$0\leq j < m$, and therefore, for any~$x', x'' \in V(m^{1+\delta}, m, x)$,
  \begin{equation}
    \label{eq:bound-diff-Ti}
    \abs{T^j(x') - T^j(x'')} \ll 2^{-\frac12(m-j)}
  \end{equation}
  for~$j<m$. We note also that~$T^j(x') \asymp T^j(x)$ for~$x' \in V(m^{1+\delta}, m, x)$ and~$j<m$, with a uniform constant, and that the condition~$T^j(x') \gg \min(m^{-1-\delta}, j^{-1} (\log j)^{-2})$ holds uniformly over~$j$. Finally, for any~$x'\in V(m^{1+\delta}, m, x)$, the condition~$r(x')\geq m$ implies~$\den(x') \geq 2^{m/2}$ by virtue of~\eqref{eq:bounds-uj} (with~$j=r(x')$).\

  Given~$x\in \gT(m^{1+\delta})$ and~$x'\in V(m^{1+\delta}, m, x)$, with the same notation as in Section~\ref{dmf} we have
  \begin{equation*}
    f(x') = \sum_{0\leq j < m/2} w_j(x') + O\Big(\den(x)^{\Re(k)} + \sum_{j\geq m/2} 2^{\Re(k)j/2} \sup_{\abs{y}\gg \min(m^{-1-\delta}, j^{-1}(\log j)^{-2})} \abs{h(y)}\Big),
  \end{equation*}
  By our hypothesis~\eqref{coh}, the second sum is~$\ll \sum_{j\geq m/2} 2^{\Re(k) j /2}(\e^{j^{1-\delta}(\log j)^2} + \e^{O(m^{1-\delta^2})}) \ll 2^{-m/5}$, and therefore
  \begin{equation}
    f(x') = \sum_{0\leq j < m/2} w_j(x') + O\Big(2^{-m/5}\Big). \label{eq:approx-fem-hunbounded}
  \end{equation}

  We now let~$x', x'' \in V(m^{1+\delta}, m, x)$ be given. Taking differences in the expansion~\eqref{eq:approx-fem-hunbounded}, we obtain
  $$
  \abs{f(x') - f(x'')} \ll \sum_{0\leq j < m/2} \abs{w_j(x') - w_j(x'')} + 2^{-m/5}. $$
  Denote temporarily
  $ \Pi_j(y) := \prod_{0\leq i < j} T^i(y), $
  which is also~$u_j(y) / u_0(y)$ in the notations of Section~\ref{sec:notations}.
  By the definition~\eqref{eq:expr-w-prod},~\eqref{dwj} of $w_j$, splitting the difference in~$w_j(x') - w_j(x'')$ we get
  \begin{align*}
    \sum_{0\leq j < m/2} \abs{w_j(x') - w_j(x'')} &{} \leq \sum_{0\leq j < m/2} \bigg(\abs{(h_j(T^j(x')) - h_j(T^j(x''))) \Pi_j(x'')^{-k}} \\
    &{} \hspace{7em} + \abs{h_j(T^j(x'))(\Pi_j(x')^{-k} - \Pi_j(x'')^{-k}\Big)} \bigg).
  \end{align*}
  Regarding the first term inside the sum, the bound~\eqref{eq:bound-diff-Ti} and our hypothesis~\eqref{coh} applied with some value of~$\eps \asymp m^{-1-\delta}$ give~$\abs{h_j(T^j(x')) - h_j(T^j(x''))} = o(1)$ as~$m\to \infty$ uniformly for all~$j<m/2$. By the bound~\eqref{eq:bounds-uj}, we further have~$\Pi_j(y) \ll 2^{-j/2}$ uniformly for~$y\in[0, 1]$ and~$j\geq 0$, and therefore
  $$  \sum_{0\leq j < m/2} \abs{(h_j(T^j(x')) - h_j(T^j(x''))) \Pi_j(x'')^{-k}} \ll o(1) \times \sum_{j\geq 0} 2^{j \Re(k)/2} = o(1) $$
  as~$m \to \infty$.
  Regarding the second term, we have as above~$h_j(T^j(x')) \ll \e^{m^{1-\delta^2}} + \e^{j^{1-\delta} (\log j)^2}$ for all~$j\geq 0$. We use the inequality~$u^{-k}-v^{-k} \ll_k \abs{u-v}^{\min(1, -k)}$ valid for~$0\leq u, v \leq 1$, which gives in our case
  $$ \Pi_j(x')^{-k} - \Pi_j(x'')^{-k} \ll \abs{\Pi_j(x') - \Pi_j(x'')}^{\min(-k, 1)}. $$
  Splitting the difference, we have
  \begin{align*}
    \abs{\Pi_j(x') - \Pi_j(x'')} \leq {}& \sum_{0\leq i \leq j} \Pi_{i-1}(x') \Pi_{j-i-1}(T^{i+1}(x'')) \abs{T^i(x') - T^i(x'')} \\
    \ll {}& \sum_{0\leq i \leq j} 2^{-i/2} \times 2^{-(j-i)/2} \times 2^{-(m - i)/2} \\
    \ll {}& 2^{-j/3} 2^{-m/4}
  \end{align*}
  using our bound~\eqref{eq:bound-diff-Ti} and thus we conclude that
  \begin{align*}
    \sum_{0\leq j < m/2} \abs{h_j(T^j(x'))(\Pi_j(x')^{-k} - \Pi_j(x'')^{-k})} \ll {}& 2^{-m \min(-k, 1) /4} \sum_{j\geq 0} \frac{ \e^{O(m^{1-\delta^2})} + \e^{O(j^{1-\delta} (\log j)^2)}}{2^{j\min(-k, 1)/3}} \\
    \ll{}& 2^{-m \min(-k, 1)/4}
  \end{align*}
  which tends to~$0$ as~$m\to \infty$.
  Grouping our bounds, we deduce~$\Delta_{1+\delta}(m) \to 0$ when~$m\to\infty$, as claimed.
\end{proof}

We now turn to the proof of Theorem~\ref{thm:kneg-fdagger-ext}. Consider~$X = (\R\setminus\Q) \cap \gS$, where~$\gS$ is defined in Lemma~\ref{bll0}, and let~$x\in X$. In particular~$x\in \gT(B)$ for some~$B>0$, and therefore so do the convergents~$x_j$ of~$x$. Summarizing, we have~$x_j\in \Q\cap\gT(B)$ for all~$j$, and~$x_j \to x$ as~$j\to\infty$. Since~$f$ has the property~$\cS(1+\delta)$ by Lemma~\ref{lem:fem-cS}, the  limit~$\fem(x) := f^*(x) = \lim_{j\to\infty} f(x_j)$ in~\eqref{eq:def-fext-unbound} exists.

Moreover, for any~$\eps$, by Lemma~\ref{bll0} we can find~$B>0$ such that~$X_\eps := (\R\setminus\Q) \cap \gT(B) \subset X$ satisfies~$\meas(X_\eps\cap [0, 1])\geq 1-\eps$. This set is also trivially invariant by~$x\mapsto x+1$. Let~$y\in X_\eps$ and~$\eta>0$ be arbitrary. We pick an integer~$m\geq B^{1/(1+\delta)}$ such that the right-hand side of~\eqref{eq:unicont-fext-unbound} has modulus at most~$\delta$. Having chosen~$m$, and since~$y\not\in\Q$, we may find~$\xi>0$ such that any number~$x\in (\R\setminus\Q)\cap[y-\xi, y+\xi]$ satisfies~$b_j(x) = b_j(y)$ for~$j\leq m$. Now let~$x\in X_\eps \cap[y-\xi, y+\xi]$ be arbitrary. Then~$x\in \gT(B)$ by definition, and therefore the formula~\eqref{eq:unicont-fext-unbound} gives~$\abs{\fem(x) - \fem(y)} \leq \eta$. We have thus proven that~$\fem|_{X_\eps}$ is continuous at~$y$ with the restricted topology, as claimed.

\subsection{The case of $\Re(k)>0$.}

\subsubsection{Proofs of Theorems~\ref{ds1} and~\ref{ds1-ext}. The domain and continuity of $\fep$}\label{stds1}

For the purpose of studying the example of the Kontsevich function~\eqref{eq:def-kontsevich}, it will be convenient to generalize the period relations to
\begin{equation}
  \begin{cases}
    f(x+1) = \vartheta f(x), & (x\in \Q), \\
    h(x) = f(x) - \vartheta^{\pm 3} \abs{x}^{-k} f(-1/x), & (\pm x \in \Q_{>0}),
  \end{cases}\label{eq:recip-f-general}
\end{equation}
where~$\vth$ is some fixed root of unity. Define
\begin{align}\label{vtj}
  \vth_j = \vth_j(b_1, b_2, \dotsc) =  \begin{cases} \vth^{\sum_{i=1}^j (-1)^i b_i} & (j\text{ even}), \\ \vth^{3+\sum_{i=1}^j (-1)^i b_i} & (j\text{ odd}). \end{cases} 
\end{align}
where~$b_1, b_2, \dotsc$ are integers. By arguing in the same way as~\eqref{ffff2}, we find for~$x = [0; b_1, \dotsc, b_r]$ that
\begin{equation}
  f(x) = \sum_{j=0}^{r-1} \vth_j\, \Big(\frac{u_j}{u_0}\Big)^{-k} h\Big((-1)^j \frac{u_{j+1}}{u_j}\Big) + \vth_r u_0^k f(0).\label{eq:expansion-f-h-general}
\end{equation}
Note that by the period relation, we have
$$ h((-1)^r) = (\vth^{(-1)^r} - \vth^{3(-1)^r+(-1)^{r+1}}) f(0), $$
which confirms that the expression~\eqref{eq:expansion-f-h-general} holds for~$x = [0; b_1, \dotsc, b_r]$ regardless of whether this expansion is canonical ($b_r>1$ if~$r>1$) or not. Similarly to~\eqref{ffff3}, we may then work with~$r$ odd, for which we observe that
$$ \vth_{r-j}(b_1, \dotsc, b_r)\vth_{j}(b_r, \dotsc, b_1) = \vth_{r}(b_1, \dotsc, b_r). $$
Changing~$j$ to~$r-j$ in the sum~\eqref{eq:expansion-f-h-general}, and using the notation~\eqref{eq:def-vj}, we deduce that for~$x = [0; b_1, \dotsc, b_r]$, $r$ odd, we have
\begin{equation}
  \vth_r^{-1}(x) q^{-k} f(x) = \Psi(\bar{x})\label{eq:rel-f-Psi}
\end{equation}
where now~$\Psi(y)$ is defined for~$y = [0; c_1, \dotsc, c_r]$, $r$ odd, as
\begin{equation}
  \Psi(y) = \sum_{j=1}^r \vth_j^{-1}(y) v_j(y)^{-k} h\Big((-1)^j \frac{v_{j-1}(y)}{v_j(y)}\Big) + f(0),\label{eq:def-psi-general}
\end{equation}
where, whenever the expansion~$x = [0;b_1, \dotsc, b_r]$ is clear from the context, we denote
\begin{equation}
  \vth_j(x) = \vth_j(b_1, \dotsc),\label{eq:def-vthj-numbers}
\end{equation}
but this quantity in fact depends on the tuple~$(b_j)$ representing~$x$.

In the case of $\Re(k)>0$ and $h(x)=O(|x|^{-\Re(k)})$ for $x\in[-1,1]\setminus\{0\}$, the existence of $\fep$ for all $x\in\R$ is straightforward. Indeed, by~\eqref{eq:expansion-f-h-general} we have 
\begin{align}\label{eq:def-Psi-proof-kpos}
  \fep(x)=\begin{cases}
    \Psi( x) & \text{if }x\in\Q, \\
    {\displaystyle \lim_{\Q\ni y\to x}\Psi( y)} & \text{if }x\notin\Q,
  \end{cases}
\end{align}
where the limit exists by virtue of the fact that the sum in~\eqref{eq:expansion-f-h-general} is uniformly convergent and for each~$j$,~$v_j(y)$ depends only on finitely many of the functions~$y \mapsto b_j(y) = \floor{1/T^{j-1}(y)}$ which are locally constant at irrationals. Note that for irrational $x=[0;c_1,c_2,\dots]$ we simply have
\begin{equation}\label{fsdp}
  \fep(x)=\sum_{j=1}^{\infty}\vth_j^{-1}v_{j}^{-k} h\bigg((-1)^{j-1}\frac{v_{j-1}}{v_{j}}\bigg)+f(0).
\end{equation}
The continuity of $\fep$ on $\R\setminus\Q$ when $h(x)=O(|x|^{-\Re(k)})$ is immediate from~\eqref{fsdp}. Suppose $x=[0;b_1,\dots,b_r]\in\Q$ with $r$ odd. If $x'\to x^-$ with $x'$ sufficiently close to $x$, then we have $x'=[0;b_1,\dots,b_r,b',b_{r+2}',\dots]$ with $b'\to\infty$. It follows immediately from the definition that $\fep(x')\to \fep(x)$ if $h(x)=o(|x|^{-\Re(k)})$ as $x\to0$, and so under this assumption~$\fep$ is continuous at~$x$ from the left.
Next, let $x'\to x^+$, then if $b_r>1$ we have $x'=[0;b_1,\dots,b_r-1,1,b',b'_{r+2},\dots]$ with $b'\to\infty$. Note that
$$ \frac{v_{r-1}(x)}{v_{r}(x)} = \overline x, \quad \frac{v_{r-1}(x')}{v_r(x')} = \frac{v_{r-1}(x)}{v_r(x)-v_{r-1}(x)} = \frac{\overline x}{1-\overline x}, \quad \frac{v_{r}(x')}{v_{r+1}(x')}=1-\overline x. $$
Then, under the hypothesis~$h(x)=o(|x|^{-\Re(k)})$, by~\eqref{eq:def-psi-general} and the relations~\eqref{eq:recip-f-general}, we obtain
\begin{align*}
  \fep(x')-\fep(x) = {}& \vartheta_r^{-1} q^{-k} \bigg(\vth^{-1} (1-\bar x)^{-k} h\Big(\frac{\overline x}{1-\overline x}\Big) + \vth h(\bar x-1) - h(\bar x)\bigg) + o(1) 
  \\
  = {}& \vartheta_r^{-1} q^{-k} \bigg( \vth^{-1}(1-\bar x)^{-k}f\Big(\frac{\bar x}{1-\bar x}\Big) + \vth h(\bar x - 1) - f(\bar x) \bigg)+o(1)
  \\
  = {}& o(1)
\end{align*}
as $b'\to\infty$.
If $b_r=1$, then $x'=[0;b_1,\dots,b_{r-1}+1,b',b'_{r+1},\dots]$ with $b'\to\infty$. We then obtain in a similar way
\begin{align*}
  \fep(x')-\fep(x) = {}& \vartheta_r^{-1}q^{-k}\bigg( \vth h(\bar x - 1) - h(\bar x) - \vth^2 (\bar x)^{-k} h\Big(\frac{\bar x - 1}{\bar x}\Big) \bigg) + o(1)   \\
  = {}& o(1)
  \end{align*}
which goes to~$0$ as~$x'\to x^+$, and concludes the proof that~$\fep$ is continuous on~$\R$.
This proves the first part of Theorem~\ref{ds1}.

Assume now that~$f$ is not (twisted) periodic, but instead satisfies $f(x) = \vartheta f(x+1)$ only for $x \in \Q\setminus[-1, 0]$, and that~$h$ satisfies $h(x)=O(\e^{|x|^{-1+\delta}})$ for some $\delta>0$. In this situation, we use Proposition~\ref{prop:fext-unbound-continuity} to extend~$\Psi$.

\begin{lemma}\label{lem:fep-cS}
  Let~$\delta>0$, and assume that~$f$ satisfies
  \begin{equation*}
    \begin{cases}
      f(x+1) = \vartheta f(x), & (x\in \Q\setminus[-1, 0]), \\
      h(x) = f(x) - \vartheta^{\pm 3} \abs{x}^{-k} f(-1/x), & (\pm x \in \Q_{>0}),
    \end{cases}
  \end{equation*}
  and that the map~$h$ defined through~\eqref{mcg} satisfies~$h(x) = O(\e^{|x|^{-1+\delta}})$.
  Then the formulae~\eqref{eq:rel-f-Psi} holds with
  $$ \Psi(y) := \sum_{j=1}^r \vth_j^{-1}(y) v_j(y)^{-k} h\big((-1)^j \frac{v_{j-1}(y)}{v_j(y)}\big) + \vth^{-1} f(-1) \qquad (y=[0;c_1, \dotsc, c_r]; r\text{ odd}), $$
  and the map~$\Psi$ has the property~$\cS(1+\delta)$.
\end{lemma}
\begin{proof}
  The verification that the equations~\eqref{eq:rel-f-Psi}, \eqref{eq:def-psi-general} hold, with~$f(0)$ replaced with~$\vth^{-1}f(-1)$, is straightforward.
  Let~$B = m^{1+\delta}$, $x\in \gT(B)$ and~$x', x'' \in V(B, m, x)$. Note in particular that~$b_j(x') = b_j(x'')$ for all~$j\leq m$, and therefore also~$v_j(x') = v_j(x'')$ and~$\vth_j(x') = \vth_j(x'')$. We deduce
  $$ \abs{\Psi(x') - \Psi(x'')} \leq \sum_{j=m+1}^{r(x')} v_j(x')^{-\Re(k)} \abs{h\big((-1)^j \frac{v_{j-1}(x')}{v_j(x')}\big)} + \sum_{j=m+1}^{r(x'')} v_j(x'')^{-\Re(k)} \abs{h\big((-1)^j \frac{v_{j-1}(x'')}{v_j(x'')}\big)}. $$
  Since~$x'\in\gT(B)$, we also have~$v_{j-1}(x') / v_j(x') \gg 1/b_j(x') \gg j^{-1}(\log j)^{-2} + B^{-1}$. The same holds for~$x''$. By our hypothesis on~$h$, and combining this with~\eqref{eq:growth-vj}, we deduce
  $$ \abs{\Psi(x') - \Psi(x'')} \ll \sum_{j\geq m+1} 2^{-j \Re(k)/2} \big(\e^{O(j^{1-\delta}(\log j)^2)} + \e^{O(m^{1-\delta^2})}\big) \ll 2^{-m \Re(k)/3}, $$
  which tends to~$0$ as~$m\to \infty$, uniformly in~$x, x', x''$. This shows the property~$\cS(1+\delta)$ for~$\Psi$.
\end{proof}

Letting~$\fep(x) := \Psi^*(x)$ with the notations of Proposition~\ref{prop:fext-unbound-continuity}, the deduction Theorem~\ref{ds1-ext} follows by arguments identical to those of Section~\ref{exsc}, which were used to prove Theorem~\ref{thm:kneg-fdagger-ext}.

\subsubsection{Proof of Theorem~\ref{ds1}. The differentiability of $\fep$.}

We assume here the hypotheses of Theorem~\ref{ds1}, in particular, that~$h(x)=O(|x|^{-\Re(k)})$ for $0<|x|<1$.
Let $X = \gS$ be as in Lemma~\ref{bll0}, let~$\eps>0$ and suppose $x=[0;b_1,b_2,\dots] \in \gS$. Let $m := m(x,\eps)$ be as in Section~\ref{dmf}. Let $x'\in(0,1)$ be such that $|x-x'|<\eps$ so that $x'=[b_1,\dots,b_m,b_{m+1}',\dots]$. Let $n\geq m$ be the least integer such that $b_{n+1}\neq b_{n+1}'$.
We write $x=[0;b_1,\dots,b_n,z]$ and $x'=[0;b_1,\dots,b_n,z']$ with 
$z=[b_{n+1};b_{n+2},\dots]$, $z'=[b'_{n+1};b'_{n+2},\dots]$ (and $b_{n+1}\neq b_{n+1}'$ by hypothesis).
By standard properties of continued fractions we have
\begin{align*}
  |x-x'|=\frac{|z-z'|}{(z v_n+v_{n-1})(z' v_n+v_{n-1})}.
\end{align*}
This is $\gg b_{n+1}^{-2}v_n^{-2} $, unless  $b_{n+1}-b_{n+1}'$ is equal to $1$ or $-1$. In the first case we have
$z-z'=1+\frac1{b_{n+2}+y}-\frac1{b_{n+2}'+y'}\geq b_{n+2}^{-1}$ for some $0\leq y,y'\leq1$.
In the second we have $z'-z=1-\frac1{b_{n+2}+y}+\frac1{b_{n+2}'+y'}\geq y/2\geq -(b_{n+3}+1)^{-1}$.
Since~$x\in \gS$, in all cases we have $\eps>|x-x'|\gg_x v_{n}^{-2} n^{-3}(\log n)^{-6} \gg v_n^{-2} (\log v_n)^{-4}$.

Finally, we have 
\begin{align*}
  \fep(x)-\fep(x')\ll \sum_{j\geq {n+1}}\frac{1}{v_{j-1}^{\Re(k)}}+\sum_{j\geq {n+1}}\frac{1}{v_{j-1}^{\prime\,\Re(k)}}
  \ll v_n^{-\Re(k)} \ll_x \eps^{\Re(k)/2}|\log \eps|^{\Re(k)}
\end{align*}
and so it follows that~$\fep$ is locally~$\alpha$-Hölder continuous at~$x$, for any~$\alpha<\frac12\Re(k)$. This completes the proof of Theorem~\ref{ds1}.

\section{Limiting distributions of quantum modular forms}

\subsection{Convergence in distribution}

\subsubsection{Almost sure stability of the CF expansion}

First we need the following lemmas.
\begin{lemma}\label{flgaq}
  Let $q\geq2$, let
  $$ \gA_q:=\{0\leq a<q\mid (a,q)=1: a/q \in \gT((\log q)(\log\log q)^2)\}. $$
  Then $|\gA_q|=\varphi(q)(1+o(1))$ as $q\to\infty$, where $\varphi$ denotes Euler's totient function.
\end{lemma}
\begin{proof}
  This is a special case of~\cite{Rukavishnikova2006}.
\end{proof}
\begin{lemma}\label{flgaq2}
  Let $q\geq2$, let 
  $$ \gT_{q}:= ([0,1)\setminus\Q) \cap \gT((\log q)(\log\log q)^2).$$
  Then $\meas(\gT_q)=1+o(1)$ as $q\to\infty$.
\end{lemma}
\begin{proof}
  This is a consequence of Lemma~\ref{bll0} with~$B=(\log q)(\log\log q)^2$.
\end{proof}

We can now deduce the following two lemmas which show that $\fep$ and $\fem$ can be well approximated by their values at suitable closeby rationals. 
We will obtain this, here again, as a consequence of the property~$\cS(\lambda)$ for some~$\lambda>1$.

\begin{lemma}\label{lem:fstar-stab-rational}
  Suppose that there exists~$\lambda>1$ such that~$f:\Q\to\C$ has the property~$\cS(\lambda)$. Then the map~$f^*$ defined in Proposition~\ref{prop:fext-unbound-continuity} satisfies
  \begin{equation}
    \label{eq:fstar-stab-rational}
    \begin{aligned}
      f^*(a/q) &= f^*(x) + o(1) \qquad \forall a\in \gA_q, \quad \forall x\in (\tfrac {a-q^{1/4}}q, \tfrac {a+q^{1/4}}q) \cap (\gT_q \cup \tfrac1q \gA_q).
    \end{aligned}
  \end{equation}
  as~$q\to\infty$, uniformly for~$a$ and~$x$ within the stated sets.
\end{lemma}
\begin{proof}
  Let~$B := (\log q)(\log\log q)^2$. We assume first~$x\in(a/q, (a+q^{1/4})/q)$, and we let~$m:=2\lceil B^{1/\lambda}\rceil+1$, so that in particular~$m=(\log q)^{1/\lambda+o(1)}$ as~$q\to\infty$. Let~$a\in \gA_q$, denote~$x' := a/q$ and let~$x\in(a/q, (a+q^{1/4})/q)$ with either~$x \cap \gT_q$, or~$x=b/q$ where~$b\in \gA_q$. We wish to prove that
  $$ f^*(x) - f^*(x') = o(1) $$
  as~$q\to\infty$, with a rate of decay depending at most on~$h$. This follows from Proposition~\ref{prop:fext-unbound-continuity} if we can prove that~$x, x'\in \gT(B)$, $r(x) \geq m$ and~$b_j(x) = b_j(x')$ for~$j\leq m$. The fact that~$x, x' \in \gT(B)$ follows by definition. Write~$x' = a/q = [0; b_1, \dotsc, b_r]$. Note that~$r\ll \log q$ by Euclid's algorithm, so that by definition of~$\gT(B)$ and our choice of~$B$, we have~$b_j \ll (\log q)(\log\log q)^2$. By induction (see the proof of Theorem~31 in~\cite{Khintchine1963}), we have~$q^{1/r} \leq 2 (b_1 \dotsc b_r)^{1/r} \ll (\log q)(\log \log q)^2$, and therefore~$r \gg (\log q)/\log\log q$. We deduce that~$r(x') \geq m$ for~$q$ large enough. Finally, let~$a^* / q^* := [0, b_1, \dotsc, b_m]$ be the~$m$-th convergent of~$a/q$. Then we have~$(q^*)^{1/m} \ll (\log q)(\log\log q)^2$ as above, and therefore~$\log(q^*) \ll m \log\log q = o(\log q)$ as~$q\to\infty$. In particular~$q^*<q^{1/3}$ for~$q$ large enough. Since~$m$ is odd, we deduce
  $$ \frac{a^*}{q^*} - \frac{a}{q} > \frac1{(q^*)^2(b_{m+1}+2)} \gg \frac1{q^{2/3}(\log q)^2} \geq \frac{q^{1/4}}q $$
  for~$q$ large enough. In particular, any~$x\in (\frac aq, \frac{a+q^{1/4}}q)$ satisfies~$a/q<x<a^*/q^*$, and it follows that~$b_j(x) = b_j(a/q)$ for~$j\leq m$. This concludes the proof.

  The remaining case~$x\in((a-q^{1/4})/q, a/q)$ follows by an identical argument taking~$m$ even instead of odd.
\end{proof}

\subsubsection{Proof of Theorems~\ref{dless0} and~\ref{dless1}: The convergence in distribution}

Let~$\delta>0$. We now assume that either~$\Re(k)<0$ and~$h$ satisfies~\eqref{coh}, or~$\Re(k)>0$ and~$h$ satisfies~$h(x) = O(\e^{\abs{x}^{-1+\delta}})$. Let $f^*$ be either $\fem$ or $\fep$, depending on the sign of~$\Re(k)$. By Lemmas~\ref{lem:fem-cS} and \ref{lem:fep-cS}, the map~$f^*$ satisfies the conclusion of Lemma~\ref{lem:fstar-stab-rational}.

We will show that for any function $G\in\mathcal C_c^\infty(\C)$ we have
\begin{equation*}
  \CR := \frac1{\varphi(q)} \sum_{\substack{a=0 \\ (a, q)=1}}^{q-1} G(f^*(a/q))=\int_0^1 G(f^*(x))\,\df \nu+o(1)
\end{equation*}
as $q\to\infty$. 

For $0\leq j < q^{3/4}$, let $I_j:=[j q^{1/4},(j+1)q^{1/4}]\cap [0,q].$ If $j<q^{3/4}-1$, by a standard estimate~\cite[eqs.~(1.13), (4.11)]{Fouvry1991}, we have
\begin{equation}\label{irre}
\sum_{a\in \gA_q\cap I_j}1\leq   \ssum{(a, q)=1 \\ a\in I_j} 1 = \frac{\varphi(q)}{q}q^{1/4}+ O(q^{1/10 }).
 \end{equation}
Given $\eps\in(0, 1)$, we let $E_{\pm}=\{j \mid  \pm ( \#  I_{j}\cap  \gA_q-\varphi(q)q^{-3/4}) > \eps  \varphi(q)q^{-3/4}\}$ and $E=E_{+}\cup E_-$. 
By~\eqref{irre} it is clear that $E_+$ is empty if we assume $q$ is large enough. Also, summing~\eqref{irre} over $j\notin E_{-}$ we deduce from Lemma~\ref{flgaq} that $\# E=\# E_{-}=O(\eps q^{3/4})$. Now, for each $j\notin E$ we fix any $y_j\in I_{j}\cap \gA_q$.
Since $\|G\|_\infty, \|G'\|_\infty<\infty$, by Lemma~\ref{flgaq} and equations~\eqref{irre} and~\eqref{eq:fstar-stab-rational}, we have
\begin{equation*}
\begin{split}
  \CR {}& = \frac1{\varphi(q)} \sum_{a\in \gA_q} G(f^*(a/q)) + o(1) \\
  &{}= \frac1{\varphi(q)} \ssum{0\leq j \leq q^{3/4}\\ j\notin E} \sum_{a\in I_j\cap \gA_q} G(f^*(a/q)) + o(1)+O(\eps)\\
  &{}= \frac1{\varphi(q)} \ssum{0\leq j \leq q^{3/4}\\ j\notin E} G(f^*(y_j)) \sum_{a\in I_j\cap \gA_q}  1+ o(1)+O(\eps).
\end{split}
\end{equation*}
For $j\notin E$ the inner sum is $\frac{\varphi(q)}{q}q^{1/4} (1+O(\eps))$ and thus by~\eqref{eq:fstar-stab-rational} and Lemma~\ref{flgaq2} we deduce
\begin{align*}
  \CR {}  &{}= \frac1{q^{3/4}} \ssum{0\leq j \leq q^{3/4}\\ j\notin E} G(f^*(y_j))   + o(1)+O(\eps)\\
  &{}= \ssum{0\leq j \leq q^{3/4}\\ j\notin E}  \int_{(\frac1qI_j)\cap \gT_q}G(f^*(x)) \df x + o(1)+O(\eps)\\
  &{}=  \int_{ \gT_q}G(f^*(x)) \df x + o(1)+O(\eps)\\
  &{}=  \int_{ [0,1]}G(f^*(x)) \df x + o(1)+O(\eps),
\end{align*}
since $\meas(\cup_{j\in E}(\frac1qI_j))=O(\eps)$. Letting $\eps\to0$ sufficiently slowly we obtain the claimed result.

This proves the first parts of Theorems~\ref{dless0} and \ref{dless1}.

\subsection{Proof of Theorem~\ref{dless0}: Diffuseness when $\Re(k)<0$}

We now focus on the case~$\Re(k)<0$ and aim to show the second part of Theorem~\ref{dless0}.

\begin{proposition}\label{prop:cdf-cont-kneg}
  Let $k$, $f$, $\fem$ and $h$ as in Theorem~\ref{thm:kneg-fdagger-ext} and let~$\phi:\C\to\R$ be a linear form. Assume further that~$h$ is real analytic on $(-1,1)\setminus\{0\}$.
  Assume that for some choice of sign~$\pm$, there exists a set of positive measure $C \seq {\pm \R_{>0}}$ and a constant $\alpha \in\R$ such that  $\phi(\fem(x))=\alpha $ for $x\in C$. Then $\phi(\fem(x))=\alpha$ for almost all $x\in \pm \R_{>0}$.
\end{proposition}

This clearly implies the second part of Theorem~\ref{dless0} and concludes its proof.

\begin{proof}[Proof of Proposition~\ref{prop:cdf-cont-kneg}]
  We start with the case $k\in\R$. Since the real part of a real-analytic function is real-analytic, we have that $\phi\circ f$ is a quantum modular form of weight $k$ and real-analytic period function $\phi\circ h$. In particular we can assume that $f$ is real valued and that $\phi$ restricted to $\R$ is the identity function.

  By Theorems~\ref{ds} and~\ref{thm:kneg-fdagger-ext}, there exists a set $X$ of full measure with $\fem$ well defined on $X$. Note that the equation~\eqref{mcg} holds on $X\setminus\{0\}$ with $f$ replaced by $\fem$, by taking the limit along convergents and using the definition~\eqref{dfd-X}.
  We also recall that in our proof of Theorem~\ref{thm:kneg-fdagger-ext}, the set~$X$ was taken to be~$X=\gS$ from Lemma~\ref{bll0}. It is clear that~$\gS$ is closed under any fixed inverse branch of the Gauss map. We may thus assume that~$X$ has this property.

  We assume $C\seq \R_{>0}$, the other case being analogous. Also, by periodicity we can assume $ C\seq [0,1]\cap X$. For any $\eps>0$, by Lebesgue's density theorem we can find an interval $I\seq (0,1)$ such that $\meas(C\cap I)\geq (1-\eps) \meas(I)$. It follows that there exists an even $g\in\N$ and an inverse branch~$U_\bb$ of~$T^g$ given by~$U_\bb(y) = [0; b_1, \dotsc, b_g+ y]$ for some coefficients~$\bb = (b_1, \dotsc, b_g)\in\N^g$, such that 
  $D := \{x\in [0,1]\cap X\mid U_\bb(x)\in C\}$
  has measure $1-O(\eps)$. Repeatedly applying the reciprocity formula to $U_{\bb}(x)$ for any~$x\in D$, we have
  \begin{equation}\label{ffret}
    \alpha= \fem(
    U_{\bb}(x)) = \sum_{j= 0}^{g-1} \Big(\prod_{i=0}^{j-1} T^i(U_{\bb}(x))\Big)^{-k} h((-1)^j T^j(U_{\bb}(x)))+\Big(\prod_{i=0}^{g-1} T^i(U_{\bb}(x))\Big)^{-k}\fem (x).
  \end{equation}
  Note that~$T^i\circ U_\bb$ is a smooth map, and in fact a homography, for any~$i\leq g$.
  In particular, solving for $\fem(x)$, we obtain that  $\fem$ is given on $D$ by a function which is real analytic on $(0,1)$. For $\eps$ sufficiently small we must have $\meas(C\cap D)\geq \meas(C)-\meas([0,1]\setminus D)>0$, and thus by analytic continuation this implies that $\fem(x)=\alpha$ for all $x\in D$. Since $\eps>0$ is arbitrary we conclude that $\fem(x)=\alpha$ for a subset of $ (0,1)$ of measure $1$.  

  Now, let us assume $k\not\in\R$ and write $k=k_1+ik_2$ with $k_i\in\R$, $k_2\neq0$. We assume $\phi(z)=\Re(z)$, the general case being analogous. We let $\alpha\in\R$ such that $\Re(\fem(x))=\alpha$ for a set $C\seq[0,1]$ of positive measure. As in the case $k\in\R$, given $\eps>0$ we can find $\bb\in\N^g$ such that $D:=\{x\in [0,1]\cap X\mid U_\bb(x)\in C\}$ has measure $1-O(\eps)$. Since $(0,1]$ can be written as the disjoint union $(0,1]=\cup_{c\geq 1}U_{(c)}((0,1])$ we have
  \begin{align*}
    \sum_{c\ge1}\meas(\{x\in U_c((0,1])\mid U_\bb(x)\notin C\})=1-\meas( D)=O(\eps)
  \end{align*}
  In particular, $\meas(\{x\in U_c((0,1])\mid U_\bb(x)\notin C\})=O(\eps)$ for all $c\in\N$ and thus also $\meas(\{x\in (0,1]\mid U_{(\bb,c)}(x)\notin C\})=O(K^2\eps)$ for all $c\leq K$. Letting $D_c := \{x\in [0,1]\cap X\mid U_{(\bb,c)}(x)\in C\}$ we then have $\meas(D_c)=1-O(\sqrt[3]{\eps})$ for all $c\leq K:=\eps^{-1/3}$. 

  Next, we write the formula~\eqref{ffret} with $\bb_1=(\bb,c_1)$ and $\bb_2=(\bb,c_2)$ in place of $\bb$ for some $c_1,c_2\leq K$.
  Taking the real part of the resulting equations, for $x \in D':=D_{c_1}\cap D_{c_2}$ we find
  \begin{align}\label{ffret2}
    \alpha &= A_j(x)+\Re (B_j(x)^{k}\fem (x)) ,\qquad j=1,2
  \end{align}
  where for $j=1,2$ we have that $B_j(x)=\Big(\prod_{i=0}^{g} T^i(U_{(\bb,c_j)}(x))\Big)^{-1}$ and $A_j$ is real analytic on $(0,1)$. Notice that $B_j(x)=v_{g-1}+(c_j+x)v_{g}$ where $v_{g-1}$ and $v_g$ are the partial denominators of $[0;b_1,\dots,b_{g-1}]$ and $[0;b_1,\dots,b_{g}]$ repectively. In particular, if $c_1,c_2\in[\sqrt K, K]$ then $B_2(x)/B_1(x)={c_2}/{c_1}+O(\sqrt[6]\eps)$ uniformly in $x\in[0,1]$.

  Writing $L_j(x):=k_2\log(B_j(x))$, the equations~\eqref{ffret2} become
  \begin{align*}
    \alpha &= A_1(x)+B_j(x)^{k_1}\Re(\fem (x))\cos(L_j(x))-B_j(x)^{k_1}\Im(\fem (x))\sin(L_j(x)),\qquad j=1,2,\ x\in D'
  \end{align*}
  This is a system in $\Re(\fem (x)),\Im(\fem (x))$ of determinant 
  $B_1(x)^{k_1}B_2(x)^{k_1}\sin(L_1(x)-L_2(x))$ and solving for $\Re(f(x))$ we find
  \begin{align}\label{neqrf}
    \Re(\fem(x))=\frac{(\alpha-A_2(x)) B_1(x)^{k_1} \sin(L_1(x))-(\alpha-A_1(x)) B_2(x)^{k_1} \sin(L_2(x))}{B_1(x)^{k_1}B_2(x)^{k_1}\sin(L_1(x)-L_2(x))}
  \end{align}
  on $x\in D'$. 
  Now, for $\eps$ sufficiently small, we have 
  \begin{align*}
    \sin(L_1(x)-L_2(x))=\sin(k_2\log(c_2/c_1)+O(\sqrt[6]\eps))
  \end{align*}
  uniformly for $x\in[0,1]$. Thus, for $\eps$ sufficiently small we can find $c_1,c_2\in \N\cap [\sqrt K,K]$ such that $\sin(L_1(x)-L_2(x))$ does not vanish on $(0,1).$ For this choice one has that the right hand side of~\eqref{neqrf} is analytic on $(0,1)$. Since $\meas(D')=1-O(\sqrt[3]\eps)$, we then reach the conclusion as in the case $k\in\R$.
\end{proof}

\subsection{The continuity of the cumulative distribution function when $\Re(k)>0$}\label{sec:cont-cumul-distr}

We prove the following generalization of Theorem~\ref{dless1}. Also, as in~Section~\ref{stds1}, we allow for a twist by a root of unity $\vartheta$ and assume $f$ satisfies~\eqref{eq:recip-f-general}.

\begin{theorem}\label{gdless1}
  Let~$\phi:\C\to\R$ be a linear form, $\Re(k)>0$, and let $h:\R\setminus\{0\}\to\R$ be a function satisfying either of the following two conditions.
  \begin{enumerate}[label=($\text{\arabic*}\textsuperscript{*}$), ref=($\text{\arabic*}^*$)]
    \item\label{it:cond-kpos-1} The function $h$ extends to a continuous function on $[-1,1]$, and either~$\Im k = 0$ and there exists~$p\in\Z$, $\alpha\in(-1, 1)$ such that~$\phi(\vth^p h(\alpha)) \neq 0$; or~$\Im k \neq 0$ and there exists~$\alpha\in(-1, 1)$ such that~$h(\alpha)\neq 0$.
    \item\label{it:cond-kpos-2}
    One has $\lim_{x\to 0^\pm}|h(x)|=\infty$ for a choice of $\pm$, there exists $\delta\in(0,1)$ such that $h(x)\ll \e^{|x|^{-1+\delta}}$ for $|x|\leq1$, and
    \begin{itemize} 
      \item either $\Im(k)=0$ and there exists $p\in\Z$ such that for all small~$\eps>0$,
      \begin{equation}
        \liminf_{x\to 0^\pm} \inf\Big\{\bigg|\phi\bigg(\vartheta^p  \frac{h(x)-\uu ^k h(\uu x)}{|h(x)|} \bigg)\bigg|: \abs{\log \uu} \in (\eps, \eps^{2/3})\Big\} > 0;\label{eq:hypo-2star-real}
      \end{equation}
      \item or $\Im(k)\neq0$ and there exists $\pphi\in\R$ such that for all small~$\eps>0$,
      \begin{equation}
        \liminf_{x\to 0^\pm} \inf\Big\{\bigg|\phi\bigg(\e(\alpha) \frac{ h(x)- \uu^{k} h(\uu x)}{|h(x)|} \bigg)\bigg|: |\log \uu| \in (\eps, \eps^{2/3}),  |\alpha-\pphi| < \eps^{2/3} \Big\} > 0. \label{eq:hypo-2star-complex}
      \end{equation}
    \end{itemize}
  \end{enumerate}
  Then the cumulative distribution function of $(\phi\circ \fep)_\ast(\df \nu)$ is continuous. 
\end{theorem}
\begin{remark}\label{rmk:relax-hypo-h}
  \begin{itemize}
    \item The hypotheses could be relaxed somehow. For instance, instead of~\ref{it:cond-kpos-1}, it suffices to ask that~$h$ has finite left- and right-limits at certain rationals along with a non-vanishing condition. This will be clear from our arguments. One could also deal with some cases when the limits~\eqref{eq:hypo-2star-real} and~\eqref{eq:hypo-2star-complex} are zero, as long as one can control the asymptotic of these quantities as $x\to 0$. We refer to Remark~\ref{nlim} below, as well as to the proof of Corollary~\ref{cordedek} when $a=0$ for an example where this consideration is relevant.
    \item When~$h$ is continuous on~$[-1, 1]$, the condition~\ref{it:cond-kpos-1} is clearly necessary, since otherwise~$\phi \circ \fep$ vanishes identically on~$[0, 1]$.
    \item Some condition of the type~\ref{it:cond-kpos-2} is necessary in order to prevent examples such as~$h(x) = 1 - \abs{x}^{-k}$, for which the function~$f$ is constant.
  \end{itemize}
\end{remark}

\begin{proof}[Proof of the second part of Theorem~\ref{dless1}]
  Assume first that~$h$ is continuous on~$[-1, 1]$ and non-zero. Then the hypothesis~\ref{it:cond-kpos-1} is clearly satisfied with either~$\phi=\Re$ or~$\phi=\Im$. We deduce that~$(\phi \circ \fep)_\ast(\df\nu)$ is diffuse, and therefore so is~$\fep_\ast(\df\nu)$.

  Assume next that $h(x)\ll\e^{\abs{x}^{-1+\delta}}$ on~$[-1, 1]\setminus\{0\}$ and~$h(x) \sim c x^{-\lambda}$ as~$x\to 0^+$. Then for~$\abs{\log \uu} \leq 1$, we deduce that as $x\to0^+$,
  $$ \frac{h(x) - \uu^k h(\uu x)}{\abs{h(x)}} = \frac{c}{\abs{c}} (1+o(1)) (1 - \uu^{k-\lambda}) \gg_{\eps,k,\lambda} 1 $$
  if additionally~$\abs{\log \uu} \geq \eps$. Similarly as above, we deduce that hypothesis~\ref{it:cond-kpos-2} holds for~$\phi=\Re$ or~$\phi=\Im$ and we deduce that~$\fep_\ast(\df\nu)$ is diffuse.

  This shows that the second part of Theorem~\ref{dless1} holds, and concludes its proof.
\end{proof}

\begin{remark}
  When~$h(x) = c x^{-\lambda}$ with~$\lambda$ satisfying~$\Re(k)>0$ but~$\lambda\not\in\R$, hypothesis~\ref{it:cond-kpos-2} no longer holds, but it is plausible that our arguments could be adapted by further localizing the values of~$x$ involved in~\eqref{eq:hypo-2star-real}--\eqref{eq:hypo-2star-complex}. We do not pursue this here.
\end{remark}

We first give the following lemma.

\begin{lemma}\label{lem:abscont-close}
  Let~$F:[0, 1]\to \R$ be measurable. Assume that for each~$\eps>0$, we can find a countable collection~$(S_j)_j$ of disjoint measurable subsets of~$[0, 1]$, such that
  $$ \sum_j \meas(S_j) \geq 1 - \eps, $$
  and for each~$j$ and~$y\in S_j$,
  \begin{equation}
    \meas(\{x \in S_j, F(x)=F(y)\}) \leq \eps \meas(S_j).\label{eq:hypo-xy-close}
  \end{equation}
  Then~$F_*(\df\nu)$ is continuous.
\end{lemma}
\begin{proof}
  Let~$a\in\R$, and~$X=F^{-1}(\{a\})$. Let~$\eps>0$ be arbitrary, and~$(S_j)_j$ given by the hypothesis.
  For each~$j$, either~$S_j\cap X$ is empty, or there exists~$y\in S_j$ such that~$F(y)=a$ and then
  $$ \meas(X\cap S_j) = \meas(\{x \in S_j, F(x) = F(y)\}) \leq \eps \meas(S_j) $$
  by hypothesis. In both cases we have~$\meas(X \cap S_j) \leq \eps \meas(S_j)$. Summing over~$j$, we find
  $$ \meas(X) \leq \meas([0, 1]\setminus \cup_j S_j) + \sum_j \meas(X \cap S_j) \leq 2\eps, $$
  and letting~$\eps\to0$ gives the conlusion.
\end{proof}

We divide the proof depending on whether \ref{it:cond-kpos-1} or \ref{it:cond-kpos-2} holds. 

\subsubsection{The case of bounded $h$.}\label{sec:cont-cdf-hbounded}

Let~$N\in\Z_{>0}$ be the order of~$\vth$ as a root of unity, so that
$$ \vth^N = 1, $$
where~$\vth$ is the automorphy factor in the generalized period relation~\eqref{eq:recip-f-general}.
We recall that $\vartheta_j$ was defined in~\eqref{vtj}. Recall also from~\eqref{eq:def-vthj-numbers} that the notation~$\vth_j(x)$ for a number~$x\in\Q$ depends on how we expand~$x$ in CF when~$x$ is rational. In what follows, we will work with the expansion of odd length.
For each~$g\geq 1$, we define~$e_g(x) \in \Z/N\Z$ through
$  \vth_g(x) = \vth^{e_g(x)}$ 
and notice that by definition we have
\begin{equation}\label{eq:rel-qgx}
  e_g(x)\equiv e_{g-1}(x)+(-1)^g b_g+3(-1)^{g+1}\pmod{N}.
\end{equation}

Given~$c_1, \dotsc, c_\ell\geq 1$, we denote
\begin{equation}\label{notj}
  \begin{split}
    J(c_1, \dotsc, c_\ell) := {}& \{x\in[0, 1], b_i(x) = c_i \text{ for } 1\leq i \leq \ell \}, \\
    \mu(c_1, \dotsc, c_\ell) := {}& \meas(J(c_1, \dotsc, c_\ell)).
  \end{split}
\end{equation}
By~\cite[p.57]{Khintchine1963}, we have
\begin{equation}
  \mu(c_1, \dotsc, c_\ell) = \frac1{v_\ell(x)(v_\ell(x) + v_{\ell-1}(x))} \asymp \frac1{v_\ell(x)^2}, \qquad (x = [0; c_1, \dotsc, c_\ell]).\label{eq:value-mu}
\end{equation}

\begin{lemma}\label{lemdis}
  Let $m\in \N$, $K>1$ and $(L_1,\dots,L_m)\in\N^m$ be fixed. Let~$V\geq1$, and for each~$e\in\Z/N\Z$ and~$\omega>0$, let $\I(e, \omega) \subset \Z\cap[\sqrt{V}, K\sqrt{V}]$ be given, and
  $$ T := \inf_{e\in\Z/N\Z,\, \omega>0} \# \I(e, \omega). $$
  Also, for~$x\in(0, 1)\setminus\Q$, let
  $$
  G_V(x):=\{g\in [V,2 V] \mid b_g(x) \in \I(e_{g-1}(x), v_{g-1}(x)),\, b_{g+i}(x) =L_i\ \forall i=1,\dots,m\}.
  $$
  and let $X_{V, \I} = X_{V,\I}(K, L_1, \dotsc, L_m)$ be the subset of $[0,1]\setminus\Q$ such that $G_V(x)$ contains at least an even and an odd integer.
  Then~$\meas(X_{V,\I})=1+o(1)$ as $V, T \to \infty$,  where the rate of decay of~$o(1)$ depends at most on~$K, L_1, \dotsc, L_m$.
\end{lemma}
\begin{proof}
  We show that $G_V(x)$ contains an even integer for $x$ in a set of measure $1+o(1)$, the odd integer case being analogous.

  By~\cite[eq.~(57)]{Khintchine1963} we have uniformly for~$\ell\geq 1$ and~$c_1, \dotsc, c_{\ell+1}\geq 1$,
  \begin{align}\label{meas}
    \tfrac13 c_{\ell+1}^{-2}\mu(c_1,\dots,c_{\ell})\leq \mu(c_1,\dots,c_{\ell+1})\leq 2 c_{\ell+1}^{-2}\mu(c_1,\dots,c_{\ell}).
  \end{align}
  In particular, writing $L:=\max(L_1,\dots, L_m)$, we have for any~$c\in\Z_{>0}$
  $$ \mu(c_1,\dots,c_\ell,c,L_1,\dots,L_m)\geq 3^{-m-1} c^{-2}L^{-2m}\mu(c_1,\dots,c_{\ell}).$$ 
  Since $\sum_{c\in\I(e_\ell, v_\ell)}c^{-2}\geq \frac1{K^2V}\#{\I(e_\ell, v_\ell)}$, where~$e_\ell = e_\ell([0;c_1, \dotsc, c_\ell])$, it then follows that the measure of $x\in [0,1]\setminus\Q$ such that none of the integers $g\in\{2\floor{V/2} + \ell (2m+2) \mid 1\leq \ell \leq V /(2m+2)\}$ satisfy $b_g(x) \in \I$ is
  $$
  \ll (1- K^{-2}(3L)^{-3m}TV^{-1})^{\floor{V/(2m+2)}} = o(1)
  $$
  as $V, T\to\infty$ with~$m, L$ fixed. The Lemma then follows.
\end{proof}

We now prove Theorem~\ref{gdless1}. To illustrate remark~\ref{rmk:relax-hypo-h}, we assume only that~$h$ is bounded on~$[-1, 1]$, not necessarily continuous, and that it has finite left- or right-limit at~$0$. We assume it is the former, the latter case being analogous, as we will comment after the proof. If $\Im(k)=0$, we suppose that~$\phi(\vth^p h(0^-))\neq0$ for some~$p\in\Z$, whereas if $\Im(k)\neq0$ we assume $h(0^-)\neq0$ and set $p:=0$.
Moreover, we let
\begin{equation}
  \kappa\in\R: \qquad \phi(\e(-\kappa) \vth^p h(0^-)) \text{ is maximal.}\label{eq:cdf-kpos-nonzero}
\end{equation}
This means, in other terms, that for any~$z\in\C$,
\begin{equation}
  \phi(\vth^p h(0^-) z) = \abs{h(0^-)} \Re(e(\kappa) z).\label{eq:rel-phi-kappa}
\end{equation}
If $\Im(k)=0$, then setting $z=1$ we see that $\kappa\not\equiv\pm\frac12\mod 1$ by hypothesis.

We define our choice of sets~$\I = \I(e, \omega)$. We let~$\xi\in(0,1)$ be a parameter, on which these sets will depend. For $e\in\Z/N\Z$, $\omega>0$, let $K=2$ if $\Im(k)=0$ and $K=\max(2,e^{3\pi/\abs{\Im(k)}})$ otherwise. Also, let $  \I(e, \omega) = \I_{V, \xi, N, k}(e, \omega)$ be defined by
\begin{equation}
  \begin{aligned}
    \I(e, \omega)  := {}& \Bigg\{b\in\Z\cap [\sqrt V, K\sqrt V] \colon 
    b\equiv 3 - 
    (p+e)\pmod{N}, \\
    & \quad \text{and }
    \begin{cases}
      -\xi \leq \{\tfrac{\Im (k)}{2\pi}\log (\omega\sqrt V)-\kappa\} + \tfrac{\Im (k)}{2\pi}\log (b/\sqrt V) \leq \xi & \text{if } k\notin\R, \\
      0 \leq \log (b/\sqrt V) \leq \xi & \text{if } k\in\R
    \end{cases}\Bigg\},
  \end{aligned}\label{eq:def-cI}
\end{equation}
where~$\{\cdot\}$ denotes the fractional part, and~$\kappa$ was defined at~\eqref{eq:cdf-kpos-nonzero}.
Here our choice of $K$ ensures that the set~$\I$ is not void for~$V$ large enough, and in fact
\begin{equation}
  \#\I \gg_{k, N} \xi \sqrt{V}\label{eq:lowbound-cI}
\end{equation}
for~$V$ large enough in terms of~$k, N$ and~$\xi$. Moreover, on the one hand, by~\eqref{eq:rel-qgx} the congruence condition ensures that whenever a number~$x\in[0,1]\setminus\Q$ satisfies~$b_g(x) \in \I(e_{g-1}(x), v_{g-1}(x))$ with $g$ even, we have $e_g(x) \equiv -p\mod N$,
and therefore we deduce
\begin{equation}
  \vth_g(x) = \vth^{-p}, \qquad (g\in G_V(x)\cap 2\Z).\label{eq:equal-vthgx}
\end{equation}
On the other hand, the condition involving~$\|\cdot\|_{\R/\Z}$ ensures that
\begin{equation}
  \Re(\e(\kappa) (v_{g-1}(x) b_g(x))^{-i\Im k}) = 1 + O(\xi^2) \qquad (\Im k \neq 0, g\in G_V(x)) \label{eq:approx-vgxImk}
\end{equation}
by the Taylor expansion of the cosine, with an absolute implicit constant.

We let $L_1,\dots,L_m\in\N$ to be chosen later and take $X_{V,\I}$, $G_V(x)$ as in Lemma~\ref{lemdis}.
For all even~$g$ and all $c_1, \dotsc, c_{g-1}\in \N$ we let
\begin{equation}
  S_{g, (c_i), \I} := \{x \in (0,1)\setminus\Q: \forall i<g, b_i(x) = c_i; g\in G_V(x)\cap2\Z; \forall j\geq 1, g-2j \not\in G_V(x)\},\label{eq:def-Jgc}
\end{equation}
that is $S_{g, (c_i), \I}$ is the subset of $ J(c_1,\dots,c_{g-1})$ such that $g$ is the smallest even element of $G_V(x)$. We will drop the subscript~$\I$ from the notation.
\begin{lemma}\label{lem:mes-xy-close}
  Assume that~$L_i < \sqrt{V}$ for all~$1\leq i \leq m$ and that $g\in2\Z\cap[V,2V]$.  For all~$c_1, \dotsc, c_{g-1},c \geq 1$ and~$\eps>0$, we have
  $$ \meas\big(\big\{ x\in S_{g, (c_i)} \colon \abs{b_g(x) - c} \leq \eps \sqrt{V} \big\}\big) \ll \eps \xi^{-1} \meas(S_{g, (c_i)}) $$
  where the constant depends on~$m$, $k$ and~$N$ at most.
\end{lemma}
\begin{proof}
  We start by observing that the condition $g-2j \not\in G_V(x)$ in the definition of $S_{g, (c_i), \I}$ reduces to a condition on the first $g-1$ partial quotients of $x$, $c_1,\dots, c_{g-1}$, if~$j>m/2$. In particular, we can assume this condition is satisfied, since otherwise $S_{g, (c_i)} = \emptyset$ and the result is trivial. Similarly, we can assume $c\in [(1-\eps)\sqrt V,(K+\eps)\sqrt V]$
  
  Next, we observe that if~$1\leq j<m/2$, then the condition $g-2j \not\in G_V(x)$ follows from the truth of the other conditions in the definition of~$S_{g, (c_i)}$. Indeed, for~$g\in G_V(x)$ we have~$b_g(x)\in\I$ and~$b_{g+i}(x) = L_i$ for~$1\leq i \leq m$, so that $b_{g-2j+2j}(x)=b_g(x) \geq \sqrt{V} > L_{2j}$ and thus $g-2j\notin G_V(x)$. Recalling the notation~\eqref{notj}, we thus have 
  $$
  \{x\in S_{g, (c_i)} \colon b_g(x)=c'\}= J(c_1,\dots,c_{g-1},c',L_1,\dots,L_m)
  $$
  for all $c' \in \I(e_{g-1}(x'), v_{g-1}(x'))$ with $x'=[0;c_1,\dots, c_{g-1}]$.  By~\eqref{eq:value-mu}, for all~$c' \in[(1-\eps)\sqrt V,(K+\eps)\sqrt V]$, we have
  $$ \mu(0; c_1, \dotsc, c_{g-1}, c', L_1, \dotsc, L_m) \asymp \mu(0; b_1, \dotsc, b_{g-1}, c, L_1, \dotsc, L_m) $$
  with a constant depending at most on~$m$ and $K$. 
  Then, using~\eqref{eq:lowbound-cI} we deduce on the one hand
  $$ \meas(S_{g, (c_i)}) \gg \xi \sqrt{V} \mu([0; c_1, \dotsc, c_{g-1}, c, L_1, \dotsc, L_m]). $$
  and on the other hand
  $$ \meas\big(\big\{ x\in S_{g, (c_i)}, \abs{b_g(x) - c} \leq \eps \sqrt{V}\big\}\big) \ll \eps \sqrt{V} \mu([0; c_1, \dotsc, c_{g-1}, c, L_1, \dotsc, L_m]). $$
  Grouping these two bounds yields our claim.
\end{proof}

Let $\eps\in(0,1)$ be fixed. We set $m=1$ and~$L_1:=L\in\N_{>0}$ be a parameter. We let $X_{V,\I}$ as in Lemma~\ref{lemdis} and assume that $L \to \infty$ as $V\to\infty$ sufficiently slowly with respect to $V$, so that we still have $\meas(X_{V,\I})\to1$ as well as $L<\sqrt V$. In particular, assuming~$V\geq 1$ is large enough, we have
\begin{equation*}
  \meas(X_{V,\I}) \geq 1-\eps.
\end{equation*}
For $g\in [V, 2V]\cap 2\Z$ and~$c_1, \dotsc, c_{g-1}\geq 1$ the sets~$(S_{g, (c_i)})$ are disjoint and, by construction, their union over all $g$ and $(c_i)$ contains~$X_{V,\I}$ and thus has measure  $\geq1-\eps$. The collection of sets~$(S_{g, (c_i)})$ will play the r\^ole of~$(S_j)$ in Lemma~\ref{lem:abscont-close}.

Recall that, since~$h$ is assumed to be bounded, we have the expression~\eqref{fsdp}.
Consider one of the sets~$S_{g, (c_j)}$, and~$x, y\in  S_{g, (c_j)}$ satisfying~$\phi(\fep(x))= \phi(\fep(y))$.
Since the terms~$j<g$ in~\eqref{fsdp} depend only on~$c_1, \dotsc, c_{g-1}$, we have~$\fep(x) - \fep(y) = F_g(x) - F_g(y)$, where
\begin{equation}
  F_g(x) = \sum_{j=g}^\infty \vth_j^{-1}(x) v_j(x)^{-k} h\Big((-1)^{j-1}\frac{v_{j-1}(x)}{v_j(x)}\Big) \label{fga1}
\end{equation}
and similarly for~$y$. If $x \in S_{g, (c_j)}$ we have $v_{g+j}(x)\gg v_{g}(x) L\, 2^{j/2}$ for $j\geq1$, whence since $g$ is even and $h$ is bounded we have
\begin{align*}
  F_g(x) & {}=\frac{\vth_g(x)^{-1}+o(1)}{v_g(x)^{k}}h(0^-) +O\bigg(\frac{L^{-\Re(k)}}{ v_g(x)^{\Re(k)}}\sum_{j\geq1}\frac {\|h\|_\infty}{2^{\Re(k)j/2}}\bigg)\\
  &{} =\frac{\vth_g(x)^{-1}+o(1)}{v_g(x)^{k}}h(0^-),\numberthis\label{fga}
\end{align*}
as $V \to \infty$.

Consider first the case~$\Im k = 0$. By~\eqref{eq:equal-vthgx}, we have
\begin{equation}
  \vth_g(x) =  \vth_g(y) = \vth^{-p}.\label{eq:equal-vthp}
\end{equation}
Also, since~$b_j(x)=b_j(y)$ for~$j<g$ and~$b_g(x)\asymp b_g(y)\asymp \sqrt V$ we have $v_{g-1}(x)\asymp v_{g-1}(y)$ and $v_{g}(x)\asymp v_{g}(y)$ with constants depending only on~$k$. Thus, by~\eqref{eq:rel-phi-kappa} and the mean value theorem we deduce that for~$\Im k = 0$ we have
\begin{align*}
  0 = \abs{\phi(F_g(x) - F_g(y))} ={}&  |h(0^-)\cos(2\pi \kappa)||v_g(x)^{-k} - v_g(y)^{-k}| + o(v_g(x)^{-\Re k}) \\
  \gg {}& |h(0^-)\cos(2\pi \kappa)|\frac{|v_g(x) - v_g(y)|}{v_g(x)^{k+1}}+ o(v_g(x)^{-\Re k}) \\
  \gg {}& |h(0^-)\cos(2\pi \kappa)|\frac{\abs{b_g(x) - b_g(y)} V^{-1/2} + o(1)}{v_g(x)^{ k}}.
\end{align*}
We deduce that $ \abs{b_g(x) - b_g(y)} = o(\sqrt{V}) $ as~$V\to \infty$.

We reach a similar conclusion when~$\Im k \neq 0$. Indeed, in this case we still have~\eqref{eq:equal-vthp}, with $p=0$.
Furthermore, since~$v_g(x) = v_{g-1} b_g(x) + v_{g-2}(x)$ we have
$$ \phi(h(0^-) v_g(x)^{-i \Im k}) = \abs{h(0^-)} \Re\big(\e(\kappa) (v_{g-1}b_g(x))^{-i\Im k} \big)(1 + O(V^{-1/2})) = 1 + O(\xi^2 + V^{-1/2}), $$
where in the last equality we used~\eqref{eq:approx-vgxImk}. Similarly to the above, as~$V \to \infty$, we have
\begin{align*}
  0 ={}& \abs{\phi(F_g(x) - F_g(y))}\\
  ={}& \abs{\phi(\vth_g(x)^{-1} h(0^-) v_g(x)^{-i \Im k}) v_g(x)^{-\Re k} - \phi(\vth_g(y)^{-1} h(0^-) v_g(y)^{-i \Im k}) v_g(y)^{-\Re k}} + o(v_g(x)^{-\Re k}) \\
  = {}& \abs{h(0^-)} \abs{v_g(x)^{-\Re k} - v_g(y)^{-\Re k}} + (o(1) + O(\xi^2)) v_g(x)^{-\Re k} \\
  \gg {}& \abs{h(0^-)} \frac{\abs{b_g(x) - b_g(y)} V^{-1/2} + o(1) + O(\xi^2)}{v_g(x)^{\Re k}}.
\end{align*}
As above, we deduce
$$ \abs{b_g(x) - b_g(y)} = (o(1) + O(\xi^2)) \sqrt{V}. $$

We now pick~$\xi = c \sqrt{\eps}$ with a sufficiently small constant~$c>0$ and assume $V$ is large enough, so that the right-hand side above is most~$\eps\sqrt{V}$ in modulus. We get
$$ \meas(\{x \in S_{g, (b_i)}, F(x) = F(y) \}) \leq \meas\big(\big\{x\in S_{g, (b_i)}, \abs{b_g(x) - b_g(y)} \leq \eps \sqrt{V}\big\}\big) \ll \sqrt{\eps} \meas(S_{g, (b_i)}) $$
by Lemma~\ref{lem:mes-xy-close}. Since~$\eps>0$ was arbitrary, this verifies the hypothesis~\eqref{eq:hypo-xy-close}, thus Lemma~\ref{lem:abscont-close} applies and yields the desired conclusion.

The case when the non-vanishing hypothesis involves~$h(0^+)$ is similar, taking~$g$ to be odd instead of even, and picking $b\equiv 3 + p + e \pmod{N}$ in~\eqref{eq:def-cI}.

\bigskip

Now, assume that $\alpha\in(-1, 1)$ satisfies Condition \ref{it:cond-kpos-1} of the Theorem. We may assume that~$\alpha\in\Q$ by continuity. We pick~$\alpha$ such that the length~$r(|\alpha|)$ of the continued fraction expansion of $|\alpha|$ is minimal.
We may assume that~$\alpha\neq 0$, or in other words~$r(|\alpha|)\geq 1$, since the complementary case was treated above. We write~$\abs{\alpha} = [0; c_1, \dotsc, c_r]$. By minimality of~$r$, we have~$h(\pm[0; c_s, \dotsc, c_r]) = 0$ for~$1<s\leq r$ if~$\Im k \neq 0$; and likewise~$\phi(\pm \vth^p h([0; c_s, \dotsc, c_r])) = 0$ if~$\Im k =0$. We repeat the arguments above with $m=r+1$,
$$ L_i=c_{r+1-i} \text{ for } 1\leq i\leq r, \qquad L_{r+1} = L $$
with $L\to\infty$ as $V\to\infty$. Note that if $b_g \asymp \sqrt{V}$ and $b_{g+j}=L_j=c_{r+1-j}$ for $1\leq j \leq r$, then
$$ \frac{v_{g+j-1}}{v_{g+j}} = [0; b_{g+j}, b_{g+j-1}, \dotsc, b_g + O(1)] \to [0;c_{r+1-j},\dots,c_r] $$
as~$V\to\infty$. Thus, if the parity of~$g$ was chosen such that~$(-1)^{g+r-1}=\sgn(\alpha)$, then by hypothesis we have
$$ \begin{cases}
  h((-1)^{g+j-1}v_{g+j-1}/v_{g+j})=o(1) & (j=1,\dots, r-1), \\
  h((-1)^{g+r-1}v_{g+r-1}/v_{g+r})=h(\alpha)+o(1).
\end{cases} $$
The rest of the argument follows \emph{mutatis mutandis}, with the estimate~\eqref{fga} being replaced by
$$ F_{g, N}(x) = \frac{\vth_{g+r}(x)^{-1}+o(1)}{v_{g+r}(x)^{k}}h(\alpha). $$
\begin{remark}
  By choosing the parity of~$g$ appropriately, it is clear that the arguments above hold under milder hypotheses, namely one-sided continuity of~$h$ at~$\alpha$ along with the vanishing of the values~$h(T^j(\alpha)^\pm)$ for~$1<j\leq r$.
\end{remark}

\subsubsection{The case of unbounded $h$.}\label{sec:cont-cdf-hunbounded}

The following lemma provides a substitute for Lemma~\ref{lemdis} in the case when~$h$ is unbounded.
\begin{lemma}\label{lemdis3}
  Let $K>1$, $\delta\in(0,1)$, and~$\rho>(2+2\delta)^{-1}$ be fixed,~$V\geq1$, and $\psi:\R_{\geq1}\to\R_{\geq1}$ be such that $\lim_{x\to\infty}\psi(x)=+\infty$. Also, for each~$e\in\Z/N\Z$ and~$\omega>0$, let $\I(e, \omega) \subset \Z\cap[\sqrt{V}, K\sqrt{V}]$ be given, and assume
  $$ T := \inf_{e\in\Z/N\Z,\, \omega>0} \#\I(e, \omega)\gg V^{\rho} $$
  Finally, for~$x\in[0, 1]\setminus\Q$, let
  $$
  G_V(x):=\{g\in [ V,2 V] \mid b_g(x)\in \I(e_{g-1}(x), v_{g-1}(x)),\,  b_{g+j}(x)<j^{1+\delta}\psi(V)\ \forall j\geq1\}.
  $$
  and $X_{V, \I}$ be the subset of $[0,1]\setminus\Q$ such that $G_V(x)$ contains at least an even and an odd integer. Then,~$\meas(X_{V,\I})=1+o(1)$ as $V \to \infty$,  where the rate of decay of~$o(1)$ depends at most on~$K,\psi,\delta$ and~$\rho$.
\end{lemma}
\begin{proof}
  We prove that $G_V(x)$ contains an even integer asymptotically almost surely, the odd case being identical.
  We fix $\eps>0$ and observe we can assume $1\leq\psi(V)< V^\eps$. 
  We then observe that by Lemma~\ref{bll0} one has that $b_{g+j}(x)<j^{1+\delta}\psi(V)$ for all $j> D:=V^{\frac1{1+\delta}+\eps}$, all $g\in[V,2V]$ and all $x$ in a subset of $[0,1)$ of measure $1+o(1)$ as $V\to\infty$. In particular, it suffices to show that the larger set
  $$ G_V'(x):=\{g\in [ V,2 V] \mid b_g(x)\in \I(e_{g-1}(x), v_{g-1}(x)),\,  b_{g+j}(x)<j^{1+\delta}\psi(V)\ \forall 1\leq j \leq D\} $$
  contains an even integer asymptotically almost surely.
  
  Now, for $m\geq0$ and $c_1,\dots,c_m\in\N$, $I\seq \N$ let
  $$
  M_{I} := \{x\in [0, 1]\setminus\Q \colon 2\Z\cap I \cap G_V'(x)\neq \emptyset\}, \qquad M_{I}(c_1, \dotsc, c_m) := M_I \cap J(c_1,\dots,c_m),
  $$
  so that we need to prove that $\meas(M_{[V,2V]})\to1$ as $V\to\infty$.
  For $m=\ell-1$  and any $\ell\in 2\Z\cap[V,2V]$, by~\eqref{meas} we have
  \begin{align*}
    \meas(M_{\{\ell\}}(c_1,\dots,c_m))\geq\mu(c_1,\dots,c_m)\frac{T}{3K^2 V}\prod_{j=1}^\infty \bigg(1-\frac1{3j^{1+\delta}\psi(V)}\bigg)\geq\mu(c_1,\dots,c_m)\frac{T}{4K^2 V}
  \end{align*}
  for $V$ large enough. Splitting into subintervals we then have that the same bound holds for any $m\leq \ell-1$.

  Next, we let $C:= [V^{1/(2+2\delta)-\eps}]\leq D$ and notice that if $0<\ell_2-\ell_1\leq C$, then $M_{\{\ell_1\}}(c_1,\dots,c_m) \cap M_{\{\ell_2\}}(c_1,\dots,c_m) = \emptyset$. Indeed, if $x\in M_{\{\ell_1\}}(c_1,\dots,c_m)$ then 
  $$b_{\ell_2}(x)\leq (\ell_2-\ell_1)^{1+\delta}\psi(V)\leq C^{1+\delta}V^\eps<\sqrt{ V}$$
  and thus $x\notin M_{\{\ell_2\}}(c_1,\dots,c_m)$. It follows that for any $W\in\N$ with $[W-C,W)\seq [V,2V]$ we have 
  \begin{align}\label{fbfm}
    \meas(M_{[W-C,W)}(c_1,\dots,c_m))=\sum_{\ell\in I\cap 2\Z} \meas(M_{\{\ell\}}(c_1,\dots,c_m))\geq \mu(c_1,\dots,c_m)\frac{T \floor{C/2}}{4K^2 V}
  \end{align}
  for any $m<W-C$.

  Let~$W\in\N$ with $V < W - 2D$ and~$W\leq 2V$, and let $m=W-C-1\geq W-D$. We observe that the condition $\ell\in G_V'(x)$ depends only on the first $\ell+D$ partial quotients of $x$. Thus, since $m\geq W-D$, we have that $M_{[V,W-2D)}(c_1,\dots,c_m)$ is either empty or is equal to $J(c_1,\dots, c_m)$. By~\eqref{fbfm}, we deduce
  \begin{align*}
    \meas(M_{[V,W)})-\meas(M_{[V,W-2D)}) &{} = \meas(M_{[W-2D,W)}\setminus M_{[V,W)}) \\
    &{} \geq \meas(M_{[W-C,W)}\setminus M_{[V,W-2D)})\\
    &{} =\sum_{c_1,\dots,c_m\geq 1\atop M_{[V,W-2D)}(c_1, \dotsc, c_m) = \emptyset}\meas(M_{[W-C,W)}(c_1,\dots,c_m))\\
    &{} \geq \frac{TC}{12K^2V}\sum_{c_1,\dots,c_m\geq 1\atop M_{[V,W-2D)}(c_1, \dotsc, c_m) =\emptyset}\mu(c_1,\dots,c_m) \\
    &{} = \frac{TC}{12K^2 V}(1-\meas(M_{[V,W-2D)})).
  \end{align*}
  We have used~$\floor{C/2}\geq C/3$ in the fourth line. We thus obtain
  \begin{align*}
    1-\meas(M_{[V,W)})\leq (1-\meas(M_{[V,W-2D)}))(1- {TC}/12V).
  \end{align*}
  Iterating we then obtain
  \begin{align*}
    1-\meas(M_{[V,2V)})\leq (1- {TC}/12K^2V)^{[V/2D]}=o(1)
  \end{align*}
  since $TC/K^2 D \asymp T V^{-1/(2+2\delta)-2\eps}\to\infty$ if $\eps$ is small enough.
\end{proof}

For $j\in\N$ and $K,\psi$ as in the lemma, we let
\begin{align*}
  Q_j(V)=\begin{cases}
    j^{1+\delta}\psi(V) & \text{if $j^{1+\delta}\psi(V)\notin[\sqrt V,K\sqrt V]$,}\\
    K\sqrt V &\text{otherwise.}
  \end{cases}
\end{align*}
We then define 
$$
G_V^*(x):=\{g\in [ V,2 V] \mid b_g(x)\in \I(e_{g-1}(x), v_{g-1}(x)),\,  b_{g+j}(x)<Q_j(V)\ \forall j\geq1\}
$$
and the corresponding set $X_{V,\I}$.
Clearly, $G_V^*(x)\geq G_V(x)$ and thus, under the hypothesis of Lemma~\ref{lemdis3}, $\meas(X^*_{V,\I})=1+o(1)$ as $V\to\infty$.

We let
$$ \kappa := \begin{cases} -\pphi, & \text{if } \Im k \neq 0, \\ 0, &\text{if }\Im k = 0, \end{cases} $$
and define the sets $\I$ as in~\eqref{eq:def-cI}. Also, we define $S^*_{g, (c_i)}=S^*_{g, (c_i), \I}$ as in~\eqref{eq:def-Jgc}, but with $G^*_V(x)$ in place of $G_V(x)$. The following Lemma is an analogue of Lemma~\ref{lem:mes-xy-close}.
\begin{lemma}
  Assume that $g\in2\Z\cap[V,2V]$ with $V\geq1$ sufficiently large.  For all~$c_1, \dotsc, c_{g-1},c \geq 1$ and~$\eps>0$, we have
  $$ \mu\big(\big\{ x\in S^*_{g, (c_i)} \colon \abs{b_g(x) - c} \leq \eps \sqrt{V} \big\}\big) \ll \eps \xi^{-1} \mu(S^*_{g, (c_i)}) $$
  where the constant depends on~$k$,~$\delta$ and~$N$ at most.
\end{lemma}
\begin{proof}
  We may assume $S^*_{g, (c_i)} \neq \emptyset$. Let~$x'\in S^*_{g, (c_i)}$. First we prove that for all $c'\in\I(e_{g-1}(x'), v_{g-1}(x'))$, we have
  \begin{equation}
    \{x\in S^*_{g, (c_i)} \colon b_g(x)=c'\} = \{x\in J(c_1,\dots,c_{g-1},c') \mid b_{g+j}(x) < Q_j(V),\ j\geq1\}.\label{eq:simplification-S}
  \end{equation}
  Any~$x$ in the left-hand side satisfies~$g\in G_V^*(x)$, and therefore~$b_{g+j}(x) < Q_j(V)$. The inclusion~$\subseteq$ follows trivially. Consider then a number~$x$ in the right-hand side. The condition~$b_g(x)=c'$ is evident, and in order to prove the inclusion~$\supseteq$, there remains to show~$x\in S_{g, (c_i)}^*$.
  Since~$x\in J(c_1, \dotsc, c_{g-1}, c')$, it suffices to prove that~$g\in G_V^*(x)$ and~$\forall j\geq 1, g-2j \not\in G_V^*(x)$. The condition~$g\in G^*_V(x)$ holds since~$c' \in \I(e_{g-1}(x'), v_{g-1}(x'))$ by hypothesis and~$(e_{g-1}(x), v_{g-1}(x)) = (v_{g-1}(x'), e_{g-1}(x'))$.
  
  Let next~$j\geq 1$; we wish to show that~$g-2j\not\in G^*_V(x)$.
  We assume that~$g-2j \in [V, 2V]$; it suffices to prove that for such~$j$, we have~$b_{g-2j}(x)\not\in \I(e_{g-2j-1}(x), v_{g-2j-1}(x))$ or that~$b_{g-2j+\ell}(x) \geq Q_\ell(V)$ for some $\ell\geq1$. We first prove that
  \begin{equation}
    b_{g-2j}(x') \not\in \I(e_{g-2j-1}(x'), v_{g-2j-1}(x')) \quad \text{or} \quad \exists\ell\in\{1, \dotsc, 2j\}, b_{g-2j+\ell}(x') \geq Q_\ell(V).\label{eq:cond-GV-xprime}
  \end{equation}
  Indeed, from the condition~$g-2j\not\in G^*_V(x')$, it suffices to prove that~$b_{g-2j+\ell}(x') < Q_\ell(V)$, if~$\ell>2j$, but this is immediate since the condition~$g\in G^*_V(x')$ implies~$b_{g-2j+\ell}(x') < Q_{\ell-2j}(V) \leq Q_\ell(V)$.
  Now we argue that the conditions~\eqref{eq:cond-GV-xprime} hold with~$x'$ replaced by~$x$. Since~$x, x'$ share the same first~$g-1$ CF coefficients, this is trivial except for the condition at~$\ell=2j$, which involves~$b_g$. It suffices to check that~$b_g(x) \geq Q_{2j}(V) \iff b_g(x') \geq Q_{2j}(V)$. By construction, we have~$Q_{2j}(V) \not\in[\sqrt{V}, K\sqrt{V}]$, so both conditions are in fact equivalent to~$Q_{2j}(V) \leq \sqrt{V}$. This shows that the conditions~\eqref{eq:cond-GV-xprime} hold with~$x'$ replaced by~$x$, and therefore~$g-2j\not\in G^*_V(x)$. This concludes the proof of the inclusion~$\supseteq$ in~\eqref{eq:simplification-S}.
  
  From there, the rest of the proof follows closely that of Lemma~\ref{lem:mes-xy-close}.
  By~\eqref{meas} and since~$Q_j(V) \geq j^{1+\delta} \psi(V)$, we deduce that 
  \begin{align*}
    \meas(\{x\in S^*_{g, (c_i)} \colon b_g(x)=c'\})&= \mu(c_1,\dots,c_{g-1},c')\prod_{j\geq1}\bigg(1+O\bigg(\frac1 {j^{1+\delta}\psi(V)}\bigg)\bigg)\\
    &= \mu(c_1,\dots,c_{g-1},c')(1+o(1))
  \end{align*}
  as $V\to\infty$. We conclude as for Lemma~\ref{lem:mes-xy-close}.
\end{proof}

The following technical lemma allows us to extract values of~$h$ which will give a dominant contribution to~$\fep$.

\begin{lemma}\label{slt}
  Let $\reta>1,C\geq1$ and $\delta\in(0,1)$. Let  $h:[-1,1]\setminus\{0\}\to\R$  be such that $h(x)\ll \e^{|x|^{-1+\delta}}$ for~$x\neq 0$, and assume $|h(x)|\to\infty$ as $x\to 0^{\pm}$ for a choice of $\pm$.  Finally, let $\xi:(0,1]\to (0,1]$ with $\lim_{x\to0^+}\xi(x)=0$. 
  Then there exist $\nu>0$ and $\psi:[1,\infty)\to\R_{>0}$ with~$\lim_{+\infty}\psi = +\infty$ such that  
  \begin{align}\label{slte}
    \sup_{(Cj^{-1-\delta} \psi(1/z))^{-1}<|x|<1}\frac{|h(x)|}{\reta^{j}}<\inf_{0<\pm y\leq \xi(z)}\frac{|h(y)|}{\psi (1/z)}
  \end{align}
  for all $j\geq1$ and all $z\in(0,\nu]$.
\end{lemma}
\begin{proof}
  Assume $|h(x)|\to\infty$ as $x\to 0^+$; the complementary case follows by changing~$h(x)$ to~$h(-x)$.
  We fix a function $\psi_1:(0,1]\to\R_{>0}$ going to $\infty$ at~$0^+$. By hypothesis we have $|h(x)|<\exp( C^{1-\delta} j^{1-\delta^2}\psi_1(z)^{1-\delta})$ for $z\in(0,\nu]$, $(C{j^{1+\delta}\psi_1(z)})^{-1}<|x|<1$ and $\nu>0$ sufficiently small.
  Now, for any $z\in(0,1]$ the maximum $\psi_2(z):=\max_{j\in\N}(\reta^{-j}\exp( C^{1-\delta} j^{1-\delta^2}\psi_1(z)^{1-\delta}))$ exists. Moreover, since $\psi_1(z)$ goes to infinity as $z\to 0^+$ then clearly so does $\psi_2(z)$. Also, by hypothesis we have $\inf_{0<y<\xi(z)}|h(y)|\to\infty$ as $z\to0^+$. It follows that we can ensure that $\psi_2(z)^2\leq \inf_{0<y\leq \xi(z)}|h(y)|$ for~$z\in (0, \nu]$ by taking $\psi_1$ that goes to infinity sufficiently slowly and~$\nu$ small enough. By construction, for $z\in(0,\nu)$ and all $j\geq 1$ we then have
  $$
  \sup_{(C{j^{1+\delta}\psi_1(z)})^{-1}<|x|<1}\frac{|h(x)|}{\reta^j}<\reta^{-j}\exp( j^{1-\delta^2}\psi_1(z)^{1-\delta})\leq\psi_2(z)\leq \frac{\inf_{0<y\leq \xi(z)}{|h(y)|}}{\psi_2(z)}.
  $$
  It is then sufficient to take $\psi(1/z):=\min(\psi_1(z),\psi_2(z)).$
\end{proof}

We shall now show that if~\ref{it:cond-kpos-2} is satisfied then the cumulative distribution function of $(\phi\circ \fep)_\ast(\df \nu)$ is continuous. We assume $\lim_{x\to 0^-}|h(x)|\to\infty$, the other case being analogous.

We take $R=2^{\Re(k)/4}$, $C=2K$ where~$K$ was defined at~\eqref{eq:def-cI}, and $\xi(w):=w^{-1/2}$, and we apply Lemma~\ref{slt}, thus finding a constant $\nu>0$ and a function $\psi(x)$ with the properties claimed in this Lemma. We then apply Lemma~\ref{lemdis3} with the same function~$\psi$.
As in the bounded case, for any $g\in [V, 2V]\cap 2\Z$ and~$c_1, \dotsc, c_{g-1}\geq 1$ the sets~$(S^*_{g, (c_i)})$ are disjoint and, given any fixed $\eps>0$, their union has measure $\geq1-\eps$, if~$V$ is large enough in terms of~$\eps$.

For any $x\in S^*_{g, (c_i)}$ and any $j\geq1$ we have $\frac{v_{g+j-1}}{v_{g+j}}\geq \frac 1{2b_{g+j}}\geq \frac{1}{2Q_j(V)}\geq \frac{1}{2K\psi(V) j^{1+\delta}}$, where we dropped the dependency on $x$ in the inequalities for easy of notation. Thus, assuming $V>1/\nu$, by~\eqref{slte} (with $z=1/V$) we have
$$
\frac{|h((-1)^{j-1}{v_{g+j-1}}/{v_{g+j}})|}{2^{\Re(k)j/2}}< \inf_{0< z\leq 1/\sqrt V}\frac{|h(-z)|}{2^{\Re(k)j/4}\psi(V)}\leq \frac{|h(-{v_{g-1}}/{v_g})|}{2^{\Re(k)j/4}\psi(V)},
$$
since $v_{g-1}/v_g\leq 1/b_g\leq 1/\sqrt V$. With $F_g(x)$ as in~\eqref{fga1} we then deduce the following analogue of~\eqref{fga}
\begin{align*}
  F_g(x) &=  \frac{\vartheta^p  h(-{v_{g-1}(x)}/{v_g(x)})}{v_g(x)^{k}}+O\bigg(\sum_{j\geq1}^\infty \frac{|h\big((-1)^{j-1}{v_{g+j-1}(x)}/{v_{g+j}(x)}\big)}{ v_g(x)^{\Re(k)}2^{\Re(k)j/2}}\bigg)\\
  &=\frac{\vartheta^p+o(1)}{v_g(x)^{k}}h(-{v_{g-1}(x)}/{v_g(x)}).
\end{align*}
We let $x,y\in S^*_{g, (c_i)}$ and write $\uu_g=\frac{v_{g}(x)}{v_{g}(y)}=\frac{b_{g}(x)v_{g-1}+v_{g-2}}{b_{g}(y)v_{g-1}+v_{g-2}}$ and $\omega_g:=v_{g-1}/v_{g}(x)$, so that we have
\begin{align*}
  F_g(x)-F_g(y) &=\vartheta^p\frac{ |h(-\omega_g)|}{v_g(x)^k}\bigg( \frac{h(-\omega_g)-\uu_g^kh(-\uu_g \omega_g)}{ |h(-\omega_g)|}+o(1)\bigg).
\end{align*}
By the definition of $\I$ we have 
$$
v_g(x)^{-i\Im(k)}=\e(\alpha_g)(1+o(1)),
$$
with $\alpha_g=\alpha_g(x)=0$ if $\Im(k)=0$ and otherwise~$\alpha_g(x)$ satisfying $|\alpha_g(x) - \pphi| \leq \xi$. Moreover, 
$\log \uu_g = \log (b_g(x)/b_g(y)) + o(1) \ll_k \xi$ and, if $|b_g(x)-b_{g}(y)|> \eps \sqrt V$, then also $|\log \uu_g|\geq \log(1+\eps/K) + o(1) \gg_k \eps$ for~$V$ large enough. Taking $\xi=\eps^{2/3}$ we then have
\begin{align}\label{reree}
  \phi(  F_g(x)-F_g(y)) &=\frac{ |h(-\omega_g)|}{|v_g(x)|^k} \phi\bigg(\vartheta^p\e(\alpha_g)\bigg( \frac{h(-\omega_g)-\uu_g^kh(-\uu_g \omega_g)}{ |h(-\omega_g)|}\bigg)+o(1)\bigg)\\
  &\neq0\notag
\end{align}
by hypothesis~\ref{it:cond-kpos-2}, provided that $|b_g(x)-b_{g}(y)|>  \eps \sqrt V$. We then conclude in the same way as we did in Section~\ref{sec:cont-cdf-hbounded}.

\begin{remark}\label{nlim}
  From the above proof, it is clear that the term~$o(1)$ in~\eqref{reree} can be replaced by~$O(1/\sqrt{V} + 1/\psi(V))$. This will be used in one special case of our applications to the cotangent sums.
\end{remark}

\begin{remark}\label{rationals}
  From the above proof, it is also immediate to see that we can modify the condition $(2^*)$, $k\not \in\R$, of Theorem~\ref{gdless1} as follows. Given $\beta\in\R$ and any small $\eps>0$, we denote by $\mathcal R=\mathcal R_{k,\eps,\beta}$ the set or reduced rationals $x=p/q$ with $|q^{-i\Im(k)}\e(-\beta)-1|\leq 10\eps^{2/3}.$ Then, we can weaken the condition $\lim_{x\to0^\pm}|h(x)|=\infty$ in $(2^*)$ by introducing the restriction $x\in \mathcal R$ in the limit\footnote{This amounts to observing that also Lemma~\ref{slt} holds under this weaker hypothesis, provided that one adds the condition $y\in \mathcal R$ in the $\inf$ on the right hand side of~\eqref{slte}.} and replace~\eqref{eq:hypo-2star-complex} by
  \begin{equation*}
    \liminf_{x\to 0^\pm\atop x\in \mathcal R} \inf\bigg\{\bigg|\phi\bigg(\e(\alpha) \frac{ h(x)- \uu^{k} h(\uu x)}{|h(x)|} \bigg)\bigg|: 
    \begin{aligned}& |\log u| \in (\eps, \eps^{2/3}),  |\alpha-\pphi| < \eps^{2/3},\ \\ 
      & \uu x \in \mathcal R, \num(\uu x)=\num(x)
    \end{aligned}\bigg\} > 0. 
  \end{equation*}
\end{remark}

\section{Arithmetic applications}

\subsection{Eichler integrals of holomorphic cusp forms}

\begin{proof}[Proof of Corollary~\ref{cormof}]
  The map~$f(x) = \tilde g(x)$ satisfies the relation~\eqref{mcg} with weight~$2-k$ and with~$h$ being a non-zero polynomial of degree at most~$k-2$, see~\cite[p.~273]{Eichler1957}. 
  Since~$h$ is Lipschitz and bounded on~$[-1, 1]$, the estimates~\eqref{coh} hold trivially. We may therefore apply Theorem~\ref{dless0}, which gives the claimed convergence in distribution. Notice that in this case $\extneg{\tilde g}$ coincides on the whole real line with $\tilde g$ as defined in~\eqref{eq:def-eichlerint}.

  Finally, assume~$(\extneg{\phi\circ \tilde g})_\ast(\df \nu)$ has an atom for some non-zero linear form $\phi$ and write $\phi(z)=\Re(\e^{i\theta}z)$ for some $\theta\in\R$ and all $z\in\C$. Then by Proposition~\ref{prop:cdf-cont-kneg} we obtain that~$\extneg{\phi\circ \tilde g}$ is constant on~$[0,1]$. By~\eqref{eq:def-eichlerint}, $\extneg{\phi\circ \tilde g}$ is given by the Fourier series
  \begin{align*}
    x\mapsto \sum_{n\geq1}\frac{\Re(e^{i\theta} a_n)}{n^{k-1}}\cos(nx)-\sum_{n\geq1}\frac{\Im(e^{i\theta} a_n)}{n^{k-1}}\sin(nx),
  \end{align*}
  which is constant on $[0,1]$ only if $\Re(e^{i\theta} a_n)=\Im(e^{i\theta} a_n)=0$, i.e. $a_n=0$, for all $n\in\N_{>0}$. Thus $g=0$ which was excluded by hypothesis.
\end{proof}

\begin{proof}[Proof of Corollary~\ref{corakd}]
  The map~$A_{k,D}$ defined in~\eqref{eq:def-AkD} coincides with~$F_{k+1, D}$ defined in~\cite[eq.~(15)]{Zagier1999}. Consider the even cusp form~$g = f_{k+1, D} \in S_{2k+2}(1)$ defined in~\cite[eq.~(53)]{Zagier1999}. By~\cite[eq.~(55)]{Zagier1999}, we have~$A_{k,D}(x) = c_{k,D} + \frac12 \tilde g(x)$ for all~$x\in\Q$ and some~$c_{k, D} \in \R$. Therefore Corollary~\ref{corakd} follows from the just proved Corollary~\ref{cormof}. 
\end{proof}

\subsection{Kontsevich's function}

\begin{proof}[Proof of Corollary~\ref{corvas}]
  Define~$h$ by formula~\eqref{rvas}. In~\cite[Theorem, p.~958]{Zagier2001}\footnote{The factors~$\zeta_8^{\pm3}$ there should be read~$\zeta_8^{\pm 1}$.}, it is proved that~$h\in C^\infty(\R)$ and that~$h$ is real-analytic except at~$0$, and moreover~$h(0) = 1$. Thus~$\varphi$ and~$h$ satisfy the hypotheses of the first part of Theorem~\ref{dless1}, with the generalized hypotheses~\eqref{eq:recip-f-general} (with~$\vth = \e(1/24)$) which were adopted in the proof. This proves the existence of the limiting distribution. The limit~\eqref{eq:def-extpos-kontsevich} arises from the relations~\eqref{eq:rel-f-Psi} and \eqref{eq:def-Psi-proof-kpos}.
  
  Let~$\theta \in \R$ and~$\phi: z \mapsto \Re(\e^{i\theta} z)$. To see that~$(\phi \circ \extpos{\varphi})_\ast(\df\nu)$ does not have atoms, we apply Theorem~\ref{gdless1}. The period function~$h$ is continuous on~$[-1, 1]$, and $h(0)=1$. Since~$\vth\not\in\R$, certainly one of~$\phi(1)$ or~$\phi(\vth)$ is non-zero, and therefore hypothesis~\ref{it:cond-kpos-1} is satisfied for some~$p\in\{0, 1\}$. Theorem~\ref{gdless1} applies and yields the desired conclusion.
  
  Finally, the continuity of $\extpos{\varphi}$ follows immediately from (the twisted version of) Theorem~\ref{ds1}.
\end{proof}

\subsection{Cotangent sums}\label{sec:cotangent-sums}

For~$b\in\Z$, $q\geq 1$ and~$(b, q) = 1$, the cotangent sums we are interested in are defined by
$$ c_a\Big(\frac bq\Big) := q^a \sum_{m=1}^{q-1} \cot\Big(\frac{\pi m b}{q}\Big) \zeta\Big(-a, \frac mq\Big), $$
see~\cite[p.~226]{Bettin2013a}. The special case~$a=-1$ corresponds to the classical Dedekind sums~\cite[p.~466]{Riemann1892}, while the case~$a=0$ corresponds to the cotangent sum from~\cite{BettinConrey2013}.
These sums were described as ``imperfect'' quantum modular forms, due to an irregular term arising in the period relation, namely the last term on the right-hand side of~\cite[eq.~(17)]{Bettin2013a}; see the discussion in Example~0 in~\cite{Zagier2010}. The point of the upcoming definition is that we can relax this lack of regularity at the cost of weakening the periodicity hypothesis to~\eqref{eq:f-weak-per}. Let
$$ \rho\Big(\frac bq\Big) = \begin{cases} \{\tfrac{\bar b}q\} & (q>1, (b, q)=1), \\ 1 & (q = 1, b>0), \\ 0 & (q=1, b<0). \end{cases} $$
By Bezout's theorem, we have for~$b\neq 0$, $q\geq 1$, $(b, q)=1$,
$$ \rho\Big(\frac bq\Big) - \rho\Big(\frac{-q}{b}\Big) = \frac1{bq}. $$
For~$x\in\Q$, let
$$ \ct_a(x) := c_a(x) + a\kappa_1(a) \den(x)^{1+a} \rho(x) \qquad (x\neq0),$$
extended arbitrarily at~$0$, and with $\kappa_1(a):=\frac{\zeta(1-a)}{\pi}$.
Note that~$c_a$ is~$1$-periodic, and that~$\rho$ satisfies the weak periodicity~\eqref{eq:f-weak-per}, from which we deduce that~$\ct_a$ also satisfies~\eqref{eq:f-weak-per}.

From~\cite[Theorem~4]{Bettin2013a}, we have
\begin{equation}
  c_a(x) - \abs{x}^{-1-a} c_a(-1/x) =h_a(x)-a\kappa_1(a)\frac{\den(x)^a}{\num(x)} \qquad (x \in \Q\setminus\{0\}),\label{eq:perrel-cta}
\end{equation}
and thus
\begin{equation}
  \ct_a(x) - \abs{x}^{-1-a} \ct_a(-1/x) =h_a(x) \qquad (x \in \Q\setminus\{0\}),\label{eq:perrel-cta}
\end{equation}
for $h_a(x) :=  - \sgn(x) i \zeta(-a) \psi_a(\abs{x})$ with $\psi_a(x)$ as in~\cite[Theorem~4]{Bettin2013a}. In particular, $h_a$ is real-analytic on~$\R_{\neq0}$. Also, by the same theorem, for $a\neq0$, we have\footnote{For $a\in\Z$ the result is obtained by continuity, cf. the explanation after Theorem~1 in~\cite{Bettin2013a}.}
\begin{equation}\label{ash}
  h_a(x) = \kappa_2(a)\sgn(x) |x|^{-1-a} + \kappa_1(a) x^{-1}+O_a(|x|)
\end{equation}
as~$x\to 0$, for $\kappa_2(a)=-\zeta(-a)\cot(\frac{\pi a}2)$, and where the error has to be replaced by $O(\abs{x}\log \abs{1/x})$ if $a=-2$. For $a=0$ we instead have
\begin{equation}\label{ash0}
  h_0(x) = -\frac{\log(2\pi |x|)-\gamma}{\pi x}+ O(|x|).
\end{equation}
We prove the distributional statements in Corollary~\ref{cordedek}, considering several cases depending on the value of~$a$. The statement in Corollary~\ref{cordedek} about the continuity of~$\extpos{c_a}$ when~$\Re(a)>0$ will follow immediately from Theorem~\ref{ds1} and the behaviour around~$0$ of the period functions.

\subsubsection{The case~$\Re(a)<-1$}
By~\cite[Theorem~1]{Bettin2013a}, it follows that both~$h_a$ and~$h_a'$ are bounded by~$\abs{x}^{O(1)}$ for all~$x\neq 0$. Thus, the conditions~\eqref{coh} are satisfied, and by Theorem~\ref{thm:kneg-fdagger-ext}, the function
\begin{equation}
  \extneg{\ct_a}(x) = \lim_{j\to \infty} \ct_a([b_0; b_1, \dotsc, b_j]), \qquad (x=[b_0;b_1, b_2, \dotsc])\label{eq:def-lim-ctaneg}
\end{equation}
exists for~$x$ in a set~$X\subset \R$ of full measure.
By Theorem~\ref{dless0}, we deduce that the multisets
$$ \{\ct_a(\tfrac bq),\ 0<b<q, (b, q) = 1\} $$
become equidistributed, as~$q\to\infty$, according to~$(\extneg{\ct_a})_*(\df \nu)$. Since
\begin{equation}
  \ct_a(x) - c_a(x) = O(\den(x)^{1+\Re(a)}) = o_a(1)\label{eq:ca-cta-close}
\end{equation}
as~$\den(x) \to \infty$, uniformly in the numerator~$\num(x)$, it follows that the same conclusion holds for the multisets
$$ \{c_a(\tfrac bq),\ 0<b<q, (b, q) = 1\}. $$

Next we prove that if $a\in\R_{<-1}$, then $(\extneg{\ct_a})_*(\df\nu)$ is diffuse on~$\R$. Assume, for the sake of contradiction, the existence of~$\lambda\in\R$ and~$C\subset X\cap\R_{>0}$ of positive Lebesgue measure such that~$\extneg{\ct_a}(x) = \lambda$ for all~$x\in C$. By Proposition~\ref{prop:cdf-cont-kneg}, we deduce that~$\extneg{\ct_a}(x) = \lambda$ for almost all~$x>0$. Using the fact that~$c_a$ is odd, which transfers to~$\extneg{\ct_a}$ almost everywhere by~\eqref{eq:ca-cta-close},  we obtain by the period relation~\eqref{eq:perrel-cta} that $h_a(x)=\lambda (1+ \abs{x}^{-1-a} )$ for almost all~$x>0$. However, this contradicts~\eqref{ash} as~$x\to 0^+$, regardless of the value of~$a$.

Now, let $a\notin\R$ and let~$\phi:\C\to\R$ a non-zero linear form. We assume by contradiction that $(\phi\circ\extneg{\ct_a})_*(\df\nu)$ is not diffuse. As above we deduce that there exists ~$\lambda\in\R$ such that~$\phi(\extneg{\ct_a}(x)) = \lambda$ for almost all~$x>0$. Composing the period relation~\eqref{eq:perrel-cta} with~$\phi$, we obtain 
$\phi(h_a(x))=  \lambda - \phi( \abs{x}^{-1-a} \extneg{\ct_a}(-1/x)) . $
We pick~$y\in(0, 1)$ so that~$\extneg{\ct_a}(-y)$ and~$\extneg{\ct_a}(1/(n+y))$ are defined in~\eqref{eq:def-lim-ctaneg} for all~$n\in\N_{>0}$, and~$\phi(\extneg{\ct_a}(1/(n+y))) = \lambda$. Taking $x=1/(n+y)$, by periodicity we then obtain 
\begin{align}\label{tcas}
  \phi(h_a(x))=\lambda - \phi( \abs{n+y}^{1+a} \extneg{\ct_a}(-y)).
\end{align}
Now,  by~\cite[Theorem~1]{Bettin2013a}\footnote{We remark that there is a typo in~\cite[eq.~(5)]{Bettin2013a}, as the minus sign in front of the sum should be removed.} the asymptotic in~\eqref{ash} can be extended to
\begin{equation}\label{ash3}
  h_a(x) = \frac{\kappa_2(a)\sgn(x)}{ |x|^{1+a}} + \frac{\kappa_1(a)}{ x}+\sum_{m=1}^M(-1)^m\frac{2B_{2m}}{(2m)!}\zeta(1-2m-a)(2\pi x)^{2m-1}+O_{a,M}(|x|^{2M+1}),
\end{equation}
for any~$M\in\N$, where $B_{2n}\in\Q_{\neq0}$ denotes the $2n$-th Bernoulli number. By the functional equation we have
$$
\frac{\zeta(1-2m-a)}{\zeta(1-2(m+1)-a)} =-\frac{2\pi}{(2m+a)(2m+a+1)}\frac{ \zeta(2m+a)}{ \zeta(2m+2+a)}=\frac{2\pi}{4m^2}\bigg(2-\frac{1+2a}{m}+O(m^{-2})\bigg)
$$
as $m\to\infty$. In particular,  $\zeta(1-2m-a)$ cannot be a real multiple of $\zeta(1-2(m+1)-a)$ for $m$ large enough and thus at least one of these terms in the expansion~\eqref{ash3} survives once composed with $\phi$. Letting $n\to\infty$, we see that this is not compatible with~\eqref{tcas} and thus we reach a contradiction.  This finishes the proof of Corollary~\ref{cordedek} when~$\Re(a)<-1$.

\subsubsection{The case~$\Re(a)>-1$ and $a\zeta(a)\neq0$}

It is remarked in~\cite[p.~227]{Bettin2013a} that~$c_a(x) \equiv 0$ whenever~$a$ is a positive odd integer. In this case Corollary~\ref{cordedek} holds for trivial reasons and thus we assume that~$a\not\in 2\Z + 1$ throughout the section. We also assume $\Re(a)>-1$ and $a\zeta(a)\neq0$  and notice that by the functional equation we have $\kappa_1(a)\neq0$.

Since~$ h_a(x) = O(\abs{x}^{-\max(1, 1+\Re(a))})$ for~$\abs{x}<1$, the hypotheses of Theorem~\ref{dless1} are satisfied and it follows that the limit
\begin{align}\label{dfct}
  \extpos{\ct_a}(x) := \lim_{j\to\infty} \den(x_{2j+1})^{-1-a} \ct_a(\bar{x_{2j+1}}), \qquad (x_{2j+1}=[b_0(x); b_1(x), \dotsc, b_{2j+1}(x)]) 
\end{align}
exists for~$x$ in a set~$X\subset \R\setminus\Q$ of full measure. By Theorem~\ref{dless1}, we have that the multiset
$$ \{q^{-1-a} \ct_a(\tfrac bq) \colon 0<b<q, (b, q) = 1\}= \{q^{-1-a} \ct_a(\tfrac {\overline b}q) \colon 0<b<q, (b, q) = 1\} $$
becomes distributed, as~$q\to\infty$, according to~$(\extpos{\ct_a})_*(\df \nu)$. But by definition of~$\ct_a$, we have $q^{-1-a} \ct_a(\frac {\overline b}q)=q^{-1-a} c_a(\frac {\overline b}q)+a\kappa_1(a)\{\frac bq\}$ for $q>1$. Letting
\begin{align}\label{rlct}
  \extpos{c_a}(x) := \extpos{\ct_a}(x) - a\kappa_1(a) \{x\},
\end{align} 
it follows that the multiset
$$ \{q^{-1-a} c_a(\tfrac bq) \colon  0<b<q, (b, q) = 1\} $$
becomes distributed according to~$(\extpos{c_a})_*(\df \nu)$, as claimed.

\medskip

We now turn to showing that the relevant measures are diffuse. 

It is convenient to make a further simple modification to $c_a$ and define
$ \ctt_a(x) := \ct_a(x)-\kappa_2(a)\sgn(x)$.
Clearly, one has that $\ctt_a(x)$ still satisfies~\eqref{eq:f-weak-per}, whereas~\eqref{eq:perrel-cta} holds with $h_a(x)$ replaced by $\breve h_a(x)=h_a(x)-\kappa_1(a)(1+ \abs{x}^{-1-a})$. In particular,~\eqref{ash} becomes
\begin{equation}\label{ash2}
  \breve h_a(x) =- \kappa_2(a)\sgn(x)  + \kappa_1(a) x^{-1}+O_a(x),\qquad x\to0.
\end{equation}
The hypothesis of Theorem~\ref{dless1} are thus still satisfied. Taking the limit as in~\eqref{dfct} one deduces that $\extpos{\ctt_a}(x)=\extpos{\ct_a}(x)$ for $x\in X$.
Also, by~\eqref{ash2} for $\eps<|\log u|<\eps^{2/3}$ we have
\begin{align}\label{fdd}
  \frac{\breve h_a(x)-u^{1+a} \breve h_a(ux)}{|\breve h_a(x)|}
    &=  \frac{\kappa_2(a) (1-u^{1+a})+\kappa_1(a)(1-  u^{a})x^{-1}+o(1)}{\kappa_2(a)+\kappa_1(a)x^{-1}+o(1)}=1-u^a+o(1)
\end{align}
as $x\to0^+$, uniformly in sufficiently small $\eps>0$. Notice that $  1-u^{a}=-a\log u+O_a(\eps^{4/3})\gg\eps$.
When~$a\in\R$, then clearly~$c_a(x) \in\R$ for all~$x\in\Q$, and therefore~$h_a(x) \in\R$ for~$x\neq 0$ as well. By~\eqref{eq:hypo-2star-real} one can then deduce that~$(\extpos{\ctt_a})_*(\df \nu)$ is diffuse,
and thus so is $(\extpos{\ct_a})_*(\df \nu)$. If  $a\notin\R$ one can show  in the same way that $(\phi(\extpos{\ct_a}))_*(\df\nu)$ is diffuse for any non-zero linear form~$\phi:\C\to\R$, upon choosing 
$\pphi:=-\arg (-a)-\kappa$ in~\eqref{eq:hypo-2star-complex}, where $\kappa$ is such that $\phi(z)=\Re(e^{i\kappa} z)$ for all $z\in\C$.

\medskip

We now wish to prove that also the measure $(\phi_a(\extpos{c_a}))_*(\df\nu)$ is diffuse (with $\phi=\Re$ if $a\in\R$). Notice that we can assume~$\phi(a\kappa_1(a))\neq 0$ since otherwise~$\phi\circ\extpos{c_a} = \phi\circ\extpos{\ct_a}$.

We start with the case $\Re(a)>1$. Suppose by contradiction that $\phi\circ\extpos{\ct_a}(x)=\lambda$ for all $x\in C\seq (0,1)$ for a set $C$ of  positive measure. Since $\Re(a)>1$, by Theorem~\ref{ds1} we can assume that $\ct_a$ is $\alpha$-Hölder continuous at any point of $C$ for any $\alpha\in(1,\tfrac12(1+\Re(a)))$. Also, since $C$ is uncountable, it contains one of its accumulation points, i.e. there exists a sequence $(z_m)_m$ in $C$ converging to $z\in C$. Then, on the one hand we have
$$
\phi(\ct_a(z))-\phi(\ct_a(z_m))=\phi(\ct_a(z)-\ct_a(z_m))=O(|z-z_m|^{\alpha})=o(|z-z_m|)
$$
as~$m\to\infty$, and on the other hand by~\eqref{rlct} we have
$$\phi(\extpos{\ct_a}(z)) - \phi(\extpos{\ct_a}(z_m)) = \phi(\extpos{c_a}(z)) - \phi(\extpos{c_a}(z_m))-\phi(a\kappa_1(a))(z-z_m)=-\phi(a\kappa_1(a))(z-z_m).$$
Since these equations are not compatible we reach the desired contradiction.

\smallskip

Now, assume $\Re(a)\in(-1,1]$. First we notice that in this case we can proceed directly with $\hat c_a(x)=c_a(x)-\kappa_2(a)\sgn(x)$. Indeed, if we let~$\hat h_a(x):=\breve h_a(x) -a\kappa_1(a)\frac{\den(x)^a}{\num(x)}\1_\Q(x)$, where $\1_\Q$ is the indicator function of the rationals, we have that $\hat h_a(x)$ still satisfies the hypothesis of Theorem~\ref{dless1}. 
Clearly the function $\extpos{c_a}(x)$ obtained this way coincides with the $\extpos{c_a}(x)$ defined above for almost all $x$. 
Also, if $\Re(a)<1$ we have that~$\hat h_a(x)\sim \breve h_a(x) $ as $x\to0$, and thus we can show as in~\eqref{fdd} that $\frac{\hat h(x)-u^{1+a} \hat h(ux)}{|\hat h(x)|}=1-u^a+o(1)$ as $x\to0$. The same argument then gives that $(\phi_a(\extpos{c_a}))_*(\df\nu)$ is diffuse.

Finally, we assume $\Re(a)=1$ (we recall $a\neq1$). We apply Lemma~\ref{slt} in the form of Remark~\ref{rationals}. Using the notation of the remark, we have $\hat h_a(x)=a\kappa_1(a)\frac{\den(x)-\den(x)^a}{\num(x)}+O(1)\asymp 1/x \to \infty$ with $x\in \mathcal R$, for $\beta\notin\Z$ and $\eps$ sufficiently small. Then, for $x,ux \in \mathcal R$, $x=p/q$, $\num(u x)=p$, $\eps<|\log u|<\eps^{2/3}$, as $x\to 0^+$ we have
\begin{align*}
  \frac{\hat h_a(x)-u^{1+a} \hat h_a(ux)}{|\hat h_a(x)|}
  &= a\kappa_1(a)  \frac{(1-  u^{a})q/p-(1-u)q^a/p}{|\breve h_a(x)|}+o(1)\\
  &= a\kappa_1(a)  \frac{(q^{a-1}-a)\log u+O(\eps^{4/3})}{|\breve h_a(x)|x}+o(1)\\
  &= a\kappa_1(a)  \frac{(\e(\beta)-a)\log u+O(\eps^{4/3})}{|\breve h_a(x)|x}+o(1),
\end{align*}
from which we can once again conclude that $(\phi_a(\extpos{c_a}))_*(\df\nu)$ is diffuse by choising $\beta$ appropriately.

\subsubsection{The case of $\zeta(a)=0$}\label{za0}

We assume $\zeta(a)=0$, so that $\kappa_1(a)=0$, $\Re(a)\in(0,1)$ and $\Im(a)\neq0$. In particular, with the same notation as in the previous section, we have that $\hat h_a(x)=-\kappa_2(a)\sgn(x)+o(1)$ as $x\to0$. Thus $\hat h_a(x)$ is continuous on $[-1,1]\setminus\{0\}$ with non-zero right and left limits $\hat h_a(0^\pm)=\mp\kappa_2(a)$ at $x=0$.\footnote{In fact one has that $\hat h_a(x)+\kappa_2(a)\sgn(x)$ extends to a function in $\mathcal C^{\infty}(\R)$.} Since $\kappa_2(a)\neq0$ and the weight $1+a\notin\R$, we immediately deduce from the condition~\ref{it:cond-kpos-1} in Theorem~\ref{gdless1} that $(\phi(\extpos{c_a}))_*(\df\nu)$ is diffuse for any non-zero linear form~$\phi$.

\subsubsection{The case of $a=0$}

We let~$\hat h_0(x):= h_0(x) +\frac1{\pi \num(x)}\1_\Q(x)$.
Using~\eqref{ash0} instead of~\eqref{ash} we obtain
$$\frac{\hat h_0(x)-u \hat h_0(ux)}{|\hat h_0(x)|}=-\frac{u\log u+O(|x|)}{\log(2\pi |x|)+\gamma}$$
which is not quite sufficient for hypothesis~\ref{eq:hypo-2star-real} to hold. However, we observe that for $h=h_0$ (and $\delta=1/2$, $C=4$, $R=2^{1/4}$, and $\xi(w):=w^{-1/2}$) the conclusion of Lemma~\ref{slt} holds with $\psi(x)=x^{-1/5}$. In view of this, and following Remark~\ref{nlim}, we obtain
$$
\frac{\hat h_0(x)-u \hat h_0(ux)}{|\hat h_0(x)|}+O(|x|^{-1/5})=-\frac{u\log u+O(|x|)}{\log(2\pi |x|)+\gamma}+O(|x|^{-1/5})
$$
which is clearly non-zero for $x>0$ small enough and $\eps<|\log u|<\eps^{2/3}$. Substituting this estimate inside~\eqref{reree} and completing the arguments in Section~\ref{sec:cont-cdf-hunbounded}, we conclude that~$(\extpos{c_0})_*(\df \nu)$ is diffuse.

\subsection{Eichler integrals of Maass cusp forms}

\begin{proof}[Proof of Corollary~\ref{cormaf}]
  For simplicity, we assume $u$ is even or odd and let $j\in\{0,1\}$ be such that $u(-\overline z)=(-1)^ju(z)$. The case where $u$ is neither is believed not to be possible, but it could as well be handled easily by splitting $u$ in its even or odd components.
  In~\cite{Bruggeman2007}, it is shown that the map~$\tilde u$ defined in~\eqref{eq:def-utilde} is a quantum modular form of weight~$2s$, whose associated period function~$h$ is defined by\footnote{This corrects a typo in~\cite[Section~1.1.3]{Bruggeman2007}, as the equation for $\tilde \psi$ should be $-|x|^{-2s}\tilde \psi(-1/x)=\tilde \psi(x)$.}
  $$ h(x) = c(s)\sgn(x)^{1-j}\psi(|x|), \qquad x\in\R_{\neq0}  $$
  and by continuity at $x=0$,
  where~$c(s) = i \pi^{-s} / \Gamma(1-s)$ is a non-zero proportionality constant~\cite[eq.~(1.12)]{Lewis2001}.
  The map~$h$ is real-analytic separately on~$\R_{\leq 0}$ and~$\R_{\geq 0}$ and is $\mathcal C^{\infty}$ on $\R$.
  Thus Theorem~\ref{ds1} applies and yields the first part of Corollary~\ref{cormaf}.
  The expansion of $\psi$ at $0$ is implicit in~\cite{Lewis2001} and can be immediately deduced by shifting the line of integration to the left in the first display of~\cite[p.~205]{Lewis2001} from which one deduces that $h(x)$ is not identically zero on $[-1,1]$. Thus, since the weight is $2s\notin\R$, Theorem~\ref{gdless1} applies in the case~\ref{it:cond-kpos-1}, and we deduce that~$(\phi\circ \extpos{\tilde u})_\ast(\df\nu)$ is diffuse. This concludes the proof of Corollary~\ref{cormaf}.
\end{proof}

\begin{remark}\label{rmk:regularity-eichler}
  By the functional equation for twists of Maa\ss{} forms $L$-functions~\cite[(A.12)-(A.13)]{Kowalski2002}, we have
  $$ \tilde u(\bar{a}/q) = q^{2s} \sum_{n\geq 1} \frac{a_n}{n^{1/2+s}}(c_1 \cos(2\pi n a/q) + c_2 \sin(2\pi n a/q)), \qquad (a\bar a\equiv 1\pmod{q}), $$
  where~$c_1, c_2$ are numbers depending on~$u$.
  With this formulation, the regularity of~$a/q \mapsto q^{-2s} \tilde u(\bar{a}/q)$ relates to the differentiability properties for~$\tilde g$ which we mentioned in Section~\ref{sec:appl-perpol-holo}.

  An important classical problem is the analogous question when~$a_n = d(n)$ is the divisor function, see~\cite{Riemann2013,Chowla1931,Wintner1937}. This can be seen as the case when~$u$ is replaced by a certain weight~$0$ real-analytic, non-cuspidal Eisenstein series~\cite[p.~62]{Iwaniec2002}. Regularity in this case can be studied from the expression above as a Fourier series; recent general results can be found in~\cite{ChamizoEtAl2017}.
\end{remark}

\bibliographystyle{../../amsalpha2}
\bibliography{../../bib2}

\providecommand{\bysame}{\leavevmode\hbox to3em{\hrulefill}\thinspace}
\providecommand{\MR}{\relax\ifhmode\unskip\space\fi MR }
\providecommand{\MRhref}[2]{%
  \href{http://www.ams.org/mathscinet-getitem?mr=#1}{#2}
}
\providecommand{\href}[2]{#2}
\begin{thebibliography}{CPRC17}

\bibitem[BV05]{BaladiVallee2005}
V.~Baladi and B.~Vall\'ee, \emph{Euclidean algorithms are {Gaussian}}, J.
  Number Theory \textbf{110} (2005), no.~2, 331--386.

\bibitem[BM19]{BalazardMartin2019}
M.~Balazard and B.~Martin, \emph{On certain approximate functional equations
  that are related to the {Gauss} transformation}, Aequationes Math.
  \textbf{93} (2019), no.~3, 563--585.

\bibitem[Ben15]{Bengoechea2015}
P.~Bengoechea, \emph{From quadratic polynomials and continued fractions to
  modular forms}, J. Number Theory \textbf{147} (2015), 24--43.

\bibitem[Bet15]{Bettin2015}
S.~Bettin, \emph{On the distribution of a cotangent sum}, Int. Math. Res. Not.
  IMRN (2015), no.~21, 11419--11432.

\bibitem[BC13a]{Bettin2013a}
S.~Bettin and J.~B. Conrey, \emph{Period functions and cotangent sums}, Algebra
  Number Theory \textbf{7} (2013), no.~1, 215--242.

\bibitem[BC13b]{BettinConrey2013}
\bysame, \emph{A reciprocity formula for a cotangent sum}, Int. Math. Res. Not.
  IMRN (2013), no.~24, 5709--5726.

\bibitem[BD]{BettinDrappeau}
S.~Bettin and S.~Drappeau, \emph{Limit laws for rational continued fractions
  and value distribution of quantum modular forms}, Preprint.

\bibitem[BFR15]{BringmannEtAl2015}
K.~Bringmann, A.~Folsom, and R.~C. Rhoades, \emph{Unimodal sequences and
  ``strange'' functions: a family of quantum modular forms}, Pacific J. Math.
  \textbf{274} (2015), no.~1, 1--25.

\bibitem[BR16]{BringmannRolen2016}
K.~Bringmann and L.~Rolen, \emph{Half-integral weight {E}ichler integrals and
  quantum modular forms}, J. Number Theory \textbf{161} (2016), 240--254.

\bibitem[Bru07]{Bruggeman2007}
R.~Bruggeman, \emph{Quantum {M}aass forms}, The {C}onference on
  {$L$}-{F}unctions, World Sci. Publ., Hackensack, NJ, 2007, pp.~1--15.

\bibitem[BLZ15]{BruggemanEtAl2015}
R.~Bruggeman, J.~Lewis, and D.~Zagier, \emph{Period functions for {M}aass wave
  forms and cohomology}, Mem. Amer. Math. Soc. \textbf{237} (2015), no.~1118,
  v+128.

\bibitem[CPRC17]{ChamizoEtAl2017}
F.~Chamizo, I.~Petrykiewicz, and S.~Ruiz-Cabello, \emph{The {H{\"o}lder}
  exponent of some {Fourier} series}, J. Fourier Anal. Appl. \textbf{23}
  (2017), no.~4, 758--777.

\bibitem[Cho31]{Chowla1931}
S.~D. Chowla, \emph{Some problems of diophantine approximation. {I}.}, Math. Z.
  \textbf{33} (1931), 544--563.

\bibitem[Eic57]{Eichler1957}
M.~Eichler, \emph{Eine {Verallgemeinerung} der {Abelschen} {Integrale}}, Math
  Z. \textbf{67} (1957), 267--298.

\bibitem[FT91]{Fouvry1991}
E.~Fouvry and G.~Tenenbaum, \emph{Entiers sans grand facteur premier en
  progressions arithmétiques}, Proc. London Math. Soc. \textbf{3} (1991),
  no.~3, 449--494.

\bibitem[GO21]{Goswami2021}
A.~Goswami and R.~Osburn, \emph{Quantum modularity of partial theta series with
  periodic coefficients}, Forum Math. \textbf{33} (2021), no.~2, 451--463.

\bibitem[Hic77]{Hickerson1977}
D.~Hickerson, \emph{Continued fractions and density results for {D}edekind
  sums}, J. Reine Angew. Math. \textbf{290} (1977), 113--116.

\bibitem[Iwa97]{Iwaniec1997}
H.~Iwaniec, \emph{Topics in classical automorphic forms}, Graduate Studies in
  Mathematics, vol.~17, American Mathematical Society, Providence, RI, 1997.

\bibitem[Iwa02]{Iwaniec2002}
\bysame, \emph{Spectral methods of automorphic forms}, second ed., Graduate
  Studies in Mathematics, vol.~53, American Mathematical Society, Providence,
  RI; Revista Matem\'{a}tica Iberoamericana, Madrid, 2002.

\bibitem[JM18]{JaffardMartin2018}
S.~Jaffard and B.~Martin, \emph{Multifractal analysis of the {B}rjuno
  function}, Invent. Math. \textbf{212} (2018), no.~1, 109--132.

\bibitem[Khi63]{Khintchine1963}
A.~Ya. Khintchine, \emph{Continued fractions}, Translated by Peter Wynn, P.
  Noordhoff, Ltd., Groningen, 1963.

\bibitem[KLL16]{KimLimEtAl2016}
B.~Kim, S.~Lim, and J.~Lovejoy, \emph{Odd-balanced unimodal sequences and
  related functions: parity, mock modularity and quantum modularity}, Proc.
  Amer. Math. Soc. \textbf{144} (2016), no.~9, 3687--3700.

\bibitem[KMV02]{Kowalski2002}
E.~Kowalski, P.~Michel, and J.~VanderKam, \emph{Rankin-{Selberg}
  {$L$}-functions in the level aspect}, Duke Math. J. \textbf{114} (2002),
  no.~1, 123--191.

\bibitem[LZ01]{Lewis2001}
J.~Lewis and D.~Zagier, \emph{Period functions for {M}aass wave forms. {I}},
  Ann. of Math. (2) \textbf{153} (2001), no.~1, 191--258.

\bibitem[LZ19]{LewisZagier2019}
\bysame, \emph{Cotangent sums, quantum modular forms, and the generalized
  {R}iemann hypothesis}, Res. Math. Sci. \textbf{6} (2019), no.~1, Paper No. 4,
  24.

\bibitem[MR16]{Maier2016}
H.~Maier and M.~Th. Rassias, \emph{Generalizations of a cotangent sum
  associated to the {E}stermann zeta function}, Commun. Contemp. Math.
  \textbf{18} (2016), no.~1, 1550078, 89.

\bibitem[NR17]{NgoRhoades2017}
H.~T. Ngo and R.~C. Rhoades, \emph{Integer partitions, probabilities and
  quantum modular forms}, Res. Math. Sci. \textbf{4} (2017), Paper No. 17, 36.

\bibitem[RG72]{Rademacher1972}
H.~Rademacher and E.~Grosswald, \emph{Dedekind sums}, The Mathematical
  Association of America, Washington, D.C., 1972, The Carus Mathematical
  Monographs, No. 16.

\bibitem[Rie92]{Riemann1892}
B.~Riemann, \emph{Gesammelte mathematische {Werke} und wissenschaftlicher
  {Nachlass}. {Herausgegeben} unter {Mitwirkung} von {Richard} {Dedekind} von
  {Heinrich} {Weber}. {Zweite} {Auflage} bearbeitet von {Heinrich} {Weber}.},
  Leipzig. {B}. {G}. {Teubner}. {X} u. 558 {S}. gr. {{\(8^{\circ}\)}} (1892).,
  1892.

\bibitem[Rie13]{Riemann2013}
\bysame, \emph{Ueber die {Darstellbarkeit} einer {Function} durch eine
  trigonometrische {Reihe}}, Cambridge Library Collection - Mathematics,
  Cambridge University Press, 2013.

\bibitem[Ruk06]{Rukavishnikova2006}
M.~G. Rukavishnikova, \emph{A probability estimate for the sum of incomplete
  partial quotients with fixed denominator}, Chebyshevski\u{i} Sb. \textbf{7}
  (2006), no.~4(20), 113--121.

\bibitem[Var93]{Vardi1993}
I.~Vardi, \emph{Dedekind sums have a limiting distribution}, Int. Math. Res.
  Not. IMRN (1993), no.~1, 1--12.

\bibitem[Win37]{Wintner1937}
A.~Wintner, \emph{On a trigonometrical series of {Riemann}.}, Amer. J. Math.
  \textbf{59} (1937), 629--634.

\bibitem[Zag99]{Zagier1999}
D.~Zagier, \emph{From quadratic functions to modular functions}, Number theory
  in progress, {V}ol. 2 ({Z}akopane-{K}o\'{s}cielisko, 1997), de Gruyter,
  Berlin, 1999, pp.~1147--1178.

\bibitem[Zag01]{Zagier2001}
\bysame, \emph{Vassiliev invariants and a strange identity related to the
  {D}edekind eta-function}, Topology \textbf{40} (2001), no.~5, 945--960.

\bibitem[Zag10]{Zagier2010}
\bysame, \emph{Quantum modular forms}, Quanta of maths, Clay {Math}. {Proc}.,
  vol.~11, Amer. Math. Soc., Providence, RI, 2010, pp.~659--675.

\end{thebibliography}

\end{document}